\documentclass[10pt,a4paper]{article}
\usepackage{amsopn}

\usepackage{graphics}
\usepackage{graphicx}
\usepackage{epsfig}

\usepackage{amssymb}
\usepackage{amsthm}
\usepackage{amsmath}
\usepackage{amsfonts}
\usepackage{lipsum}
\usepackage[left=2.5cm,right=2.5cm,bottom=3.5cm,top=3.5cm]{geometry}
\usepackage{empheq}
\usepackage{epstopdf}
\usepackage{changepage}
\usepackage{rotating}
\usepackage{adjustbox}
\usepackage{graphbox}
\usepackage{color}
\usepackage{xcolor}
\usepackage{dsfont}
\usepackage{dutchcal}
\usepackage{hyperref} 
\usepackage{float}
\usepackage{resizegather}
\usepackage{multicol}
\usepackage{booktabs} 
\usepackage{colortbl} 
\usepackage{multirow} 
\usepackage{longtable,tabu} 
\usepackage{hhline}   
\usepackage{anyfontsize} 
\usepackage{everysel}
\usepackage{psfrag}
\usepackage{titlesec}
\usepackage{bbm}
\usepackage{bm}
\usepackage{pgf,tikz}

\DeclareMathOperator{\dive}{\nabla \cdot}
\DeclareMathOperator{\gra}{\nabla}
\DeclareMathOperator{\grae}{\nabla}

\DeclareMathOperator{\dS}{\mathrm{dS}}
\DeclareMathOperator{\dV}{\mathrm{dV}}

\newcommand{\g}{\mathbf{g}}
\newcommand{\x}{\mathbf{x}}

\newcommand{\Flux}{\mathbf{\mathcal{F}}}

\newcommand{\btau}{\boldsymbol{\tau}}
\newcommand{\tbtau}{\boldsymbol{\tau}^{\ast}}

\newcommand{\press}{p}
\newcommand{\Press}{P}

\newcommand{\halb}{\frac{1}{2}}
\newcommand{\vel}{\mathbf{u}}

\newcommand{\W}{\mathbf{W}}
\newcommand{\hW}{\hat{\mathbf{W}}}
\newcommand{\Wvel}{\W}

\newcommand{\tW}{{\W}^{\ast}}
\newcommand{\htW}{\hat{\W}^{\ast}}
\newcommand{\ttW}{{\W}^{\ast\ast}}
\newcommand{\httW}{\hat{\W}^{\ast\ast}}
\newcommand{\tWa}[2]{{\W}^{\ast\,{#1}}_{#2}}
\newcommand{\tvel}{\vel^{\ast}}
\newcommand{\errres}{\mathbf{f}\left(\tW\right)}

\newcommand{\mass}{\mathbf{M}}
\newcommand{\stiffness}{\mathbf{K}}

\newcommand{\Dt}{\Delta t\,}

\newcommand{\nun}{\boldsymbol{n}}
\newcommand{\nw}{\boldsymbol{\eta}}

\newcommand{\Fpm}{\mathcal{C}}

\newtheorem{theorem}{Theorem}

\theoremstyle{remark}

\hypersetup{
	pdftitle={An implicit staggered hybrid finite volume/finite element solver for the incompressible Navier-Stokes equations},    
	pdfauthor={A. Lucca, S. Busto, M. Dumbser},     
	pdfcreator={A. Lucca, S. Busto, M. Dumbser},   
	colorlinks=true,       
}

\title{An implicit staggered hybrid finite volume/finite element solver for the incompressible Navier-Stokes equations} 

\allowdisplaybreaks

\begin{document}

\begin{center}
	\textbf{ \Large{An implicit staggered hybrid finite volume/finite element solver for the incompressible Navier-Stokes equations} }
	
	\vspace{0.5cm}
	{A. Lucca\footnote{alessia.lucca@unitn.it}, S. Busto\footnote{saray.busto@unitn.it}, M. Dumbser\footnote{michael.dumbser@unitn.it}}
	
	\vspace{0.2cm}
	{\small		
		\textit{$^{(1)}$ Department of Mathematics, University of Trento, Via Sommarive 14, 38123 Trento, Italy}
		
		\textit{$^{(2)}$ Departamento de Matem\'atica Aplicada a la Ingenier\'ia Industrial, Universidad Polit\'ecnica de Madrid, Jos\'e Gutierrez Abascal~2, 28006, Madrid, Spain}
		
		\textit{$^{(3)}$ Laboratory of Applied Mathematics, DICAM, University of Trento, via Mesiano 77, 38123 Trento, Italy}
	}
\end{center}

\hrule

\thispagestyle{plain}
\begin{center}
	\textbf{Abstract}
\end{center} 

We present a novel fully implicit hybrid finite volume/finite element method for incompressible flows. Following previous works on semi-implicit hybrid FV/FE schemes, the incompressible Navier-Stokes equations are split into a pressure and a transport-diffusion subsystem. The first of them can be seen as a Poisson type problem and is thus solved efficiently using classical continuous Lagrange finite elements. On the other hand, finite volume methods are employed to solve the convective subsystem, in combination with Crouzeix-Raviart finite elements for the discretization of the viscous stress tensor. For some applications, the related CFL condition, even if depending only in the bulk velocity, may yield a severe time restriction in case explicit schemes are used. To overcome this issue an implicit approach is proposed. The system obtained from the implicit discretization of the transport-diffusion operator is solved using an inexact Newton-Krylov method, based either on the BiCStab or the GMRES algorithm. To improve the convergence properties of the linear solver a symmetric Gauss-Seidel (SGS) preconditioner is employed, together with a simple but efficient approach for the reordering of the grid elements that is compatible with MPI parallelization. 
Besides, considering the Ducros flux for the nonlinear convective terms we can prove that the discrete advection scheme is kinetic energy stable.  
The methodology is carefully assessed through a set of classical benchmarks for fluid mechanics. A last test shows the potential applicability of the method in the context of blood flow simulation in realistic vessel geometries.

\vspace{0.2cm}
\vspace{0.4cm}

\hrule

\section{Introduction}\label{sec:intro}

The development of numerical methods for the simulation of incompressible flows is a wide field of research  since they allow the solution of many industrial, environmental and biological problems. 
One of those applications is the simulation of blood flow in the human cardiovascular system.  
Nowadays the study and treatment of diverse pathologies may require the use of invasive techniques, which might be a risk for the life of the patient. A common example of this situation is the presence of stenosis in a vessel. An abnormal narrowing in an artery may lead to many syndromes, lowering the life quality of the patient and even causing its death. To analyse in detail the impact of the stenosis without using invasive techniques an alternative may be the use of simulation tools, see e.g. \cite{TDKPSBWZ99,FTM17,fossan2018,lucca2023}. The physical geometry can be obtained via medical images which allow the definition of a computational domain to be used in numerical simulations, \cite{Steinman2002}. Therefore, having a methodology able to efficiently solve this kind of flows would constitute an important step forward in personalized medicine that might help the medics in their decisions they need to take concerning the optimal treatment of each patient. For a non-exhaustive overview of some computational methods that have already been successfully applied to the simulation of the human cardiovascular system, the reader is referred to   \cite{Quarteroni1,Quarteroni2,Quarteroni3,Quarteroni4,Formaggia:2009,Fung:1997,Quarteroni:2009,Quarteroni:2004,Sherwin:2003,Sherwin:2003b,Mueller2013,Mueller2014,Mueller2015} and references therein.

From the numerical point of view, some of the most widespread methodologies to simulate incompressible flows fall in the framework of pressure based semi-implicit solvers, see e.g. \cite{HW65,Cho67,PS72,Pat80,BCG89,Cas14,Guer06}. They rely on the computation of the pressure by deriving a Poisson-type equation from the mass and momentum conservation equations while an approximation of velocity is obtained from a transport-diffusion subsystem that is then updated with the contribution of the new pressure. Let us remark that performing an adequate splitting of the equations allows the decoupling of the bulk flow velocity from and the fast sound waves, \cite{TV12}. Then an unconditionally stable implicit algorithm can be used to solve the elliptic, linear and symmetric positive definite pressure subsystem, while the non-symmetric and nonlinear transport-diffusion subsystem, if solved explicitly, will be characterized by a CFL number that depends only on the bulk velocity. Within this framework, many families of methods have been developed depending on the approach selected to solve those subsystems including continuous finite elements (FE), \cite{BrezzFor,TaylorHood,BSS12,zien3,HR88}, discontinuous Galerkin schemes (DG), \cite{TD16,AMRDGSI,DFFMST18,TB19}, or finite volume methods (FV), \cite{HW65,RNKCC14,Toro}, among others. 
Moreover, extensions of pressure-based methods showing the potential of this approach to solve also all Mach number flows have been proposed in the last decades see, for instance, \cite{PM05,DC16,TD17,DLDV18,AbateIolloPuppo,BP2021,LPT22,BQRX22} and references therein.

Despite the wide variety of methods available inside each one of the aforementioned families, their combination within one algorithm is less common. In the framework of hybrid FV/FE schemes, a new family of methods combining FV and FE methods has been built for different mathematical models, including Newtonian and non Newtonian incompressible flows \cite{BFSV14,BFTVC17,HybridNNT}, weakly compressible flows \cite{Hybrid1,HybridALE}, all Mach number flows \cite{Hybrid2,HybridMPI} and the shallow water equations \cite{HybridSW}. The main idea behind this methodology is to employ a semi-implicit scheme on staggered unstructured meshes, where the pressure subsystem is discretized implicitly according to a classical finite element method, while the transport-diffusion subsystem is solved using an explicit finite volume scheme. 
Although the time step restriction imposed by the CFL condition, is less restrictive than for fully explicit schemes, the main concern of this methodology is that it still results too demanding when addressing haemodynamics in complex 3D geometries. Therefore, in order to make our methodology suitable for this  kind of applications in this paper we propose its extension to a fully implicit scheme in which Crouzeix-Raviart basis functions are used in combination with an implicit finite volume approach to discretize the convective-viscous system.

The use of an implicit scheme to solve the transport-diffusion equations leads to a nonlinear and nonsymmetric system that must be solved in each time step. To circumvent the direct solution of such systems a classical approach is the use of a Newton-Krylov method, with iterative solvers that make use of consecutive matrix-vector products and which allow a matrix-free implementation. In particular, we will employ a Newton algorithm combined with one of the following Krylov subspace methods: either a generalized minimal residual approach (GMRES), \cite{SS86}, or a stabilized biconjugate gradient algorithm (BiCGStab), \cite{Vorst92}. 

As it is well known, an important aspect when employing these iterative algorithms is the definition of an efficient preconditioner that would improve their convergence behaviour, \cite{ESW14,CGS02}. For instance, we can make use of the LU factorization symmetric Gauss-Seidel method (LU-SGS), \cite{YJ88,LBL98}, or one of its modified versions, like the one proposed by Menshov and collaborators, \cite{MN00}. When selecting the preconditioning technique to be used it is also very important to pay attention to its suitability for parallelization, \cite{CPTU18}, since we would like to design an method that is able to deal with a very large number of degrees of freedom. Note that to gain in efficiency depends also on a proper ordering of the nodes and the distribution of the computational load among the available CPUs is also important, \cite{CGS02}. For more advanced preconditioners, including also a special focus on haemodynamics, see  \cite{DGQ14,KOV15,Hoc20,LYDM20} and references therein.

As mentioned before, we will make use of an iterative linear solver to get the solution of the convective-diffusive subsystem. However, convective terms are clearly nonlinear. Therefore, we need to introduce a linearization allowing us to use the Newton-Krylov approach. To this end, two different options are mainly considered in the bibliography: linearization of the continuous equations and then performing the spatial discretization, \cite{KOV15}, or first performing a space discretization and then linearization only where it is strictly necessary, \cite{MN00}. In the present paper, we will consider \textcolor{black}{the} second approach. Consequently, once a numerical flux function is selected we may need to define the corresponding linearized version to be used inside the Krylov subspace methods. Meanwhile, the objective function in the Newton algorithm can take into account the original nonlinear flux. This duplicity of numerical flux functions could be avoided using, for instance, a semi-implicit Ducros flux function which is already linear in the unknown velocity field. Besides, the use of the Ducros flux in the proposed algorithm has as extra advantage that it results in a kinetic energy stable scheme, as proven in Section~\ref{sec.kineticenergystability}.

Let us note that the development of fully implicit schemes for the incompressible Navier-Stokes equations is a well-established subject and many contributions have already been made, see e.g. \cite{turek1996comparative,LBL98,BASSICPR2006,Bassi2007,NPC11,RCV13,LGZJ21,SEG21} and references therein. Nevertheless, most of those methods discretize the complete system of Navier-Stokes equations together, i.e. without considering a splitting of the momentum and pressure computation, and/or employ a unique family of numerical methods for the spatial discretization, in contrast with the new hybrid method proposed in this paper. The use of a splitting technique jointly with the discretization of the two resulting subsystems using different numerical methods in staggered grids provides a compact stencil that is well suited for an efficient parallelization thanks also to the use of a matrix-free approach for the computation of the pressure and momentum subsystems. 

The rest of the paper is organized as follows. In Section \ref{sec:numdisc1}, we recall the incompressible Navier-Stokes equations and present the new fully implicit staggered hybrid FV/FE method. First, to outline the overall algorithm, we focus on the semi-discretization in time and the splitting of the equations into a convective-diffusive subsystem and a pressure subsystem. Next, we introduce the staggered unstructured grids  used in our scheme. Within the convective-diffusion stage, we use an implicit finite volume methodology for the convective terms, while the viscous terms are discretized at the aid of implicit Crouzeix-Raviart elements. Next, we describe the inexact Newton-Krylov methods that are used to solve the resulting large and sparse nonlinear systems, based on matrix-free SGS-preconditioned BiCGStab or GMRES algorithms. 
Then, the projection stage is discretized using classical continuous Lagrange finite elements and the final  correction step is performed in the post-projection stage. 
Section \ref{sec:numericalresults} presents several classical benchmarks from computational fluid mechanics and the results obtained with the proposed algorithm are validated against either analytical solutions, numerical reference solutions or available experimental data. Moreover, the flow on a real 3D coronary tree is studied as last test case. We compare the computational efficiency of the new fully-implicit hybrid FV/FE method with the one of previous semi-implicit FV/FE schemes. In Section \ref{sec:conclusions}, the paper closes with some concluding remarks and an outlook to future research.

\section{Governing equations and numerical discretization} \label{sec:numdisc1}

\subsection{Governing equations} \label{sec:goveq}
As mathematical model we consider the incompressible Navier-Stokes system for Newtonian fluids which reduces to the momentum equation and the divergence-free condition of the velocity field. Denoting $\mathbf{w}=\rho\vel$ the momentum, $\mathbf{u}$ the velocity vector, $\rho$ the density, and $\press$ the pressure, we have

\begin{subequations}\label{eqn.NavierStokes}
\begin{align}
    \dive \vel &=0, \label{eq:sys1}\\
    \frac{\partial \mathbf{w}}{\partial t}+\dive \mathcal{F(\mathbf{w})}+\grae \press-\dive\btau &= \rho\g,  \label{eq:sys2}
\end{align}
\end{subequations}
where $\mathcal{F}$ is the convective flux tensor defined as  $\Flux(\mathbf{w})=\frac{1}{\rho} \mathbf{w} \otimes \mathbf{w}$, $\btau=\mu(\gra\vel+\gra\vel^t)$ is the viscous part of the Cauchy stress tensor, with $\mu$ the laminar viscosity, and $\g$ denotes the gravity vector. Let us remark that neglecting gravity effects, $\g=\boldsymbol{0}$, and assuming an inviscid fluid, $\mu=0$,  \eqref{eqn.NavierStokes} reduces to the incompressible Euler equations

\begin{subequations}\label{eqn.Euler}
\begin{align}
	\dive \vel &=0, \label{eq:mass_euler}\\
	\frac{\partial \mathbf{w}}{\partial t}+\dive \mathcal{F(\mathbf{w})}+\grae \press &=\boldsymbol{0},  \label{eq:momentum_euler}
\end{align}
\end{subequations}
which admit an extra conservation law for the kinetic energy density of the form

\begin{equation}
	\frac{\partial}{\partial t} \left(\halb \rho  \vel^2 \right) + \dive \left(\vel  \, \halb \rho  \vel^2 \right) 	+ \dive \left( \vel \, \press \right) = 0.	\label{eqn.kinen.euler} 
\end{equation}
This property may be also desirable at the discrete level and the kinetic energy stability of the proposed scheme will be analyzed in Section \ref{sec.kineticenergystability}.

\subsection{Overall algorithm} \label{sec:numdisc}

Following the methodology already employed in \cite{BFSV14,BFTVC17,HybridMPI}, to discretize the Navier-Stokes equations we make use of a projection method which decouples the computation of the momentum and pressure unknowns. Denoting by $\W^{n}$, $\Press^{n}$ the discrete approximation of the momentum and pressure at time $t^{n}$, namely $\mathbf{w}(\x,t^{n})$ and $\press(\x,t^{n})$, and performing a semi-discretization in time of \eqref{eqn.NavierStokes}, while all spatial operators are still kept continuous, and since we assume a constant density $\rho$ we have
\begin{subequations}\label{eqn.semi-discreteNavierStokes}
\begin{align}
	\frac{1}{\Delta t}(\tW-\Wvel^n)+\dive\Flux(\W^*)+\gra \Press^n-\dive\btau^*&=\rho\g, \label{eq:timesys1}\\
	\frac{1}{\Delta t}(\W^{n+1}-\tW)+\gra(\Press^{n+1}-\Press^n)&=0, \label{eq:timesys2}\\
	\dive \Wvel^{n+1}&=0.
	\label{eq:timesys3}
\end{align}
\end{subequations}
In the former system, we have introduced an intermediate approximation of the momentum field,  $\tW$, which accounts for the update of convective and diffusive terms but that does not necessarily verify the divergence free condition yet. Combining \eqref{eq:timesys2} with \eqref{eq:timesys3} yields the well-known pressure-Poisson equation for the pressure correction 
\begin{equation}
	\nabla^2 \left( \Press^{n+1}-\Press^n \right) = \frac{1}{\Delta t} \nabla \cdot \tW. 
	\label{eqn.pressure.poisson} 
\end{equation}

Then, to get the final value of the momentum, $\tW$ must be corrected using the pressure at the new time step as
\begin{equation}
	\W^{n+1}= \tW-\Delta t \gra(\Press^{n+1}-\Press^n).\label{eqn.momentum_correction}
\end{equation}
Choosing the momentum $\W$ on the convective and viscous terms of \eqref{eq:timesys1} to be the one given at time $t^n$, then we would derive an explicit approach for its discretization, as it has been done in \cite{BFSV14,BFTVC17,HybridMPI,HybridNNT}. Contrarily, taking $\W:=\tW$ we have the sought implicit approach where the following two subsystems need to be solved:

\textbf{Transport diffusion subsystem}
\begin{equation}
	\frac{1}{\Delta t}(\tW-\Wvel^n)+\dive\Flux(\tW)+\gra \Press^n-\dive\btau(\tW)=\rho\g. \label{eqn.momentum}
\end{equation}

\textbf{Pressure subsystem}
\begin{align}
	\frac{1}{\Delta t}(\W^{n+1}-\tW)+\gra(\Press^{n+1}-\Press^n)&=0, \label{eqn.pressure1}\\
	\dive\vel^{n+1}&=0. \label{eqn.pressure2}
\end{align}
The algorithm for the discretization of the equations above is subdivided into three main stages where different numerical methods are employed according to the nature of the equations, \cite{BFSV14,BFTVC17,Hybrid1,Hybrid2,HybridMPI,HybridNNT,HybridSW,HybridALE}:
\begin{enumerate}
	\item Transport-diffusion stage. System \eqref{eqn.momentum} is solved implicitly making use of a Newton-Krylov approach. The spatial discretization of the equations is done by combining an implicit finite volume scheme for the computation of the convective terms and the use of Crouzeix-Raviart basis functions for the calculation of viscous terms.
	
	\item Projection stage. The pressure subsystem \eqref{eqn.pressure1}-\eqref{eqn.pressure2} is solved implicitly via classical $\mathbb{P}^{1}$ continuous finite elements directly applied to the pressure-Poisson equation \eqref{eqn.pressure.poisson}.
	
	\item Post-projection stage. Taking into account \eqref{eqn.momentum_correction}, the intermediate momentum, $\tW$, computed at the first stage, is updated using the pressure computed in the projection stage and thus providing the final approximation $\W^{n+1}$. 
\end{enumerate}

Let us note that in \eqref{eqn.semi-discreteNavierStokes}, we have considered a pressure-correction technique so that the pressure subsystem depends not only on the new pressure, but also on the difference of pressures between two subsequent time steps, $\delta\Press:=\Press^{n+1}-\Press^n$. A slightly different scheme could be developed by avoiding the pressure term in the momentum equation:
\begin{subequations}\label{eqn.semi-discreteNavierStokes_npc}
	\begin{align}
		\frac{1}{\Delta t}(\tW-\Wvel^n)+\dive\Flux(\tW)-\dive\btau(\tW) &=\rho\g, \label{eq:timesys1_npc}\\
		\frac{1}{\Delta t}(\W^{n+1}-\tW)+\gra \Press^{n+1}&=0, \label{eq:timesys2_npc}\\
		\dive\W^{n+1}&=0.		\label{eq:timesys3_npc}
	\end{align}
\end{subequations}
This approach can also be derived using an appropriate splitting of the Navier-Stokes equations, \cite{TV12,TD14,HybridALE}.

\subsection{Staggered unstructured meshes}
An important aspect of the proposed methodology is that the spatial discretization will be performed using staggered grids of the face-type, based on a primal mesh made of triangles (2D)/tetrahedra (3D), \cite{BDDV98,VC99,USFORCE2,BFTVC17,BTBD20}. 
Denoting by $T_k$ a primal element, $\{T_k,\, k=1,\cdots,n_{\mathrm{primal}}\}$ being the set of primal elements covering the computational domain, we can build half dual elements $T_{k,\, l}$, with $l$ an index on the number of boundaries of a primal element, by computing the barycentres of the primal cells, $B_k$, and connecting them to the vertices of the edges/faces, $V_{m}$. In case an edge/face is located at the interior of the domain, we glue together the two subelements related to this face so that the resulting quadrilateral/polyhedron will constitute the dual cell $C_i$ with area/volume $|C_i|$. For boundary faces the related dual elements coincide with the corresponding primal subelements. In Figure \ref{fig:primal2dualmesh} a sketch on the staggered mesh construction in 2D is shown.
We can observe that the nodes $N_i$, $i=1,\cdots,n_{\mathrm{dual}}$, of the dual mesh are defined as the barycentres of the edges/faces of the elements of the primal mesh, which correspond to the location of the nodes of the Crouzeix-Raviart basis functions used later in this work. Given a node $N_i$ its set of neighbours, $\left\lbrace N_j,\, j\in\mathcal{K}_i\right\rbrace $, is composed of the remaining dual cells build on the edges/faces of the two primal elements to which it belongs. Moreover, the boundary of $C_i$ is defined as $\partial C_i=\Gamma_i=\bigcup\limits_{j\in\mathcal{K}_i} \Gamma_{ij}$ where $\Gamma_{ij}$ is the interface between cell $C_i$ and cell $C_j$. We denote the outward unit normal vector of $\Gamma_{ij}$ as $\nun_{ij}$ and we define $\nw_{ij}=\nun_{ij}||\nw_{ij}||$, where $||\nw_{ij}||=|\Gamma_{ij}|$ is the length/area of $\Gamma_{ij}$. 
Moreover, the barycenter of the edge/face between two dual cells $C_i$ and $C_i$ is denoted by $N_{ij}$.
\begin{figure}
	\centering
	\includegraphics[width=0.4\linewidth]{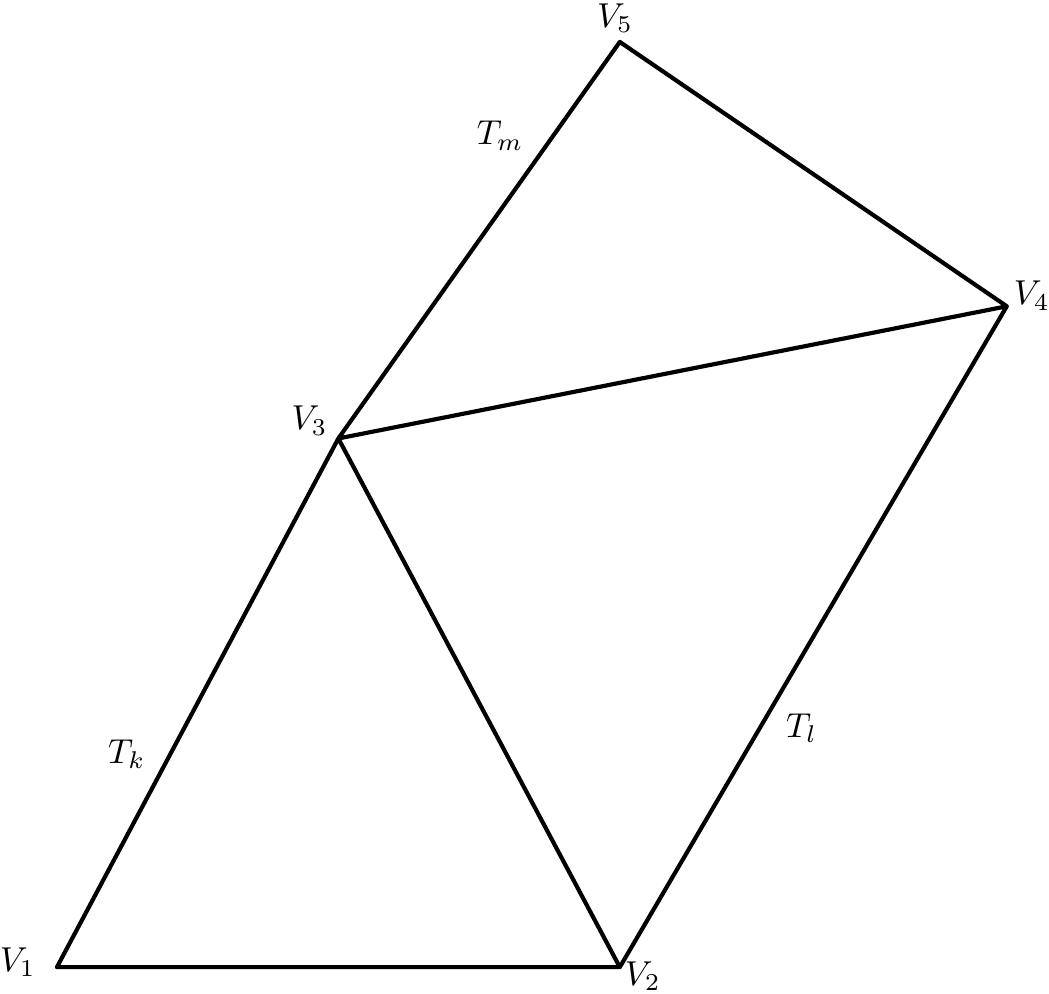}
	\includegraphics[width=0.4\linewidth]{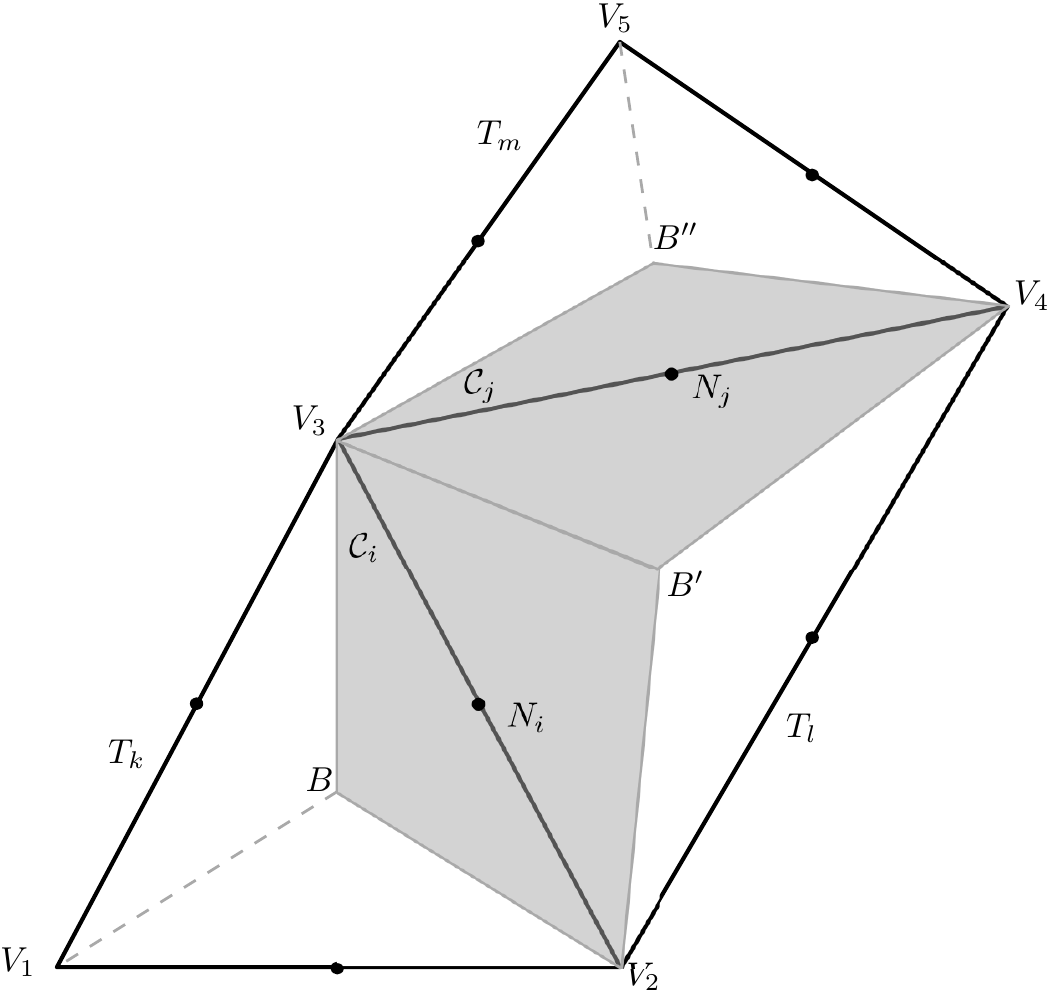}
	\caption{Construction of the staggered grids in 2D. Left: primal triangular grid with elements $T_k$, $T_l$, $T_m$ of vertex $\{V_1,\,V_2,\, V_3\}$, $\{V_2,\,V_3,\, V_4\}$ and $\{V_3,\,V_4,\, V_5\}$ respectively. Right: dual elements $C_i$ and $C_j$ (shadowed in grey) with corresponding nodes $N_i$, $N_j$ and vertex the barycenters of the primal elments $B$, $B'$, $B''$ and the primal vertex $V_2$, $V_3$ and $V_4$. }
	\label{fig:primal2dualmesh}
\end{figure}

As it can be seen in the following sections, this kind of mesh may have several advantages:
\begin{itemize}
	\item Complex geometries can be easily discretized thanks to the use of primal meshes composed of simplex elements, i.e. triangles / tetrahedra. 
	\item An exact interpolation from the primal vertex to the dual cells and back to the primal mesh can be employed for interpolation between the staggered grids, \cite{HybridNNT}.
	\item Even for a second order accurate approach in space, a compact stencil can be obtained with the proper combination of different numerical methods. 
	\item The dual mesh structure together with an appropriate choice for the location of the degrees of freedom results on an easy coupling between the finite volume and the Crouzeix-Raviart approaches as well as on a simple transfer of data with the continuous finite element scheme.
\end{itemize}

\subsection{Implicit transport-diffusion stage}\label{subsec:transdiff} 

Differently from what has been previously done in the context of hybrid FV/FE schemes \cite{Hybrid1,Hybrid2,HybridMPI,HybridALE}, we propose an implicit algorithm for the discretization of the transport-diffusion subsystem \eqref{eqn.momentum}. Accordingly, we get the system
\begin{equation}
	\tW+\Delta t\dive\Flux(\tW)-\Delta t\dive\tbtau = \Wvel^n - \Delta t\gra \Press^n+ \Delta t\rho\g, \label{eq.implicitmomentum}
\end{equation}
with $\tW$ the intermediate momentum, that may not be divergence free yet, the flux tensor $\Flux(\tW)$ and the viscous stress tensor $\btau\left(\tW\right)$, depending on the new $\tW$. Note that the pressure gradient term is also included, but computed explicitly from the pressure obtained at the previous time step, i.e., using a pressure correction approach. Therefore, the momentum will be updated in the post-projection stage to account for the new pressure.

The convective term in \eqref{eq.implicitmomentum} is nonlinear and in general couples all momentum equations with each other, hence employing a direct solver for the approximation of the solution of \eqref{eq.implicitmomentum} at each time step may be too demanding from a computational point of view.
A well known technique to circumvent this problem is employing an inexact Newton algorithm, see \cite{DES82}. This approach allows the use of implicit schemes for nonlinear equations at a reasonable cost by linearizing the discretization of the nonlinear convective term. Hence, to get $\tW:=\Wvel^{n}+\delta \tW$ we employ an inexact Newton algorithm combined with a Krylov subspace method for the solution of linear systems, see also \cite{Bellavia1,Bellavia2,Bellavia3} and references therein.  
The objective is thus to find the root of the vector function
\begin{equation}
	\mathbf{f}\left( \tW \right) = \tW -\W^n + \Delta t\dive\Flux(\tW)-\Delta t\dive\tbtau  + \Delta t\gra \Press^n- \Delta t\rho\g=0. \label{eq.bigfunc}
\end{equation}
which is achieved via the following \textit{inexact} Newton method, using $\tW_0 = \Wvel^n$ as initial guess and the usual Newton iteration with line-search globalization,   
\begin{eqnarray}
	\mathbf{J}\left( \tW_k \right) \, \Delta \tW_k &=& -\mathbf{f}(\tW_k), \label{eqn.newt.linsys} \\ 
	\tW_{k+1} & = & \tW_{k} + \tilde{\delta}_k \, \Delta \tW_k, \label{eqn.newt.update} 
\end{eqnarray}
where $\Delta \tW_k$ is the Newton step, $\mathbf{J}\left( \tW_k \right)  = \partial \mathbf{f}\left( \tW \right) / \partial \tW$ is the Jacobian of the nonlinear function $\mathbf{f}\left( \tW \right)$ in \eqref{eq.bigfunc} and $0 \leq \tilde{\delta}_k \leq 1$ in \eqref{eqn.newt.update} is a suitably chosen scalar  for a simple linear search globalization technique that in each Newton iteration guarantees 
\begin{equation}
\left\| \mathbf{f}\left( \tW_{k+1} \right) \right\| < \left\| \mathbf{f}\left( \tW_{k} \right) \right\|.
\label{eqn.linesearch} 
\end{equation}
Throughout this paper we initially set $\tilde{\delta}_k=1$ and then divide it by two until \eqref{eqn.linesearch} holds. 
In our inexact Newton algorithm the tolerance of the linear solver for \eqref{eqn.newt.linsys} is dynamically set to $\epsilon_k = 10^{-2} \left\| \mathbf{f}\left( \tW_k \right) \right\|$ in the cases where the fully implicit Rusanov flux is used for the discretization of the nonlinear convective terms. When instead the semi-implicit Ducros flux is employed, the resulting system \eqref{eq.bigfunc} is already linear and hence only one single Newton iteration is used. In this case, we set $\epsilon_k = 10^{-10}$.  \textcolor{black}{In order to obtain a more sophisticated initial guess, rather than simply setting  $\tW_0 = \Wvel^n$ as done in this paper, in the future one may also consider the use of a semi-Lagrangian scheme for the advection-diffusion terms, like the ones successfully employed in \cite{Casulli1990,CasulliCheng1992,TBR13,SIDGConv,HOSemiLagDG,BonaventuraSL}. The use of a semi-Lagrangian method could provide a physically more meaningful value for  $\tW_0$, in particular for large time steps $\Delta t$. However, this is clearly out of the scope of the present paper and left to future research. }

Once the residual $\left\| \mathbf{f}\left( \tW_k \right) \right\|$ is below a prescribed tolerance $\epsilon$, the Newton iterations stop and the last $\tW_k$ computed is taken as the input value for the projection stage. 
Let us stress again that if the numerical flux function used for the approximation of the nonlinear convective terms is linear (like the semi-implicit Ducros flux shown later in this paper), then the Newton algorithm is no longer necessary and only the Krylov subspace method needs to be applied in order to solve the resulting linear system providing $\tW$. 

In what follows, we will detail the different parts of the proposed implicit method. First we will focus on the computation of the residual needed to define the stop criterium for the Newton method. Next, the computation of the approximated solution at the new time step using Krylov subspaces methods is described. Finally, the use of a preconditioner for the system solver is addressed.

\subsubsection{Residual computation}\label{sec.residualcomputation}
To calculate the residual $\mathbf{f}\left( \tW \right)$ we proceed in two steps. First, we compute the contribution of the nonlinear convective terms, the pressure at time $t^n$ and of the gravity source to the residual, denoted by $\delta\ttW$ and defined as 
\begin{equation}
	\delta\ttW = \Delta t\dive\Flux(\tW)  + \Delta t\gra \Press^n- \Delta t\rho\g. \label{eq.bigfunc_conv}
\end{equation}
It is computed via an implicit finite volume method on the dual mesh detailed later. Then, the obtained result  is included in the final equation that contains also the viscous contribution, i.e.  
\begin{equation}
	\mathbf{f}\left( \tW \right) = \tW - \W^n + \delta\ttW - \Delta t\dive\tbtau = 0, 
	\label{eq.bigfunc_visc}
\end{equation}
where $\dive\tbtau$ will be approximated using Crouzeix-Raviart finite elements on the primal grid.

\paragraph{Convective-pressure contribution to the residual.}
Let us start focusing in \eqref{eq.bigfunc_conv}. 
Integrating on a control volume $C_{i}$ and applying Gauss theorem to the convective terms, we have

\begin{equation}
	\delta \ttW_{i} = \frac{\Delta t}{\left|C_{i}\right|} \int\limits_{\Gamma_{i}} \Flux(\tW)\cdot \nun_{i} \dS
	+ \frac{\Delta t}{\left|C_{i}\right|} \int\limits_{C_{i}} \left( \grae \Press^{n}\right)_{i}   \dV
	- \frac{\Delta t}{\left|C_{i}\right|} \int\limits_{C_{i}} \rho_{i}\g \dV. \label{eq.bigfuncFV}
\end{equation} 
Given a $\tW$ the contribution due to advection,
\begin{equation}
	 \frac{\Delta t}{\left|C_{i}\right|} \int\limits_{\Gamma_{i}} \Flux(\tW)\cdot \nun_{i} \dS
	 =  \frac{\Delta t}{\left|C_{i}\right|} \sum_{j \in\mathcal{K}_{i}} |\Gamma_{ij}| \Flux(\tW_i,\tW_j)\cdot \nun_{ij},
	 \label{eqn.flux_residual}
\end{equation}
with $\mathcal{K}_{i}$ the set of neighbors of cell $C_i$, can be computed, for example, using a fully-implicit Rusanov flux function, $\phi^{\mathrm{R}}$,  of the form
\begin{align}
	\Flux^{\mathrm{R}}(\tW_i,\!\tW_j)\!\cdot\! \nun_{ij} & = \phi^{\mathrm{R}}(\tW_{i},\!\tW_{j},\!\W^{n}_{i},\!\W^{n}_{j},\!\nun_{ij})\notag\\ &= \halb \left( \Flux(\tWa{}{i}) +\Flux(\tWa{}{j})\right) \cdot \nun_{ij} - \halb \alpha_{ij}^{r}\left( \tW_{j}-\tW_{i} \right) \! ,\notag \\ 
	\alpha_{ij}^{r} &= 2 \left|u_{ij}^{n}\right|+c_{\alpha}, \qquad 
	u_{ij}^{n} = \halb \left(\vel^{n}_{i}+\vel^{n}_{j} \right) \cdot \nun_{ij}, \label{eqn.semiimplicitrusanov}
\end{align}
with $\nun_{ij}$ the unit normal vector of the face $\Gamma_{ij}$, outward oriented with respect to cell $C_i$, and $c_{\alpha}$ an extra artificial numerical viscosity coefficient which may increase the robustness of the final scheme (see \cite{Hybrid2,HybridALE} for a detailed analysis on the role and use of this coefficient). 
Instead of employing a fully-implicit Rusanov flux function one can also consider other numerical fluxes, such as a semi-discrete Ducros flux \cite{Ducros1999,Ducros2000,Moin2009}. In this case the numerical flux function $\phi^{\mathrm{D}}$ reads
\begin{align}
	\Flux^{\mathrm{D}}(\tW_i,\!\tW_j)\!\cdot\! \nun_{ij} & =  \phi^{\mathrm{D}}(\tW_{i},\!\tW_{j},\!\W^{n}_{i},\!\W^{n}_{j},\!\nun_{ij})	\notag \\  
	&= \halb\! \left( \tWa{}{i} +\tWa{}{j}\right) u_{ij}^{n} - \halb \alpha_{ij}^{d}\left( \tW_{j}-\tW_{i} \right) \! ,
	\notag \\ 		
	u_{ij}^{n} & = \halb \left(\vel^{n}_{i}+\vel^{n}_{j} \right) \cdot \nun_{ij}, 
	\qquad 
	\alpha_{ij}^{d} = \left|u_{ij}^{n}\right|+c_{\alpha},   
	\label{eqn.semiimplicitducros}
\end{align}
which is \textit{linear} in $\tW$, hence the Newton iterations outside the Krylov subspace algorithm are no longer needed. Consequently, the computational cost of the overall algorithm reduces. For numerical evidence of this fact see the numerical results reported in Section \ref{sec:numericalresults_TGV}.
Let us note that the Ducros flux function has also as second advantage: it is provably kinetic energy stable, as demonstrated in Section~\ref{sec.kineticenergystability}. 

Regarding the pressure term, we make use of the pressure at the previous time step, $\Press^{n}$, which is defined at the vertices of the primal grid. Thus its gradient can be easily approximated inside each primal element using a Galerkin approach based on $\mathbb{P}^{1}$ finite element basis functions. Then, a weighted average of the obtained values at the two parts of each finite volume gives a constant value in each dual element, $\left( \grae \Press^{n}\right)_{i}$ that is integrated in the cell.

Finally, assuming the gravity and the density to be constant in the dual elements, the source term can be directly integrated in each cell. In case more complex algebraic source terms are given, they are integrated on each cell using a sufficiently accurately numerical quadrature rule.

\paragraph{Viscous contribution to the residual.} \label{sec:viscousterms}
To discretize the viscous terms we make use of Crouzeix-Raviart basis functions on the primal grid, which have as nodes the barycenters of the edges/faces, i.e. the nodes of the dual mesh of the finite volume scheme where $\delta\ttW$ has been computed. Multiplying \eqref{eq.bigfunc_visc} by a test function $\varphi_{\ell}$ of the Crou\-zeix-Raviart type, integrating over the domain $\Omega$, applying integration by parts to the viscous terms and neglecting boundary contributions we get
\begin{equation}
	 \errres = \int\limits_{\Omega} \left( \tW - \W^n \right) \varphi_{\ell} \dV 
	+\int\limits_{\Omega} \delta\ttW \varphi_{\ell}\dV
	+\Delta t \int\limits_{\Omega}  \mu \gra \tW \cdot \gra \varphi_{\ell}\dV = 0. \label{eqn.bigfunc_visc_integ}
\end{equation}
We now write $\tW$, $\W^{n}$ and $\delta\ttW$ as a linear combination of the basis functions and 
the degrees of freedom $\tW$, $\W^{n}$ and $\delta\ttW$,
\begin{equation}
	\tW = \sum_{i=1}^{N} \varphi_{i} \tW_{i},\qquad \W ^{n}= \sum_{i=1}^{N} \varphi_{i} \W^{n}_{i},\qquad \delta \ttW = \sum_{i=1}^{N} \varphi_{i} \delta \ttW_{i} \label{eqn.linearcombination}
\end{equation}
with $N$ the number of basis functions, three in 2D and four in 3D. The degrees of freedom of the Crouzeix-Raviart elements are located at the barycenters of the primal faces, which can be also seen as the nodes of the dual elements that are used in the finite volume method for the nonlinear convective terms. Hence, the degrees of freedom of the CR elements \textit{coincide} with the data used in the finite volume scheme, which explains the choice for these particular finite elements to discretize the viscous terms in the new implicit hybrid FV/FE method presented in this paper. Substituting \eqref{eqn.linearcombination} in \eqref{eqn.bigfunc_visc_integ}, and denoting by $\mass$ and $\stiffness$ the global mass and stiffness matrix, respectively, yields the following discrete (nonlinear) system 
\begin{equation}
   \mathbf{f} \left( \htW \right) = \mass \left( \htW-\hW^{n}\right) 
	+ \mass \, \delta\httW
	+ \Delta t \mu \stiffness \, \htW = 0.
\end{equation}
Note that, to evaluate the above equation, we proceed via a loop on the primal elements which provides the value of the residual on each Crouzeix-Raviart vertex and it can be taken as the approximation on the corresponding dual cell. Thus there is no need to construct the global mass and stiffness matrix; the computation is done in a matrix-free manner, so communication between CPUs is minimized when using MPI parallelization. The element-local mass matrices of CR elements are diagonal in two space dimensions, 
but they are not in 3D. We therefore employ \textit{mass lumping} in order to obtain a diagonal 
mass matrix in three-dimensional calculations.

\subsubsection{Computation of $\tW$}\label{sec:wstarcomputation}

The value of the intermediate velocity at each Newton iteration, $\tW_k$, is computed using an iterative solver for nonsymmetric systems embedded within the Newton algorithm \eqref{eqn.newt.linsys}-\eqref{eqn.newt.update}. In particular, we consider two options: the generalized minimal residual method (GMRES), \cite{SS86}, and a biconjugate gradient stabilized method (BiCStab), \cite{Vorst92}. Both methods can be used to solve general nonsymmetric linear systems like the one in \eqref{eqn.newt.linsys}, thus a linearisation of the numerical flux is performed in the computation of the matrix-vector product of the Jacobian $\mathbf{J}(\mathbf{W}_k^*)$ with the sought newton step $\Delta \mathbf{W}_k^*$.

The selected iterative methods are based on obtaining iteratively and approximate solution of system   \eqref{eqn.newt.linsys} via a series of matrix-vector products, i.e. we never assemble or compute the Jacobian directly, but make use of a matrix-free implementation of the linear solver. 

\paragraph{Matrix-vector product.}
The semi-implicit Ducros flux function \eqref{eqn.semiimplicitducros} is already linear, hence the associated matrix-vector product needed by the linear solver simply reads 
\begin{equation} 
	\frac{\partial \Flux^{\mathrm{D}}(\tW_i,\!\tW_j) \!\cdot\! \nun_{ij}}{\partial (\tW_i, \tW_j )} \cdot (\Delta \tW_i,  \Delta \tW_j )   
	=  \halb u_{ij}^{n} \left( \Delta \tWa{}{i} + \Delta \tWa{}{j}\right)  - \halb \alpha_{ij}^{d}\left( \Delta \tW_{j} - \Delta \tW_{i} \right). 
	\label{eqn.lin.ducros}
\end{equation} 
On the other hand, the linearization of the Rusanov flux function in normal direction, which is quadratic in $\tW$, reads
\begin{equation} 
	\frac{\partial \Flux^{\mathrm{R}}(\tW_i,\!\tW_j) \!\cdot\! \nun_{ij}}{\partial (\tW_i, \tW_j )} \cdot (\Delta \tW_i,  \Delta \tW_j )   
	=  
	 \left( u_i^* \Delta \tWa{}{i} + u_j^* \Delta \tWa{}{j}\right)  - \halb \alpha_{ij}^{r}\left( \Delta \tW_{j} - \Delta \tW_{i} \right), 
	\label{eqn.lin.rusanov}
\end{equation} 
with the normal velocities 
\begin{equation}
    u_i^* = \frac{1}{\rho_i} \tW_i \cdot \nun_{ij}, \qquad 
	u_j^* = \frac{1}{\rho_j} \tW_j \cdot \nun_{ij}, 
\end{equation}
and where $\tW_{i},\,\tW_{j}$ denote the intermediate momentum in cell $C_i$ and $C_j$ of the current Newton iteration, $\Delta \tW_{i},\, \Delta \tW_{j}$ the associated Newton steps, while $\W_{i}^{n},\W_{j}^{n}$ are the momentum obtained at the previous time step. Note that in order to ease the linearization of the Rusanov flux the numerical dissipation coefficient $\alpha_{ij}^n$ has deliberately been computed at the old time $t^n$. 

Finally, to account for the viscous fluxes we use the Crouzeix-Raviart approach presented in Section~\ref{sec:viscousterms}, which is already linear in the unknowns and thus the application of the matrix-vector product in the linear solver is immediate. 

\paragraph{Preconditioner.}
To improve the convergence of the iterative solver, a preconditioner is used together with the GMRES or BiCGStab algorithms. In particular, we consider a Symmetric Gauss-Seidel (SGS) method,  \cite{YJ88,LBL98,CPTU18,MN00}.  
Consequently,  instead of directly solving \eqref{eqn.newt.linsys} we solve an equivalent, explicitly preconditioned system of the form 
\begin{equation} 
	\widetilde{\mathbf{J}}\left(  \tW_k  \right) \, \Delta\tW_k = -\widetilde{\mathbf{f}}\left( \tW_k \right), 
\end{equation} 	
where the preconditioned Jacobian matrix  $\widetilde{\mathbf{J}}\left(  \tW_k  \right)$ and the preconditioned right hand side $\widetilde{\mathbf{f}}\left( \tW_k \right)$ are obtained by multiplying the original system by a preconditioning matrix $\mathcal{P}^{-1}$ from the left as: 
\begin{equation}
	\mathcal{P}^{-1} {\mathbf{J}}\left(  \tW_k  \right) \, \Delta\tW_k = - \mathcal{P}^{-1} {\mathbf{f}}\left( \tW_k \right). 
\end{equation} 
Hence, within the Krylov method, after each matrix-vector product we apply the preconditioner (see \cite{LBL98} for a detailed pseudocode in the framework of GMRES algorithm).

The LU-SGS preconditioning developed for hyperbolic equations \cite{YJ88,LBL98,CPTU18,MN00} is based on the idea of using the linearized numerical flux function, already employed inside the Krylov algorithm, to construct $\mathcal{P}$. Since the momentum equations include viscous terms, they are also taken into account in the preconditioner using the Crouzeix-Raviart finite element approach already employed in Section \ref{sec:wstarcomputation}. Note that the SGS method proposed in \cite{YJ88,LBL98,CPTU18,MN00} is slightly different from the standard SGS preconditioner \cite{MeisterLGS}. 
Formally, the Jacobian matrix $\mathbf{J}(\tW_k)$ is decomposed into its diagonal $\mathcal{D}$, a lower triangular matrix $\mathcal{L}$ and an upper triangular part $\mathcal{U}$ as follows: 
\begin{equation}
	\mathbf{J}(\tW_k) = \mathcal{L} + \mathcal{D} + \mathcal{U}.  
	\label{eqn.decomposition} 
\end{equation}
The SGS preconditioner according to  \cite{YJ88,LBL98,CPTU18,MN00} reads
\begin{equation}
	\mathcal{P}^{-1} = \left(  \mathcal{D} + \mathcal{U} \right)^{-1} \mathcal{D} \, \left( \mathcal{D} + \mathcal{L}  \right)^{-1}.  
	\label{eqn.sgs3} 
\end{equation}
The application of the preconditioner to a generic vector $\mathbf{Q}$, i.e. the calculation of $\mathcal{P}^{-1} \mathbf{Q}$, then consists in a forward sweep
\begin{equation}
	 \widetilde{\mathbf{Q}}  =  \left( \mathcal{D} + \mathcal{L}  \right)^{-1} \mathbf{Q},
	 \label{eqn.forward.sweep} 
\end{equation}
which is followed by a backward sweep
\begin{equation}
	 \mathcal{P}^{-1} \mathbf{Q} =  
	 \left(  \mathcal{D} + \mathcal{U} \right)^{-1} \mathcal{D} \, \widetilde{\mathbf{Q}}.
	 \label{eqn.backward.sweep} 
\end{equation} 
Thanks to the linearization of the numerical flux functions \eqref{eqn.lin.ducros}  or  \eqref{eqn.lin.rusanov}, both sweeps can be carried out in a matrix-free fashion, where in the forward sweep only flux contributions from elements with lower element number are taken into account and in the backward sweep only flux contributions from elements with higher element number.  

To obtain the contributions of the upper and lower triangular part easily on a general unstructured mesh we follow \cite{LBL98,CPTU18,MN00} and split the set of neighbors $\mathcal{K}_i$ of a cell $C_i$ into two subsets $\mathcal{K}_i^-$ and $\mathcal{K}_i^+$ with $\mathcal{K}_i = \mathcal{K}_i^- \cup \mathcal{K}_i^+$ so that elements $j \in \mathcal{K}_i^- < i$ and $ j \in \mathcal{K}_i^+ > i$. The contribution of the linearized convective terms to the Jacobian matrix of the Newton method is in the following denoted by 
\begin{equation}
	\Fpm = \Fpm^{L} + \Fpm^{D} + \Fpm^{U}. 
\end{equation}
The lower diagonal part of the matrix-vector product for a generic cell $C_i$ reads 
\begin{equation}
	\Fpm^{L}_i \Delta \tW =  \frac{\Delta t}{|C_i|}  \, \sum \limits_{j \in \mathcal{K}_i^-}  |\Gamma_{ij}| \, \frac{\partial \Flux(\tW_i,\!\tW_j) \!\cdot\! \nun_{ij}}{\partial \tW_j } \cdot \Delta \tW_j,  
\end{equation}
the contribution to the upper diagonal part is given by 
\begin{equation}
	\Fpm^{U}_i \Delta \tW =  \frac{\Delta t}{|C_i|} \, \sum \limits_{j \in \mathcal{K}_i^+}  |\Gamma_{ij}| \, \frac{\partial \Flux(\tW_i,\!\tW_j) \!\cdot\! \nun_{ij}}{\partial \tW_j } \cdot \Delta \tW_j,  
\end{equation}
while on the diagonal we have the contribution 
\begin{equation}
	\Fpm^{D}_i \Delta \tW =  \frac{\Delta t}{|C_i|} \, \sum \limits_{j \in \mathcal{K}_i}  |\Gamma_{ij}| \, 
	\frac{\partial \Flux(\tW_i,\!\tW_j) \!\cdot\! \nun_{ij}}{\partial \tW_i } \cdot \Delta \tW_i, 	 
\end{equation}
with the linearizations given in \eqref{eqn.lin.ducros} and \eqref{eqn.lin.rusanov} for the Ducros flux and the Rusanov flux, respectively. 
In particular, for the semi-implicit Ducros flux we have 
\begin{equation}
	\Fpm^{L}_i \Delta \tW =  \frac{\Delta t}{|C_i|}  \, \sum \limits_{j \in \mathcal{K}_i^-}  |\Gamma_{ij}| \, \halb \left( u_{ij}^{n} - \alpha_{ij}^{d} \right) \Delta \tWa{}{j},  
\end{equation}
\begin{equation}
	\Fpm^{U}_i \Delta \tW =  \frac{\Delta t}{|C_i|} \, \sum \limits_{j \in \mathcal{K}_i^+}  |\Gamma_{ij}| \, \halb 
	\left( u_{ij}^{n} - \alpha_{ij}^{d} \right) \Delta \tWa{}{j},  
\end{equation}
and 
\begin{equation}
	\Fpm^{D}_i \Delta \tW =  \frac{\Delta t}{|C_i|} \, \sum \limits_{j \in \mathcal{K}_i}  |\Gamma_{ij}| \, 
	\halb \left( u_{ij}^{n} + \alpha_{ij}^{d} \right) \Delta \tW_{i}. 	 
\end{equation}

The final Jacobian, including the viscous terms, is then formally given by 
\begin{equation}
	\mathbf{J}(\tW_k) = \mathbf{M} + \mathbf{M} \, \left( \Fpm^{L} + \Fpm^{D} + \Fpm^{U} \right) - \Delta t \mu \left( \mathbf{K}^L + \mathbf{K}^D + \mathbf{K}^U \right), 
\end{equation}
where the stiffness matrix $\mathbf{K}$ has also been split into its diagonal part $\mathbf{K}^{D}$, an upper triangular matrix $\mathbf{K}^U$ and a lower triangular one $\mathbf{K}^L$. Note that the decomposition of the Jacobian itself is never needed in the preconditioner, only the associated matrix-vector products with a generic vector $\mathbf{Q}$. 

The matrix-vector products associated with $\mathbf{K}^L \Delta \tW$ and $\mathbf{K}^U \Delta \tW$ are again easily carried out at the aid of the sets $\mathcal{K}_i^-$ and $\mathcal{K}_i^+$, as for the convective terms. 
Since the global mass matrix is diagonal (this is naturally the case for P1 CR elements in 2D and via mass-lumping in 3D) the diagonal of the Jacobian, which is explicitly needed in the SGS preconditioner, reads 
\begin{gather}
	\mathcal{D} =   \mathbf{M}  + \mathbf{M} \Fpm^{D}  + \Delta t \, \mu  \, \mathbf{K}^{D}.
\end{gather}

The symmetric Gauss-Seidel method shown above cannot be parallelized, hence in the parallel implementation of our scheme we carry out the two sweeps \eqref{eqn.forward.sweep} and \eqref{eqn.backward.sweep} only over the elements contained in each CPU. However, in all our numerical experiments carried out in parallel we have found that this is not a problem, since in our scheme the SGS method is only used as a preconditioner for a GMRES / BiCGStab algorithm, which instead can be properly parallelized by exchanging the norms of the residuals between all CPUs. 

\paragraph{Mesh reordering.} 
Reordering the elements of the mesh can bring a noticeable performance improvement in the preconditioner step and is furthermore also well suited for parallel implementation, hence increasing the overall efficiency of the Newton-Krylov method. The reordering leads to a reduction of the matrix bandwidth, such that most non-zero elements are present near the main diagonal. In practice, if the elements are reordered properly along the main flow direction, the combination of the forward and backward sweep of the SGS preconditioner perform similar to a direct solver based on Gauss elimination, in particular for the Ducros flux. For the simple linear scalar advection equation in 1D based on the implicit upwind scheme, the SGS preconditioner even corresponds to the exact direct solver.    

While for structured collocated grids quite easy reordering techniques are available, for instance the chess board colouring \cite{MenPav17}, performing an accurate and efficient reordering 
of elements in the case of unstructured grids can become quite challenging and in general requires the reordering of graphs, based on the local flow velocity and mesh connectivity, see e.g. \textcolor{black}{\cite{SN97,CWM15,BEPRT20}}. 

Here we propose the use of a much simpler algorithm, which may not provide the best possible reordering, but it still considerably reduces the number of iterations needed to achieve convergence in the linear solver. At the same time it is compatible and easy to implement for parallel codes, since it does not need any extra communication between CPUs.

As already \textcolor{black}{mentioned} before, the approach considered in this paper is very simple and is easy to implement. It uses a purely spatial reordering of the control volumes of the dual mesh based on an \textit{a priori} defined main flow direction, say $\mathbf{v}_0$, that needs to be determined for each test problem. In the simplest case $\mathbf{v}_0$ corresponds to one of the main coordinate directions $\mathbf{e}_i$. We then construct a total number $N_B$ of equidistant \textit{one-dimensional bins} of size $\Delta \lambda$ along the chosen  unitary main flow direction $\mathbf{v}_0$ with $\left\| \mathbf{v}_0 \right\|=1$ and define a straight line starting in a point $\mathbf{x}_0 \in \Omega$ as
\begin{equation}
	\mathbf{x} = \mathbf{x}_0 + \lambda \mathbf{v}_0, \qquad \lambda \in \mathbb{R}. 
\end{equation} 
The uniform grid of bins of length $\Delta \lambda$ along $\mathbf{v}_0$ is defined as $\lambda_j = j \Delta \lambda$ with $j \in \mathbb{Z}$. For the mesh reordering we then run over the dual mesh, i.e. over the control volumes $C_i$, and compute the projection of the barycenter $\mathbf{x}_i$ of cell $C_i$ onto the straight line $\mathbf{x}$ as 
\begin{equation}
	 \xi_i = \left( \mathbf{x}_i - \mathbf{x}_0 \right) \cdot \mathbf{v}_0. 
\end{equation} 
The control volume $C_i$ is then sorted into bin number 
\begin{equation}
	 j = \textnormal{int} \left( \frac{\xi_i}{\Delta \lambda} \right). 
\end{equation}
The equidistant one-dimensional bins of size $\Delta \lambda$ are placed along the centerline from the inlet's network towards the outlets. The control volumes of the dual mesh are then sorted into bin numbers according to the concept of the Voronoi diagram.
Once all dual elements have been sorted into their corresponding bins they are reordered by running over all bins in ascending order, i.e. along the flow direction $\mathbf{v}_0$. Within each bin the element order is simply determined by the order in which the elements have been put into the bin, i.e. we use a first-in-first-out (FIFO) principle.  
In case MPI parallelization is used, each CPU simply performs the same type of reordering, but only for its own cells. The main shortcoming of this approach is that the number of bins to be employed is not automatically computed but must be estimated taking into account the size of the domain, the characteristic dual element size and the mesh distribution between CPUs. Furthermore, the choice of the main flow direction according to which elements are sorted depends on the test case and may not be unique for complex flows. \textcolor{black}{When dealing with complex geometries, representing arterial or venous vessels, the main flow direction $\mathbf{v_0}$ is assumed to be driven by the centerlines of the 3D volume meshes.}

\subsubsection{Second order scheme}

The former approach leads to a first order scheme in space and time. To improve the accuracy in space a MUSCL or ADER methodology can be followed, \cite{TT02,TT05,TT06,Toro,BFTVC17}. Accordingly, the fluxes in \eqref{eqn.flux_residual} and \textcolor{black}{\eqref{eqn.lin.ducros}-\eqref{eqn.lin.rusanov}} are calculated using the boundary extrapolated values at the dual faces. To compute the reconstructed values, we first evaluate the discrete gradients of each momentum variable to be used employing the Crouzeix-Raviart basis functions. Then, they are interpolated from the primal elements to the dual cells as a weighted average. Further details on this approach can be found for instance in \cite{Hybrid2,HybridALE}. Finally the extrapolated momentum at each side of a face shared by elements $C_i$ and $C_j$,  $\W_{ij}^{-}$, $\W_{ij}^{+}$, are computed from the momentum on the cell of the corresponding side, $\W_{i}$, $\W_{j}$,  plus the contribution of its gradients, $\nabla \W_{i}$, $\nabla \W_{j}$ as follows:
\begin{equation*}
	\W_{ij}^{-} =\W_{i} + \nabla \W_{i}\cdot \Delta \x_{ij}^{-},\qquad \W_{ij}^{+} = \W_{j} +\nabla \W_{j}\cdot \Delta \x_{ij}^{+}
\end{equation*}
with $\Delta \x_{ij}^{-}$, $\Delta \x_{ij}^{+}$ the vectors between the barycenters of cells $C_i$ and $C_j$ and the barycenter of the common face $\Gamma_{ij}$.

\subsubsection{Discrete kinetic energy stability}\label{sec.kineticenergystability}
As mentioned in Section \ref{sec:goveq}, for an inviscid fluid under zero gravity effects system \eqref{eqn.NavierStokes} admits an extra energy conservation law in terms of the kinetic energy density. Hence, when deriving a discretization of the Euler equations it could be of interest to check if the obtained scheme is also kinetic energy stable. As demonstrated below, this property is verified for the proposed scheme when the Ducros numerical flux function is used.

\begin{theorem}
	Assuming constant density ($\rho = const.$), vanishing viscosity ($\mu=0$), zero gravity ($\g=0$), 
	a divergence-free velocity field at time $t^n$  
	\begin{equation}
		  \sum \limits_{j\in\mathcal{K}_{i}} | \Gamma_{ij} | u_{ij}^n = 0, \qquad 
		  u_{ij}^n = \halb \left( \mathbf{u}_i^n + \mathbf{u}_j^n \right) \cdot \nun_{ij} \label{eqn.divfree}
	\end{equation}
	and vanishing boundary fluxes ($u_{ij}^n = 0 \, \, \forall \Gamma_{ij} \in \partial \Omega$) 
	the finite volume scheme for the discretization of the nonlinear convective terms based on the 
	semi-implicit Ducros flux function 
	\begin{equation}
		\mathbf{u}_i^* = \mathbf{u}_i^n  - 
		\frac{\Delta t}{|C_i|} \sum \limits_{j\in\mathcal{K}_{i}}  | \Gamma_{ij} |  
		\left( \halb u_{ij}^n  \left( \mathbf{u}_i^* + \mathbf{u}_j^* \right) - 
		\halb {\alpha}_{ij}^n \left( \mathbf{u}_j^* - \mathbf{u}_i^* \right) \right), \label{eqn.momdiscFV} 
	\end{equation}
	with ${\alpha}_{ij}^n = | u_{ij}^n | + c_{\alpha} \geq 0$ is kinetic energy stable in the sense 
	\begin{equation}
		\int \limits_{\Omega} \halb \left( \mathbf{u}^* \right) ^2 d\mathbf{x} \leq 
		\int \limits_{\Omega} \halb \left( \mathbf{u}^n \right) ^2 d\mathbf{x}. 		  
	\end{equation}
\end{theorem}
\begin{proof}
Taking the dot product of $\tvel_{i}$ by \eqref{eqn.momdiscFV} yields
\begin{equation}
	\tvel_i \cdot \tvel_i = \tvel_i \cdot \vel_i^n  - 
	\frac{\Delta t}{|C_i|} \sum \limits_{j\in\mathcal{K}_{i}}  | \Gamma_{ij} |  
	\left( \halb u_{ij}^n  \tvel_i \cdot \left( \tvel_i + \tvel_j \right) - 
	\halb {\alpha}_{ij}^n \tvel_i \cdot \left( \tvel_j - \tvel_i \right) \right).\label{eqn.kienesta1}
\end{equation}
On the other hand, taking into account the divergence-free property of the velocity, \eqref{eqn.divfree},  we have
\begin{gather}
	 \sum \limits_{j\in\mathcal{K}_{i}}  | \Gamma_{ij} |  
	 \halb u_{ij}^n  \tvel_i \cdot \left( \tvel_i + \tvel_j \right)  
	=\sum \limits_{j\in\mathcal{K}_{i}}  | \Gamma_{ij} |  
	\halb u_{ij}^n  \left( \tvel_i\right)^2
	+ \sum \limits_{j\in\mathcal{K}_{i}}  | \Gamma_{ij} |  
	 \halb u_{ij}^n  \tvel_i \cdot \tvel_j \notag\\
	 =\halb \left( \tvel_i\right)^2\sum \limits_{j\in\mathcal{K}_{i}}  | \Gamma_{ij} |  
	  u_{ij}^n  
	 + \sum \limits_{j\in\mathcal{K}_{i}}  | \Gamma_{ij} |  
	 \halb u_{ij}^n  \tvel_i \cdot \tvel_j 
	 =\sum \limits_{j\in\mathcal{K}_{i}}  | \Gamma_{ij} |  
	 \halb u_{ij}^n  \tvel_i \cdot \tvel_j,
\end{gather}
where the second equality comes from the fact that $\tvel_{i}$ does not depend on $j$.
Applying the former relation to \eqref{eqn.kienesta1} and adding and subtracting $$\frac{\Delta t}{|C_i|} \sum\limits_{j\in\mathcal{K}_{i}}  | \Gamma_{ij} | \halb\tvel_j \cdot\halb \alpha_{ij}^n\left( \tvel_{j}-\tvel_{i} \right)$$ and $\halb\vel_{i}^{n}\cdot \vel_{i}^{n}$ at the right hand side, lead to

\begin{gather}
	\halb\tvel_i \cdot \tvel_i + \halb\tvel_i \cdot \tvel_i = \halb \tvel_i \cdot \vel_i^n + \halb \tvel_i \cdot \vel_i^n  - \halb\vel_{i}^{n}\cdot \vel_{i}^{n} + \halb\vel_{i}^{n}\cdot \vel_{i}^{n}
	\notag \\
	-\frac{\Delta t}{|C_i|} \sum \limits_{j\in\mathcal{K}_{i}}  | \Gamma_{ij} |  
	  \left(  u_{ij}^n \halb \tvel_i \cdot \tvel_j
	- \halb {\alpha}_{ij}^n \halb\tvel_i \cdot \left( \tvel_j - \tvel_i \right) 
	- \halb {\alpha}_{ij}^n \halb\tvel_j \cdot \left( \tvel_j - \tvel_i \right)  \right. \notag\\\left.
	- \halb {\alpha}_{ij}^n \halb\tvel_i \cdot \left( \tvel_j - \tvel_i \right)
	+ \halb {\alpha}_{ij}^n \halb\tvel_j \cdot \left( \tvel_j - \tvel_i \right)
	\right).
\end{gather}
Reordering terms, we get

\begin{gather}
	\halb\tvel_i \cdot \tvel_i = \halb\left( \tvel_i -\vel_{i}^{n} \right) \cdot \left( \vel_i^n - \tvel_i \right)     + \halb\vel_{i}^{n}\cdot \vel_{i}^{n}
	\notag\\
	- \frac{\Delta t}{|C_i|} \sum \limits_{j\in\mathcal{K}_{i}}  | \Gamma_{ij} |  
	\left(  u_{ij}^n \halb \tvel_i \cdot \tvel_j - \halb {\alpha}_{ij}^n   \left(\halb \left( \tvel_j\right) ^{2} - \halb\left( \tvel_i\right) ^2 \right) 
	+ \halb {\alpha}_{ij}^n \halb \left( \tvel_j - \tvel_i \right)^2
	\right).
\end{gather}
Since ${\alpha}_{ij}^n \geq 0$, then

\begin{gather}
	\halb\left( \tvel_i\right)^2 = \halb\left( \vel_{i}^{n}\right)^2 
	- \frac{\Delta t}{|C_i|} \sum \limits_{j\in\mathcal{K}_{i}}  | \Gamma_{ij} |  
	\left(  u_{ij}^n \halb \tvel_i \cdot \tvel_j - \halb {\alpha}_{ij}^n   \left(\halb \left( \tvel_j\right) ^{2} - \halb\left( \tvel_i\right) ^2 \right) \right) \notag\\
	- \halb\left( \tvel_i -\vel_{i}^{n} \right)^2
	- \frac{\Delta t}{|C_i|} \sum \limits_{j\in\mathcal{K}_{i}}  | \Gamma_{ij} |  
	\left(  \halb {\alpha}_{ij}^n \halb \left( \tvel_j - \tvel_i \right)^2 	\right) \notag\\
	\leq \halb\left( \vel_{i}^{n}\right)^2 
	- \frac{\Delta t}{|C_i|} \sum \limits_{j\in\mathcal{K}_{i}}  | \Gamma_{ij} |  
	\left(  u_{ij}^n \halb \tvel_i \cdot \tvel_j - \halb {\alpha}_{ij}^n   \left(\halb \left( \tvel_j\right) ^{2} - \halb\left( \tvel_i\right) ^2 \right) \right)
\end{gather}
which corresponds to a discrete kinetic energy inequality in each cell $C_i$.  
Finally, integrating over the domain results
\begin{gather}
	\int\limits_{\Omega}\halb\left( \tvel\right)^2 \dV \leq \int\limits_{\Omega}\halb\left( \vel^{n}\right)^2 \dV
	-  \Delta t\sum\limits_{C_i} \sum \limits_{j\in\mathcal{K}_{i}}  | \Gamma_{ij} |  
	\left(  u_{ij}^n \halb \tvel_i \cdot \tvel_j - \halb {\alpha}_{ij}^n   \left(\halb \left( \tvel_j\right)^{2} - \halb\left( \tvel_i\right)^2 \right) \right).\label{eqn.kienesta7}
\end{gather}
Since $\nun_{ij}=-\nun_{ji}$ when summing the fluxes over all cells they cancel,
$$| \Gamma_{ij} |u_{ij}^n \halb \tvel_i \cdot \tvel_j=-| \Gamma_{ji} | u_{ji}^n \halb \tvel_j \cdot \tvel_i,$$
apart from the ones on the boundary of the domain.  Thus, assuming vanishing boundary fluxes and taking into account that the last term in \eqref{eqn.kienesta7} is the dissipative kinetic energy flux, we obtain
\begin{gather}
	\int\limits_{\Omega}\halb\left( \tvel\right)^2 \dV \leq \int\limits_{\Omega}\halb\left( \vel^{n}\right)^2 \dV,
\end{gather}
hence the scheme is kinetic energy stable.
\end{proof}

\subsection{Projection stage}
A standard continuous finite element method is adopted to solve the pressure Poisson equation  \eqref{eqn.pressure.poisson}. Let $\psi\in V_0$ be a test function, $V_0=\{\psi\in\mathcal{H}^1:\int_\Omega \psi dV=0\}$. Multiplication of \eqref{eqn.pressure.poisson} by a test  function $\psi$ and integration over the domain $\Omega$ yields 
\begin{equation}
    \int\limits_\Omega \nabla^2 (\Press^{n+1}-\Press^n) \psi\, \dV=\frac{1}{\Delta t}\int\limits_\Omega \nabla \cdot \tW \psi \, \dV.
    \label{eq:weak1}
\end{equation}
Then, applying Green's formula to both sides we obtain the following weak problem:
\textit{Find $\delta\Press=\Press^{n+1}-\Press^n\in V_0$ satisfying
\begin{equation}
    \int\limits_\Omega \grae\delta\Press\cdot\grae \psi\, \dV=\frac{1}{\Delta t}\int\limits_\Omega \tW\cdot \grae \psi\, \dV-\frac{1}{\Delta t}\int\limits_{\partial\Omega}\psi\,  \W^{n+1}\cdot\nun \dS
    \label{eq:weak}
\end{equation}
for all $z\in V_0$}.

\noindent The weak problem \eqref{eq:weak} can be seen as a Poisson problem for the unknown $\delta\Press$ so it can be discretized using $\mathbb{P}^1$ basis functions, which leads to a symmetric positive definite system once appropriate boundary conditions have been applied. Then, a matrix-free conjugate gradient method is employed to solve the resulting final system.

\subsection{Post-projection stage}
Once the pressure correction $\delta\Press$ is obtained at the vertex of the primal grid, we can recover the pressure at the new time step as $\Press^{n+1}=\Press^{n}+\delta\Press$. Moreover, the corrected value for the linear momentum is computed on the dual mesh by updating $\tW$ with $\gra \delta\Press$ according to \eqref{eqn.momentum_correction}.
The needed pressure gradient $\gra \delta\Press$ is first approximated at each primal element by using the gradients of the classical basis function for $\mathbb{P}^1$ finite elements. Then, it is interpolated on the dual grid by considering the weighted average of the contributions of the two primal subtetrahedra used to build each dual cell. Further details on this update as well as the modified method for the solution of the alternative pressure system \eqref{eq:timesys2_npc}-\eqref{eq:timesys3_npc} can be found in \cite{HybridALE}.

\subsection{Boundary conditions}
In this section we present the key points to treat the different types of boundary conditions to be employed in the test cases analyzed in the next section: periodic boundary conditions, velocity inlet, pressure inlet, pressure outlet, Dirichlet boundary conditions and viscous and inviscid walls. 

If a periodic boundary condition is considered, the neighbouring triangular / tetrahedral dual elements through the periodic boundaries, which should in fact be quadrilaterals / polyhedrons of six faces, are identified. Within the transport-diffusion stage the contributions due to convective and diffusive terms are computed at each half of this theoretically merged element 
and then a correction is performed to incorporate the contributions of the other periodic half. This results on a duplication of the boundary elements must be taken into account when computing the residuals. Therefore each of them contributes only with half its weight. On the other hand, the periodic vertices on the primal grid are merged and the resulting mesh is employed directly to get the pressure in the projection stage. Let us note that when periodic boundary conditions are imposed everywhere the pressure problem is singular, since the pressure could freely change by an additive constant. To avoid it, we simply fix one of the vertices to a predefined value that may be taken from a mean pressure, the exact solution, if available, or the initial data.

For the remaining boundary conditions, we must take into account that the algorithm employed in the transport-diffusion stage requires the computation of a residual, matrix-vector products and a preconditioner so boundary conditions must be consistently imposed throughout the Newton-Krylov algorithm as well as for the pressure system:

\begin{itemize}
	\item Velocity inlet boundary conditions. The momentum at the boundary, $\rho\vel_{BC}$, is weakly imposed by computing the convective fluxes within the residual considering a right state of the form:
	\begin{equation}
		\rho\vel_j = 2\rho\vel_{BC} - \rho\vel_i
	\end{equation}
	with $\rho\vel_i$ the linear momentum computed at the boundary cell. 
	Regarding the computation of the linearized fluxes inside the matrix-vector product and the flux contribution in the preconditioner, we define the right state to be the opposite of the inner one. The given velocity is also employed to compute the boundary integral of the pressure system, i.e. the last term in \eqref{eq:weak}.
	
	\item Dirichlet boundary conditions for the velocity field or viscous wall/no slip boundary conditions. The velocity is imposed strongly in the corresponding boundary. To this end, the value obtained once the convective and diffusive terms contributions have been computed is overwritten with the exact momentum. This correction is also performed after the post-projection stage. Like for velocity inlet boundaries, the pressure is computed taking into account the exact velocity in the boundary integral.
	
	\item Inviscid wall boundary. In this case the velocity is weakly defined inside the momentum solver by considering the state:
	\begin{equation}
		\rho\vel_j = \rho\vel_i- 2 \left( \rho\vel_{i}\cdot\nun\right)  \nun.
	\end{equation}
	For the matrix-vector products and the preconditioner the approach followed is the same than for the velocity inlet.
    Meanwhile, inside the pressure system the boundary integral contribution is set to zero.
    
    \item Pressure outlet. The velocity is left free in the momentum equations. To this end the convective fluxes are computed using the inner state. Then the pressure is imposed at the corresponding nodes in the pressure solver.
    
    \item Pressure inlet boundary conditions. They are a combination of a velocity inlet for the solution of the momentum equations and a pressure outlet boundary condition for the pressure system.
\end{itemize}

\section{Numerical results}\label{sec:numericalresults}
The proposed methodology is validated at the aid of several benchmarks from fluid mechanics including the Taylor-Green vortex test case and the ABC problem, to analyse the accuracy, as well as the first problem of Stokes and the lid driven cavity, employed to study the behaviour of the method for viscous flows. Having in mind possible applications of the scheme to blood low simulations in vessels, the flow over a backward-facing step and the flow over a cylinder for the inviscid and a set of viscous flows are also analysed. Moreover, we study the classical Hagen-Poiseuille and Womersley problems and the flow through an ideal artery with stenosis. Finally a real 3D coronary tree geometry is considered.

Unless stated otherwise, the density is set to $\rho=1$ and the convective-diffusive subsystem is solved by employing the SGS preconditioned Newton-BiCGStab algorithm with the second order in space Ducros flux function and $c_{\alpha}=0$.

\subsection{Convergence study in 2D} \label{sec:numericalresults_TGV}
As first test case, we consider the classical steady state Taylor-Green vortex benchmark defined in $\Omega = [0, 2\pi]\times[0, 2\pi]$, for which the known exact solution reads:
\begin{gather*}
    p(x_1,x_2,t)= \frac{1}{4}(\cos(2x_1)+\cos(2x_2)), \quad \rho(x_1,x_2,t) =1,\\
    u_1(x_1,x_2,t)=\sin(x_1)\cos(x_2),\quad 
    u_2(x_1,x_2,t)=-\cos(x_1)\sin(x_2).
\end{gather*}

To perform the analysis of the error and order of accuracy, we employ the five meshes with decreasing cell sizes presented in Table~\ref{tab:TGV1}.
We test the SGS-preconditioned Newton-GMRES method and the preconditioned Newton-BiCGStab method employing a Ducros semi-implicit scheme of both first and second order. Moreover, simulations are run also considering the Newton-BiCGStab-SGS method with the first and second order Rusanov scheme as numerical flux.   
The obtained $L_2$ error norms at the final time, $t=1.0$, and the corresponding convergence rates, computed for any variable $V$ as
\begin{equation}
    E(V)_{M_i}=||V-V_{M_i}||_{L_2(\Omega)},\qquad 
    o_{V_{M_i/M_j}}=\frac{\log(E(V)_{M_i}/E(V)_{M_j})}{\log(h_{M_i}/h_{M_j})},
    \label{eq:errors}
\end{equation}
are reported in Table~\ref{tab:TGV2}. 
We observe that the sought order of accuracy is obtained for all test cases considered.

\begin{table}[H]
    \centering    
    \begin{tabular}{lccc}
    \hline
    Mesh & Primal elements & Vertices & Dual elements \\
    \hline
    $M_1$ & 512 & 289 & 800 \\
    $M_2$ & 2048 & 1089 & 3136 \\
    $M_3$ & 8192 & 4225 & 12416 \\
    $M_4$ & 32768 & 16129 & 49408 \\
    $M_5$ & 131072 & 66049 & 197120 \\
    \hline
    \end{tabular}
    \caption{Taylor-Green vortex. Mesh features.}
    \label{tab:TGV1}
\end{table}

\begin{table}
    \centering
    \begin{tabular}{cccccccccc}
    \hline
       & $E_{M_1}$ & $E_{M_2}$ & $E_{M_3}$ & $E_{M_4}$ & $E_{M_5}$ & $o_{M_1/M_2}$ & $o_{M_2/M_3}$ & $o_{M_3/M_4}$& $o_{M_4/M_5}$\\
    \hline 
    \multicolumn{10}{c}{GMRES, Ducros, 1$^{st}$ order }
    \\\hline
     $p$ & 0.639 & 0.327 & 0.151 & 7.1e-2 & 3.5e-2 & 0.96 & 1.12 & 1.09 & 1.03\\ 
     \rule[-0.3cm]{0mm}{0cm}
     $\mathbf{W}$ & 0.798 & 0.45 & 0.246 & 0.131 & 6.8e-2 & 0.83 & 0.87 & 0.91 & 0.95\\ 
    \hline
    \multicolumn{10}{c}{GMRES, Ducros, 2$^{nd}$ order }
    \\\hline
     $p$ & 0.173 & 5.1e-2 & 1.3e-2 & 3.3e-3 & 8.2e-4 & 1.77 & 1.96 & 1.99 & 2.0\\ 
     \rule[-0.3cm]{0mm}{0cm}
     $\mathbf{W}$ & 4.7e-2 & 9.3e-3 & 2.1e-3 & 5.2e-4 & 1.2e-4 & 2.32 & 2.13 & 2.03 & 2.17\\ 
    \hline
    \multicolumn{10}{c}{BiCGStab, Ducros, 1$^{st}$ order }
    \\\hline
     $p$ & 0.639 & 0.327 & 0.151 & 7.1e-2 & 3.5e-2 & 0.96 & 1.12 & 1.09 & 1.03\\ 
     \rule[-0.3cm]{0mm}{0cm}
     $\mathbf{W}$ & 0.798 & 0.45 & 0.246 & 0.131 & 6.8e-2 & 0.83 & 0.87 & 0.91 & 0.95\\ 
    \hline
    \multicolumn{10}{c}{BiCGStab, Ducros, 2$^{nd}$ order }
    \\\hline
     $p$ & 0.173 & 5.1e-2 & 1.3e-2 & 3.3e-3 & 8.2e-4 & 1.77 & 1.96 & 1.99 & 2.0\\ 
     \rule[-0.3cm]{0mm}{0cm}
     $\mathbf{W}$ & 4.7e-2 & 9.3e-3 & 2.1e-3 & 5.2e-4 & 1.2e-4 & 2.32 & 2.13 & 2.03 & 2.17\\ 
    \hline
    \multicolumn{10}{c}{BiCGStab, Rusanov, 1$^{st}$ order }
    \\\hline
     $p$ & 1.121 & 0.754 & 0.394 & 0.195 & 9.6e-2 & 0.57 & 0.93 & 1.01 & 1.02\\ 
     \rule[-0.3cm]{0mm}{0cm}
     $\mathbf{W}$ & 1.124 & 0.646 & 0.356 & 0.19 & 9.9e-2 & 0.8 & 0.86 & 0.9 & 0.94\\ 
    \hline
    \multicolumn{10}{c}{BiCGStab, Rusanov, 2$^{nd}$ order }
    \\\hline
    $p$ & 0.145 & 4.7e-2 & 1.3e-2 & 3.2e-3 & 8.2e-4 & 1.63 & 1.89 & 1.96 & 1.98\\ 
     $\mathbf{W}$ & 4.5e-2 & 9.3e-3 & 2.2e-3 & 5.2e-4 & 1.2e-4 & 2.28 & 2.11 & 2.04 & 2.16\\ 
    \hline
    \end{tabular}
    \caption{Taylor-Green vortex. Observed $L_2$ errors in space and time, $E_{M_i}$, and convergence rates, $o_{M_i/M_{i+1}}$.}
    \label{tab:TGV2}
\end{table}
To analyze also the performance of the different algorithms all simulations of this test case have been carried out in serial on one single CPU core of an AMD Ryzen Threadripper 3990X workstation with 64 cores and 128 GB of RAM.  
Table~\ref{tab:TGV3} reports the total computational time, the computational time per dual element and time step, namely
\begin{equation}
	t_e=\frac{\text{CPU time}}{\text{N. dual elements}\cdot \text{N. time steps}},
\end{equation}
and the number of time steps employed. 
As expected, using the Ducros flux function, and thus avoiding Newton iterations, is less time consuming than employing the implicit Rusanov flux, which requires the Newton algorithm for the linearization of the convective terms. This difference can be observed better for the first order scheme due to the lower accuracy of the method which forces the Krylov an Newton methods to perform more iterations to attain the stop criteria tolerance than when using the more accurate second order scheme. This fact also justifies the smaller CPU time of the second order in space approach with respect to the first order scheme for fine grids. For instance, for mesh $M_5$, when running the Ducros flux function in the BiCGStab algorithm we need $2$ iterations to reach the tolerance in each time step while for the first order scheme we perform $4$ iterations of the Krylov algorithm per time step. \textcolor{black}{Besides, the smaller CPU time required for the BiCGStab-Rusanov algorithm for the $M5$ grid with respect to $M4$ is due to the smaller error committed by the Krylov solver for $M5$, which avoids the necessity of Newton iterations. Meanwhile, for $M4$ or coarser grids, at least two iterations of the Newton loop are needed to reach the prescribed tolerance.} 
\begin{table}
    \centering
    \hspace*{-0.5cm}\begin{tabular}{cccccc}
    \hline
     & $M_1$ & $M_2$ & $M_3$ & $M_4$ & $M_5$ \\
    \hline
    \multicolumn{6}{c}{GMRES, Ducros, 1$^{st}$ order }
    \\\hline
    CPU time [s] & 0.2085 & 1.8085 & 18.6728 & 234.1114 & 2818.3383\\ 
    $t_e$ [ms] & 0.013 & 0.0144 & 0.0188 & 0.0296 & 0.0447\\ 
    Time steps & 20 & 40 & 80 & 160 & 320\\ 
    \hline
    \multicolumn{6}{c}{GMRES, Ducros, 2$^{nd}$ order }
    \\\hline
    CPU time [s] & 0.2267 & 1.7674 & 17.142 & 196.3942 & 212.272\\ 
    $t_e$ [ms] & 0.0142 & 0.0141 & 0.0173 & 0.0248 & 0.0034\\ 
    Time steps & 20 & 40 & 80 & 160 & 320\\ 
    \hline
    \multicolumn{6}{c}{BiCGStab, Ducros, 1$^{st}$ order }
    \\\hline
    CPU time [s] & 3.6388 & 1.7925 & 18.7976 & 221.8836 & 2763.4684\\ 
    $t_e$ [ms] & 0.2274 & 0.0143 & 0.0189 & 0.0281 & 0.0438\\ 
    Time steps & 20 & 40 & 80 & 160 & 320\\ 
    \hline
    \multicolumn{6}{c}{BiCGStab, Ducros, 2$^{nd}$ order }
    \\\hline
    CPU time [s] & 4.6886 & 1.7353 & 17.5796 & 183.9679 & 205.1983\\ 
    $t_e$ [ms] & 0.293 & 0.0138 & 0.0177 & 0.0233 & 0.0033\\ 
    Time steps & 20 & 40 & 80 & 160 & 320\\ 
    \hline
    \multicolumn{6}{c}{BiCGStab, Rusanov, 1$^{st}$ order }
    \\\hline
    CPU time [s] & 0.9933 & 7.3803 & 48.1604 & 466.2547 & 4647.945\\ 
    $t_e$ [ms] & 0.0621 & 0.0588 & 0.0485 & 0.059 & 0.0737\\ 
    Time steps & 20 & 40 & 80 & 160 & 320\\ 
     \hline
    \multicolumn{6}{c}{BiCGStab, Rusanov, 2$^{nd}$ order }
    \\\hline
    CPU time [s] & 45.2971 & 6.0887 & 26.9274 & 260.3956 & 220.1567\\ 
    $t_e$ [ms] & 2.8311 & 0.0485 & 0.0271 & 0.0329 & 0.0035\\ 
    Time steps & 20 & 40 & 80 & 160 & 320\\ 
    \hline
    \end{tabular}
    \caption{Taylor-Green vortex. CPU time, CPU time per element, $t_e$, and number of time steps.}
    \label{tab:TGV3}
\end{table}

\subsection{3D Arnold-Beltrami-Childress flow}
To assess the accuracy also for the three dimensional case, we consider the Arnold-Beltrami-Childress (ABC) flow originally introduced in \cite{Arnold1965,Childress} and also studied in \cite{TD16}. The exact solution for the inviscid incompressible Navier-Stokes equations in a periodic domain reads
\begin{gather*}
	u_1(x_1,x_2,x_3,t)=(\sin(x_3)+\cos(x_2)),\\
	u_2(x_1,x_2,x_3,t)=(\sin(x_1)+\cos(x_3)),\\
	u_3(x_1,x_2,x_3,t)=(\sin(x_2)+\cos(x_1)), 
\end{gather*}
\begin{equation}
	p(x_1,x_2,x_3,t)=-(\cos(x_1)\sin(x_2)+\sin(x_1)\cos(x_3)+\sin(x_3)\cos(x_2))+c, 
	\label{eq:ABCtest}
\end{equation}
with $c\in \mathbb{R}$. The convergence study is performed using the implicit hybrid method with a preconditioned Newton-BiCGStab scheme and a Ducros semi-implicit scheme of both first and second order. An artificial viscosity coefficient of $c_{\alpha}=1.0$ is set for the second order in space scheme. The computational domain is given by $\Omega=[-\pi,\pi]^3$ and a sequence of four successively refined meshes is used. The main features of those meshes are reported in Table \ref{tab:ABC1}. 
As initial condition we impose the exact solution and the dynamics evolves up to $t=1$ with a time step set to $\Delta t=0.05$ for the coarser grid and then properly scaled according to the mesh refinement. Since we impose periodic boundary conditions on all sides, we have a set of solutions for the pressure differing by a constant $c$. In order to test also the convergence rate of the pressure, we set the constant $c$ in \eqref{eq:ABCtest} a posteriori equal to the mean value of the resulting numerical pressure, as done in \cite{TD16}.
The obtained pressure and the velocity streamlines for mesh $M_3$ are depicted in Figure~\ref{fig:ABC} for a qualitative comparison with available reference data. 
Meanwhile the resulting $L_2$ error norms and convergence rates computed according to \eqref{eq:errors} are reported in Table~\ref{tab:ABC2}. We observe that the first order scheme is slightly below the expected accuracy especially for the pressure variable, while the second order scheme achieves the sought order of accuracy. 
\begin{table}[H]
	\centering
	\begin{tabular}{lcc}
		\hline
		Mesh & Primal elements & Vertices  \\
		\hline
		$M_1$ & 8640 & 2197 \\
		$M_2$ & 69120 & 15625 \\
		$M_3$ & 552960 & 117649 \\
		$M_4$ & 4423680 & 912673 \\
		\hline
	\end{tabular}
	\caption{ABC Test. Mesh features.}
	\label{tab:ABC1}
\end{table}
\begin{table}
	\centering
	\begin{tabular}{cccccccc}
		\hline\rule[-0.3cm]{0mm}{0cm}
		& $E_{M_1}$ & $E_{M_2}$ & $E_{M_3}$ & $E_{M_4}$ & $o_{M_1/M_2}$ & $o_{M_2/M_3}$ & $o_{M_3/M_4}$\\
		\hline
		\multicolumn{8}{c}{BiCGStab, Ducros, 1$^{st}$ order }
		\\\hline
		$p$ & 2.9 & 1.739 & 0.939 & 0.522 & 0.74 & 0.89 & 0.85\\ 
		\rule[-0.3cm]{0mm}{0cm}
		$\mathbf{W}$ & 4.346 & 2.265 & 1.176 & 0.603 & 0.94 & 0.95 & 0.96\\ 
		\hline
		\multicolumn{8}{c}{BiCGStab, Ducros, 2$^{nd}$ order }
		\\\hline
		$p$ & 0.79 & 0.227 & 5.1e-2 & 1.2e-2 & 1.8 & 2.15 & 2.05\\ 
		$\mathbf{W}$ & 0.807 & 0.198 & 4.6e-2 & 7.3e-3 & 2.02 & 2.11 & 2.67\\ 
		\hline
	\end{tabular}
	\caption{ABC Test. Observed $L_2$ errors in space and time, $E_{M_i}$, and convergence rates, $o_{M_i/M_{i+1}}$.}
	\label{tab:ABC2}
\end{table}

In addition, Table~\ref{tab:ABC3} reports the computational time, the computational time per dual element and the number of time steps needed to reach the final time. As in the convergence study for the Taylor-Green vortex in 2D, Section~\ref{sec:numericalresults_TGV}, for fine grids we observe a smaller CPU time of the second order in space approach with respect to the first order scheme, justified by the fact that the use of a higher accuracy method requires less iterations to attain the tolerance of the stop criteria. Again, for mesh $M_4$, when running the second order scheme in the BiCGStab algorithm we need 2 iterations to reach the tolerance in each time step while for the first order scheme we perform 4 iterations of the Krylov algorithm per time step.

\begin{table}
	\centering
	\begin{tabular}{ccccc}
		\hline 
		& $M_1$ & $M_2$ & $M_3$ & $M_4$ \\
		\hline
		\multicolumn{5}{c}{\small{BiCGStab, Ducros, 1$^{st}$ order} }
		\\\hline
		 CPU time [s] & 3.187 & 15.0431 & 199.1015 & 2855.7914\\ 
		$t_e$ [ms] & 0.0615 & 0.1451 & 0.1309 & 0.123\\ 
		\rule[-0.3cm]{0mm}{0cm}
		Time steps & 20 & 40 & 80 & 160\\ 
		\hline
		\multicolumn{5}{c}{\small{BiCGStab, Ducros, 2$^{nd}$ order} }
		\\\hline
		 CPU time [s] & 8.3805 & 38.4215 & 233.7558 & 681.022\\ 
		$t_e$ [ms] & 0.1617 & 0.3706 & 0.1537 & 0.0293\\ 
		Time steps & 20 & 40 & 80 & 160\\ 
		\hline
	\end{tabular}
	\caption{ABC Test. CPU time, CPU time per element, $t_e$, and number of time steps.}
	\label{tab:ABC3}
\end{table}
\begin{figure}
	\centering
	\includegraphics[trim =10 10 10 10,clip,width=0.45\linewidth]{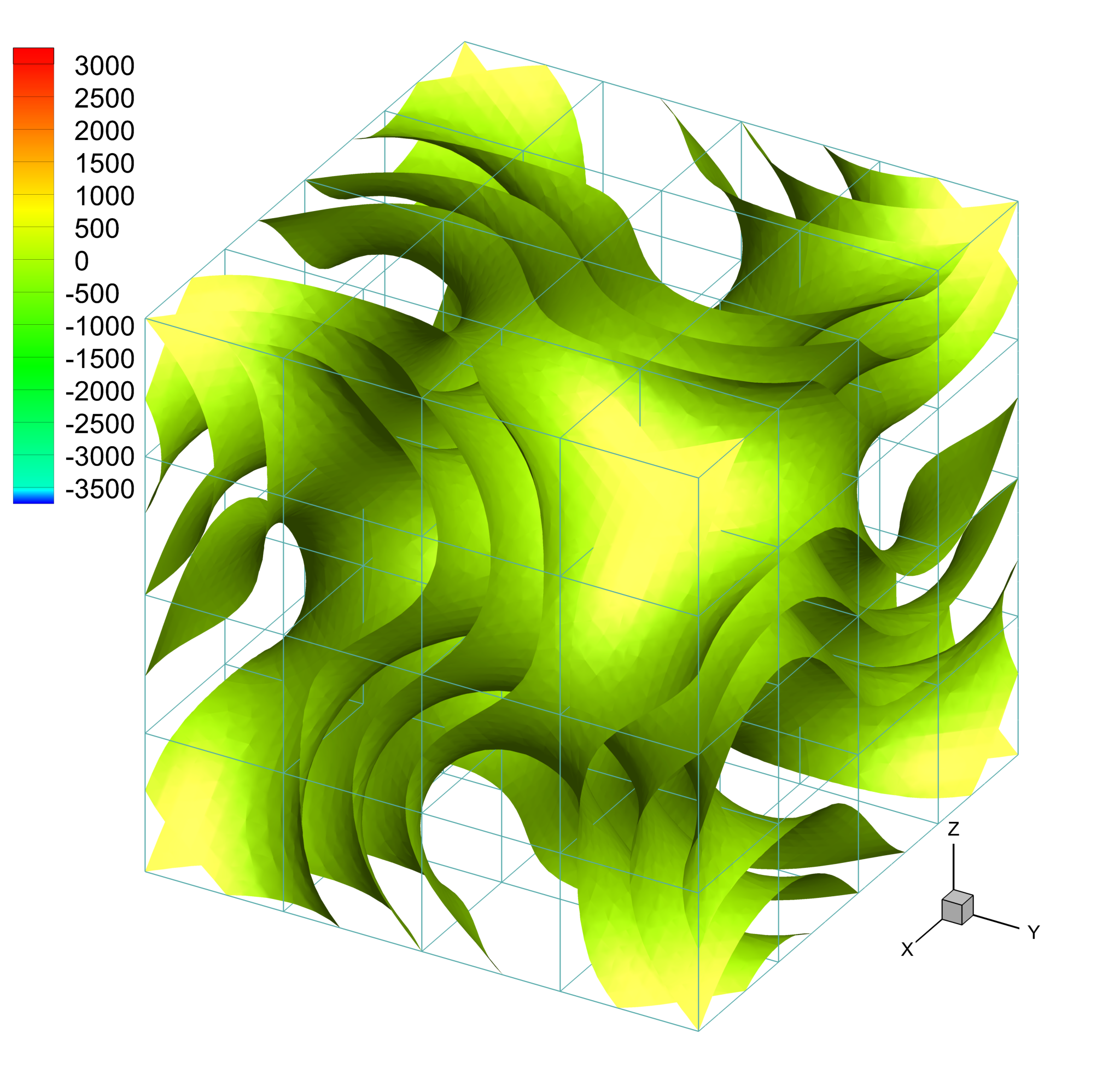}\quad\includegraphics[trim =10 10 10 10,clip,width=0.45\linewidth]{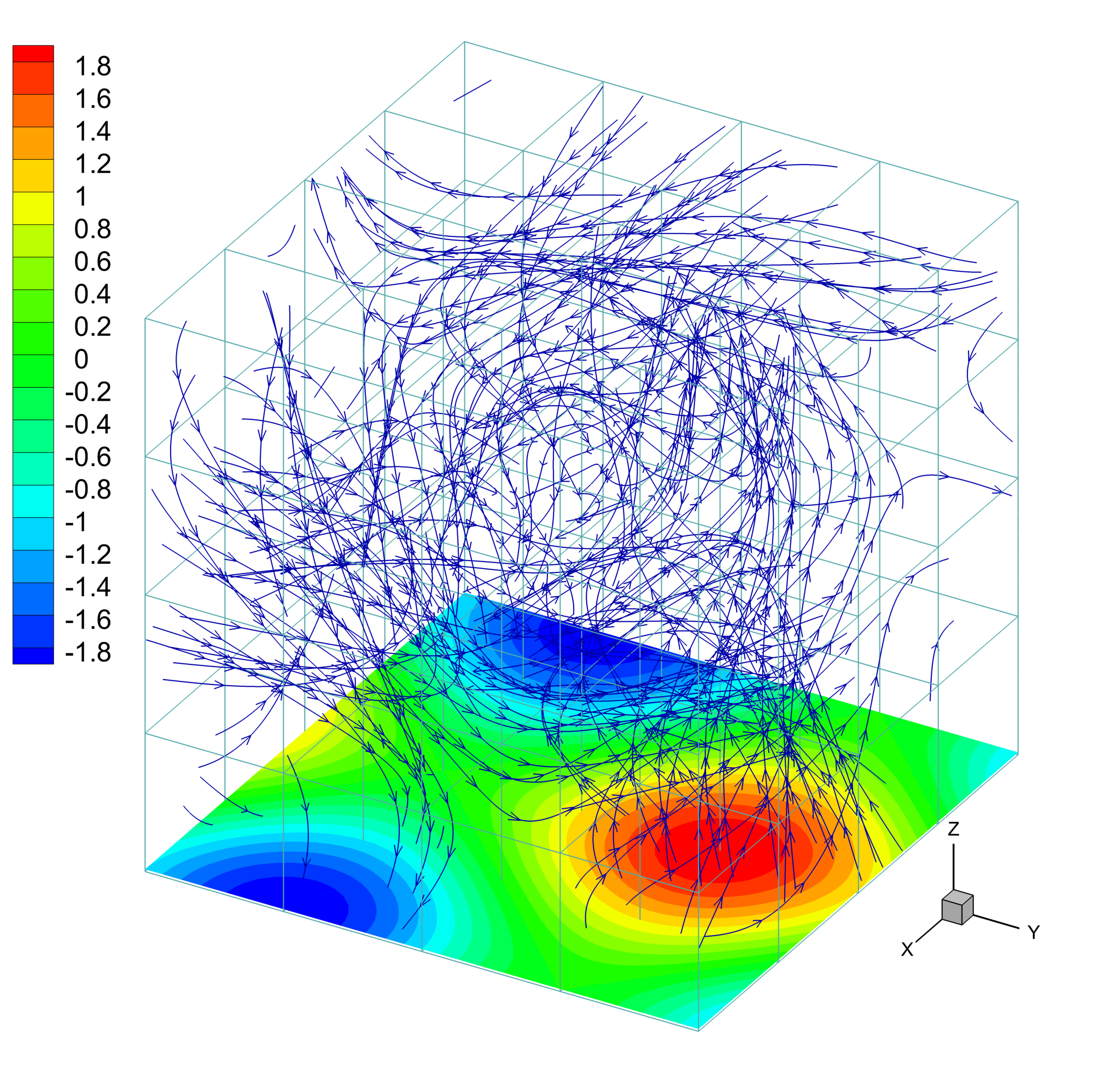}
	\caption{ABC Test. Pressure isosurfaces at levels $p=[-0.8,0.0,0.8]$ (left) and velocity streamlines (right).}
	\label{fig:ABC}
\end{figure}

\subsection{The first problem of Stokes}
We now simulate the first problem of Stokes, \cite{SG16}, which is one of the rare unsteady problems for which an exact solution of the incompressible Navier-Stokes equations is known.
We consider the computational domain $\Omega=[-0.5, 0.5]\times[-0.05,0.05]$, a mesh composed of $128000$ primal elements and the following initial condition:
\begin{equation*} 
	p(x_1,x_2,0)=0, \quad u_1(x_1,x_2,0)= 0, \quad u_2(x_1,x_2,0)=
	\begin{cases}
		\phantom{+}0.1, &\text{if }x_1>0,\\
		-0.1, &\text{if }x_1<0.
	\end{cases}
\end{equation*}
To reproduce the classical 1D test case, periodic boundary conditions are set along \mbox{$x_2$-direction,} whereas the initial condition is imposed on the left and right boundaries of the domain.
The fluid density is $\rho=1$, while a sequence of viscosity values are considered $\mu \in \{10^{-2},10^{-3},10^{-4}\}$. The simulations are run using the BiCGStab-Newton algorithm with a second order semi-implicit Ducros numerical flux function. The fixed time step is $\Delta t=0.01$ and we set a final time of $t=1.0$.
The numerical results for the velocity component $u_2$ are validated against the exact solution of the Navier-Stokes equations which is given by

\begin{equation*}
	u_2(x_1,x_2,t)=\frac{1}{10}\mathrm{erf}\left(\frac{x_1}{2\sqrt{\mu t}}\right).
\end{equation*}
Figure \ref{fig:FS2D} presents the comparison between the reference solution and the numerically obtained one-dimensional cut along the $x_1$-direction at $y=0$. An excellent agreement can be observed for all viscosities considered.
\begin{figure}
	\centering
	\includegraphics[trim =10 10 10 10,clip,width=0.31\linewidth]{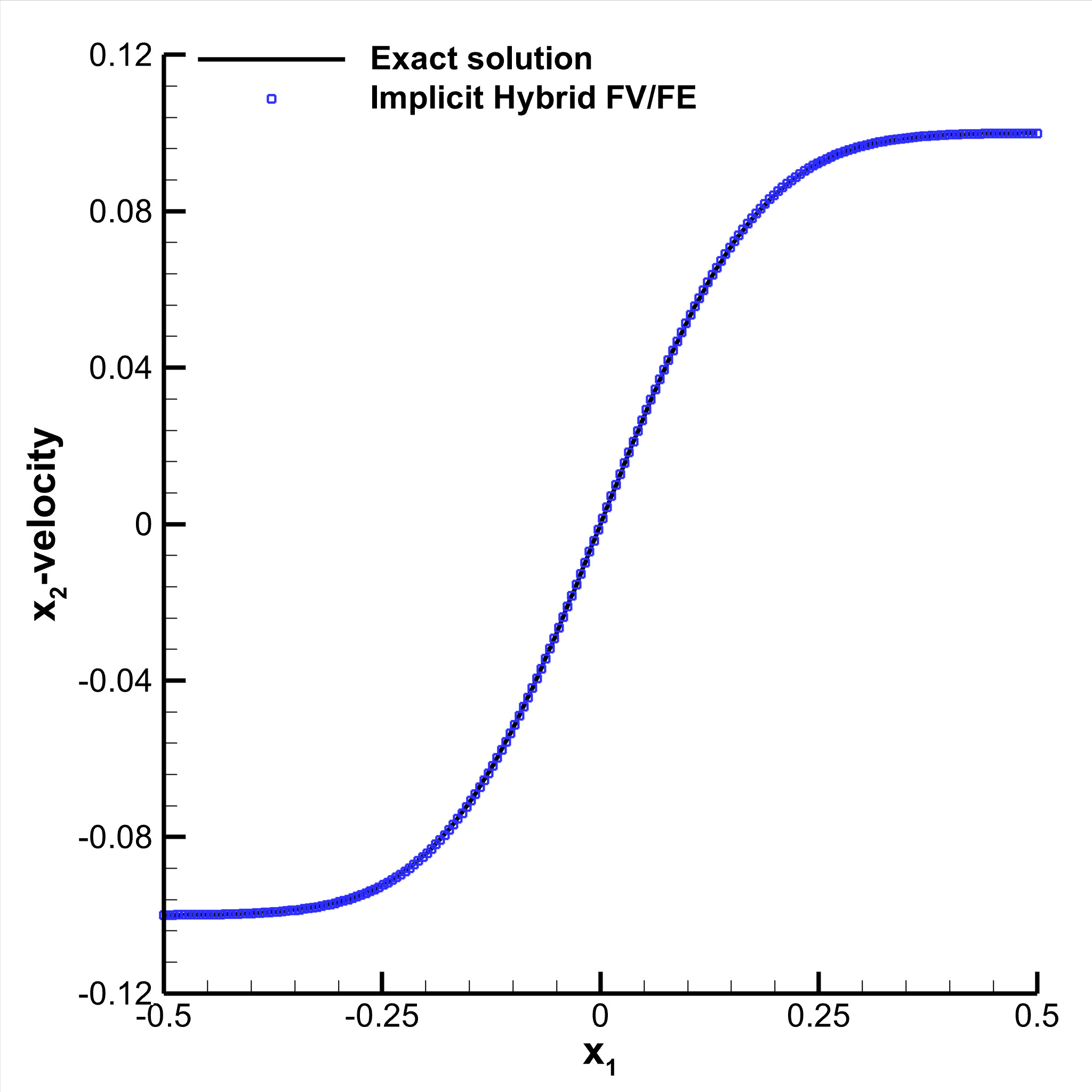}\quad
	\includegraphics[trim =10 10 10 10,clip,width=0.31\linewidth]{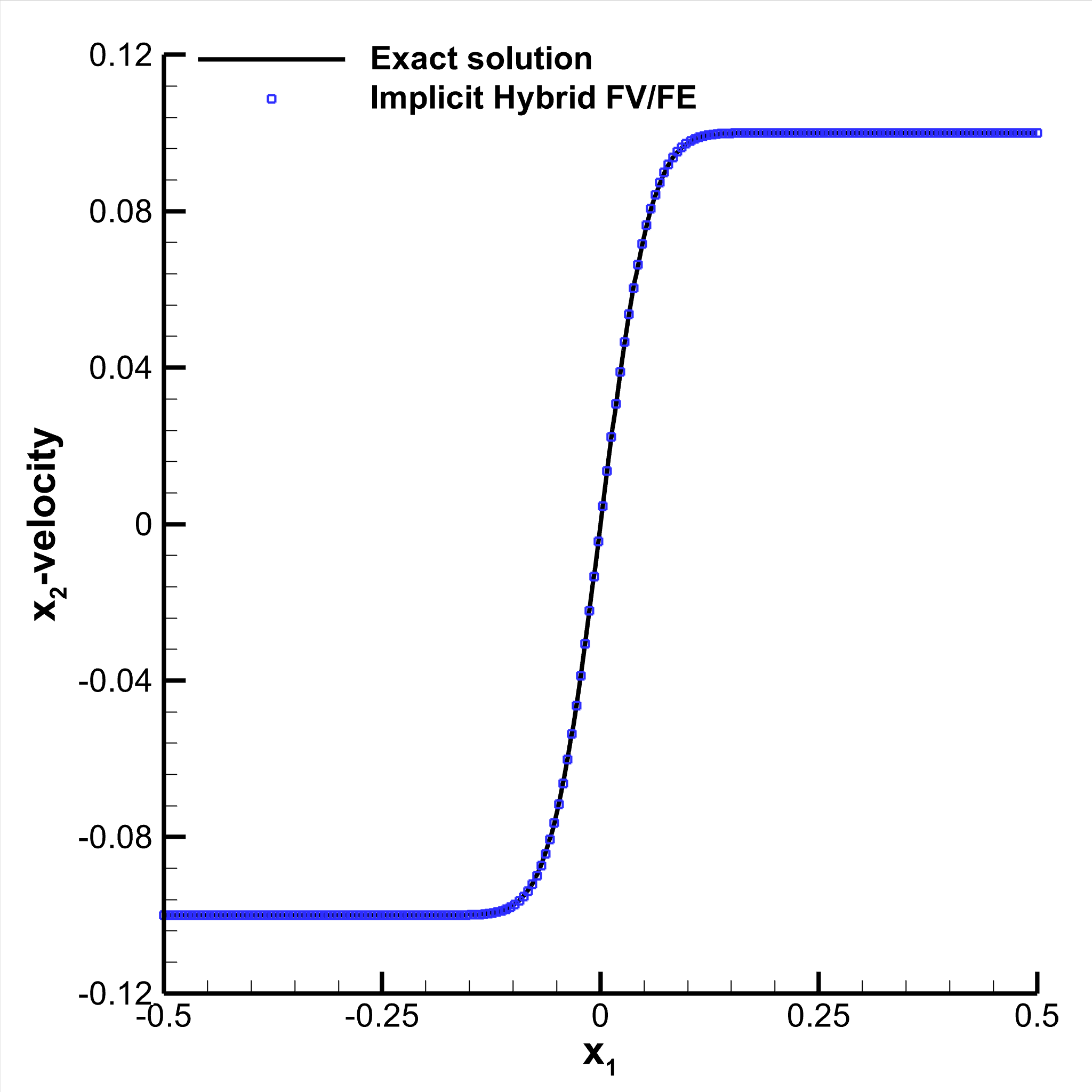}\quad
	\includegraphics[trim =10 10 10 10,clip,width=0.31\linewidth]{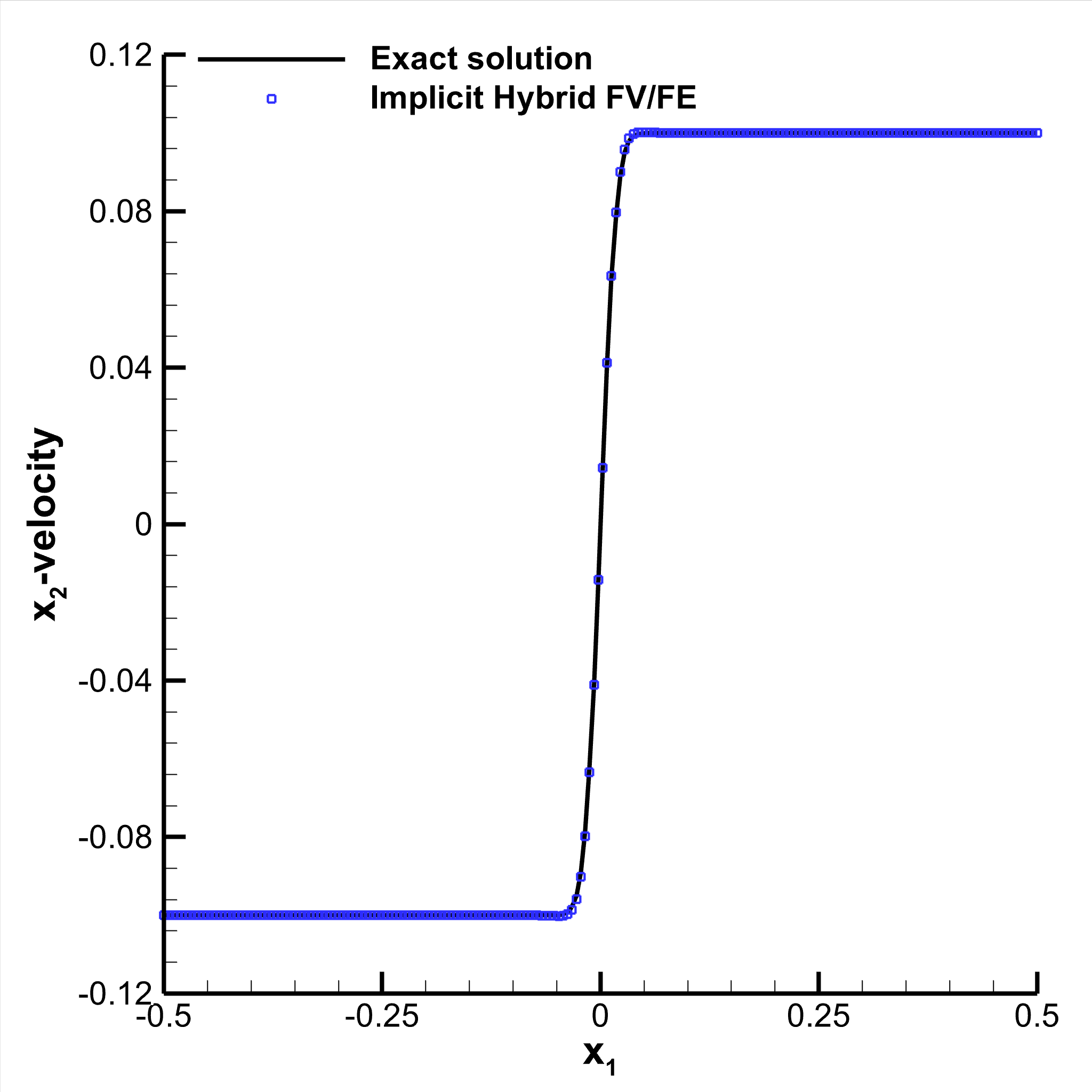}
	\caption{First problem of Stokes. Comparison between the exact solution and the numerical results computed with the implicit hybrid FV/FE scheme on the one-dimensional cut along the $x_1$-direction at $y=0$ at the final time $t_{end}=1.0$. The viscosity values considered are: $\mu=10^{-2}$ (left), $\mu=10^{-3}$ (middle), $\mu=10^{-4}$ (right).}
	\label{fig:FS2D}
\end{figure}

\subsection{Lid-driven cavity}
The lid-driven cavity problem has has widely been used as a validation test for incompressible flow solvers, \cite{GGS82,AK05,TD17,HybridMPI}. We discretize the square domain $\Omega=[0,1]\times[0,1]$ with a triangular grid made of $2906$ primal elements. Dirichlet boundary conditions are imposed on all sides: no slip conditions are applied on lateral boundaries and at the bottom, while a fixed velocity field $\mathbf{u}=(1,0)$ is imposed at top boundary. 
Moreover, the fluid density is set to $\rho=1$ and the fluid viscosity is fixed to $\mu=10^{-2}$, resulting in a Reynolds number of $100$. To run the simulation, we consider an initial fluid at a rest with pressure $p=1$ and we let the time step vary at each time iteration according to the condition $CFL=100$:
\begin{equation}
	\Dt = \min_{C_{i}}\left\lbrace \Dt_{i}\right\rbrace, \qquad \Dt_{i} = \textnormal{CFL} \frac{r_{i}^2}{( \left|\zeta\right|_{\max} + c_\alpha) r_{i} 
	}
	\label{eqn.cfl.conditionNS}
\end{equation}
with $\left|\zeta\right|_{\max}$ the maximum absolute eigenvalue related to the convective terms and $r_i$ the incircle diameter of each dual control volume.
The preconditioner inside the Newton-BiCGStab algorithm makes use of a reordering of the elements in $x_1$-direction. 
The obtained velocity profiles along the vertical and horizontal lines passing through the geometric center of the cavity  are compared against the reference solution from Ghia et al. \cite{GGS82} in Figure~\ref{fig:LDC}. The contour plot of the velocity components together with the velocity vectors are also depicted. 
\begin{figure}
    \centering
    \includegraphics[trim =10 10 10 10,clip,width=0.45\linewidth]{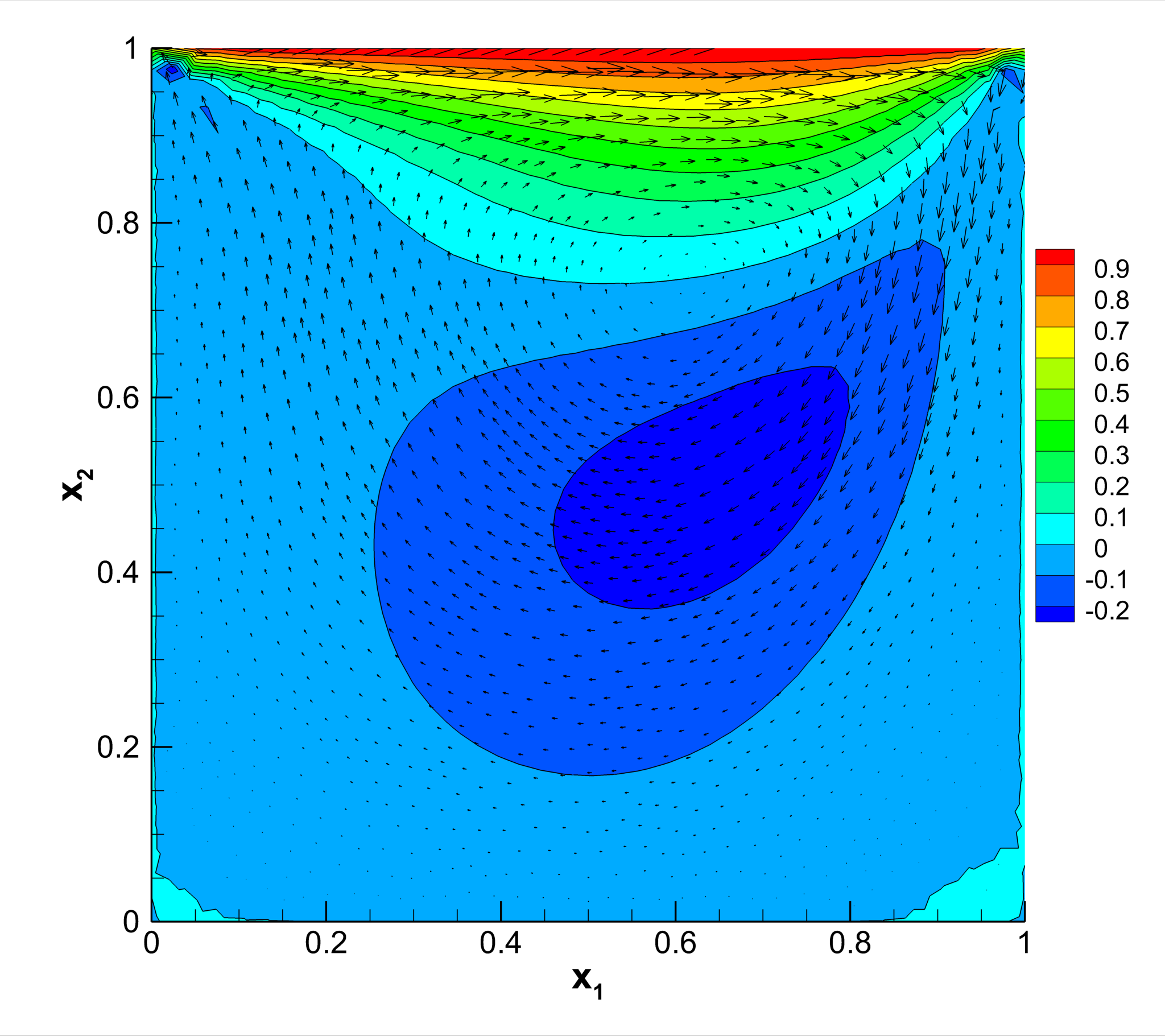}\quad\includegraphics[trim =10 10 10 10,clip,width=0.45\linewidth]{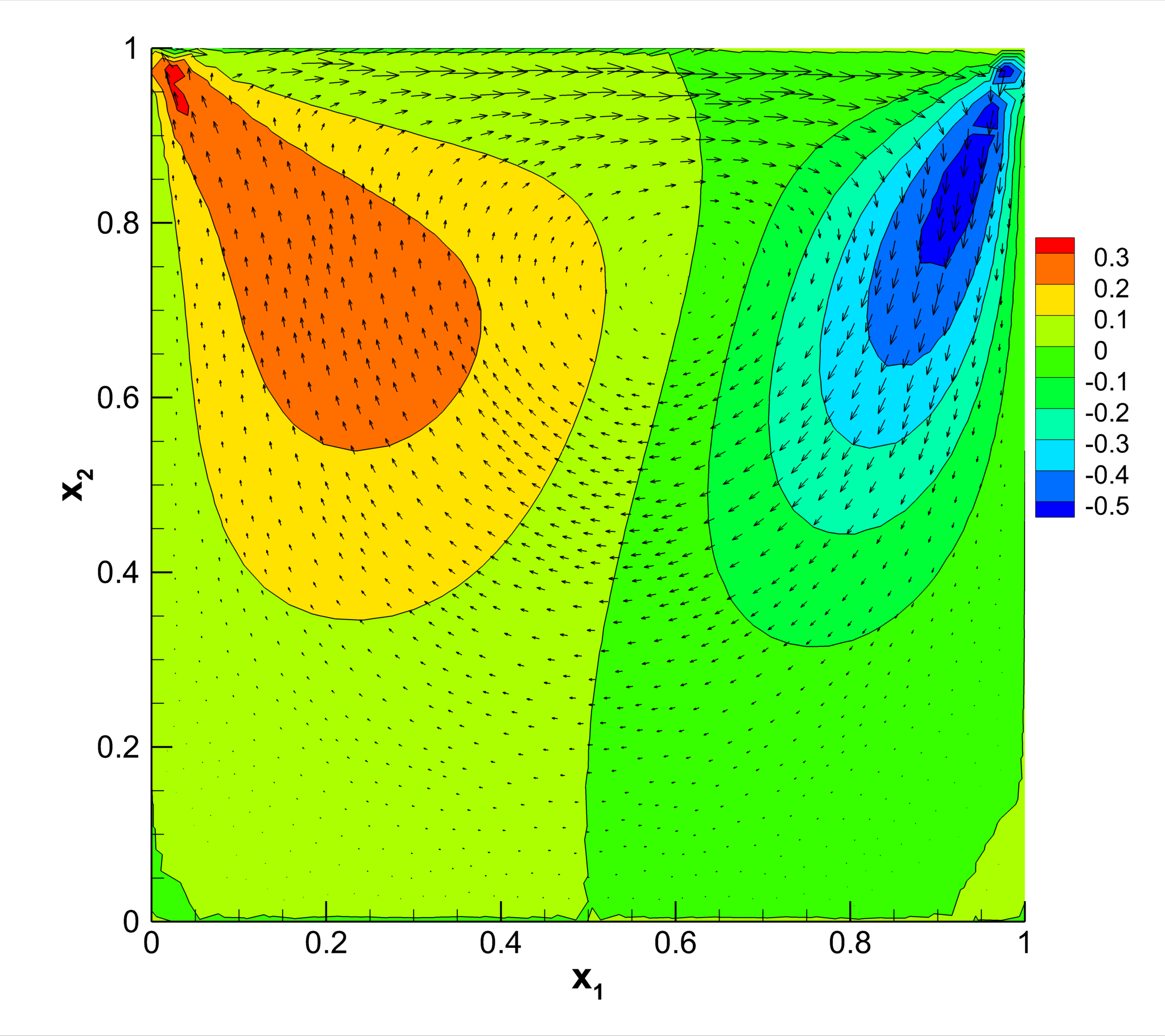}\\
    \includegraphics[trim =10 10 10 10,clip,width=0.7\linewidth]{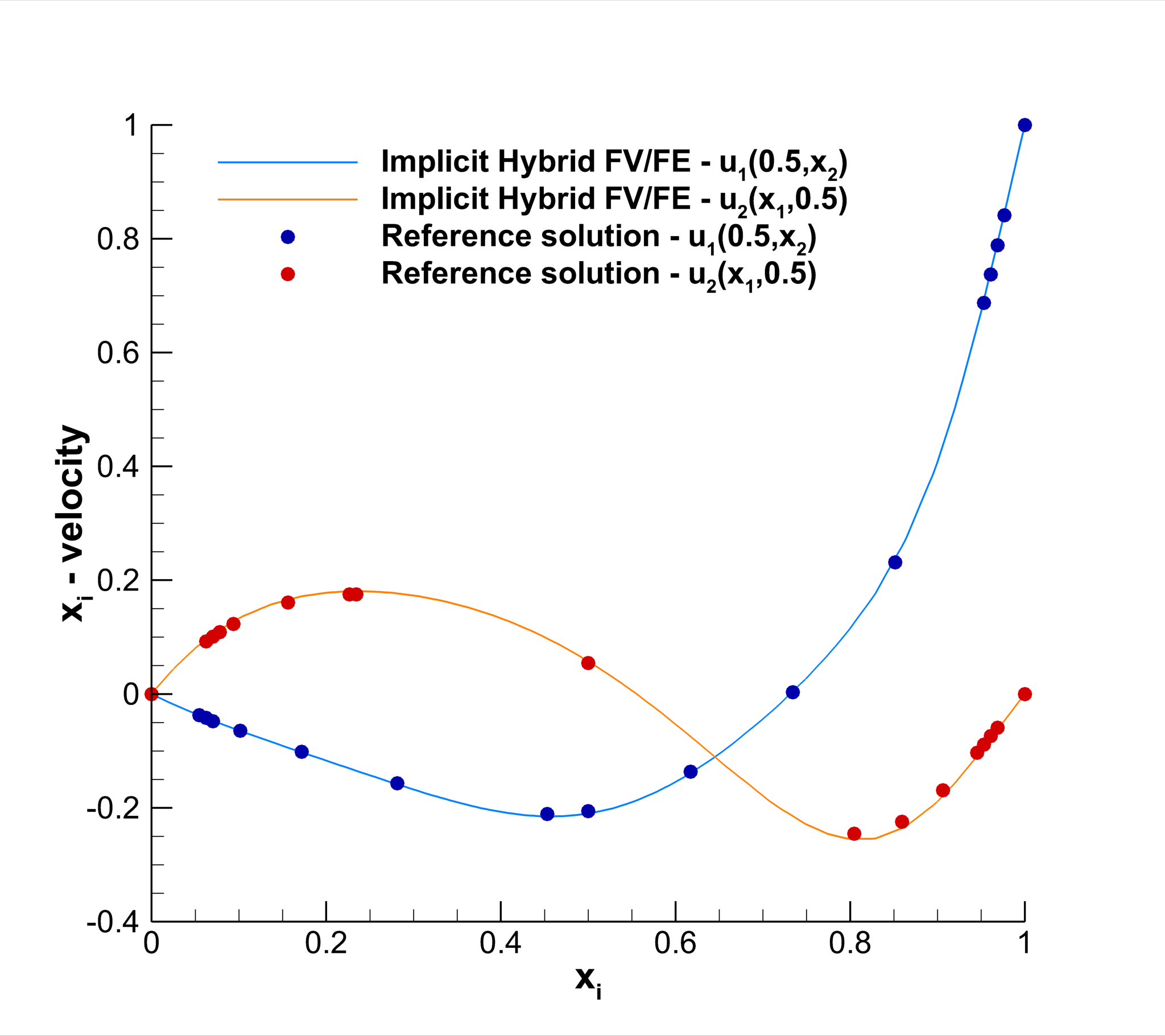}
    \caption{Lid driven cavity. Top: contour plot of the $x$-component of velocity vector, $u_1$, (left) and of $y$-component, $u_2$, (right) obtained using the implicit hybrid FV/FE scheme. Bottom: 1D cuts for $u_1$ and $u_2$ along the vertical and horizontal lines passing through the geometric center of the cavity . The reference solution for comparison is taken from \cite{GGS82}. }
    \label{fig:LDC}
\end{figure}

We now analyse the computational cost of the implicit scheme with respect to the semi-implicit hybrid FV/FE methodology proposed in \cite{HybridMPI,HybridNNT}. To solve the convective part of the semi-implicit scheme, the second order in space explicit approach with Rusanov flux function is used, while the counterpart of the fully implicit scheme employs the preconditioned Newton-BiCGStab method also with a second order in space Rusanov flux function. The computational time for both tests, run in parallel on $32$ CPUs of an AMD Ryzen Threadripper 3990X are reported in Table~\ref{tab:LDC1}. We can observe that for this test case the use of the new fully implicit scheme proposed in this paper reduces the computational cost of the simulation by a factor of around $110$ when using the Rusanov flux function and of about $145$ for the Ducros flux.
\begin{table}
    \centering
    \begin{tabular}{lcc}
    \hline
     Numerical Flux & Semi-implicit FV/FE scheme & New implicit FV/FE scheme \\
    \hline
    Rusanov & $4809.89$ s & $41.85$ s \\
    Ducros & $4829.65$ s & $32.89$ s\\
    \hline
    \end{tabular}
    \caption{Lid driven cavity. CPU times required by the novel fully implicit hybrid FV/FE scheme and the semi-implicit hybrid FV/FE method in \cite{HybridNNT}.}
    \label{tab:LDC1}
\end{table}

\subsection{Backward-facing step flow}
We study the fluid flow over a backward-facing step at different Reynolds numbers. The structure of the domain follows that reported in \cite{erturk2008}. 
We choose a step height of $h=0.097$. Then, the inlet boundary is located $20$ step heights upstream the step, while the outlet boundary is chosen $300$ step heights away from the step. The larger channel, downstream of the step, has a height of $H=0.2$ i.e. we consider a expansion ratio, the ratio of the channel height $H$ downstream of the step to the channel height $h_i$ upstream of the step, equal to $1.942$. 
At the inlet we impose the exact Poiseuille velocity profile, whereas a Neumann boundary condition is considered at the outlet. On all other boundaries, no-slip wall boundary conditions are imposed. For the current test the fluid density is set to $\rho=1$ and the kinematic viscosity $\nu$ is chosen in order to obtain the desired Reynolds number, \cite{armaly1983}, which is given by Re$=\frac{D U}{\nu}$, with $D=2h_i$ and $U$ the mean inlet velocity. The time step is set to $\Delta t= 0.01$ and the simulation is run with the implicit hybrid FV/FE scheme until steady state is reached. Figure~\ref{fig:BSF2} depicts the $x_1$-component of the resulting velocity field and the streamlines at Re$=400$. In addition, in Figure~\ref{fig:BSF2_recp} we plot the resulting recirculation point $X_1$, normalized by the step height $h$, for different Reynolds numbers. Herein, we compare our results against the experimental data reported in \cite{armaly1983}, the numerical results obtained in \cite{TD14} using a staggered semi-implicit discontinuous Galerkin scheme with polynomial approximation degree of order $N=3$, and the explicit second-order upwind finite difference scheme presented in \cite{erturk2008}.
\begin{figure}
	\centering
	\includegraphics[trim =10 10 100 100,clip,width=0.9\linewidth]{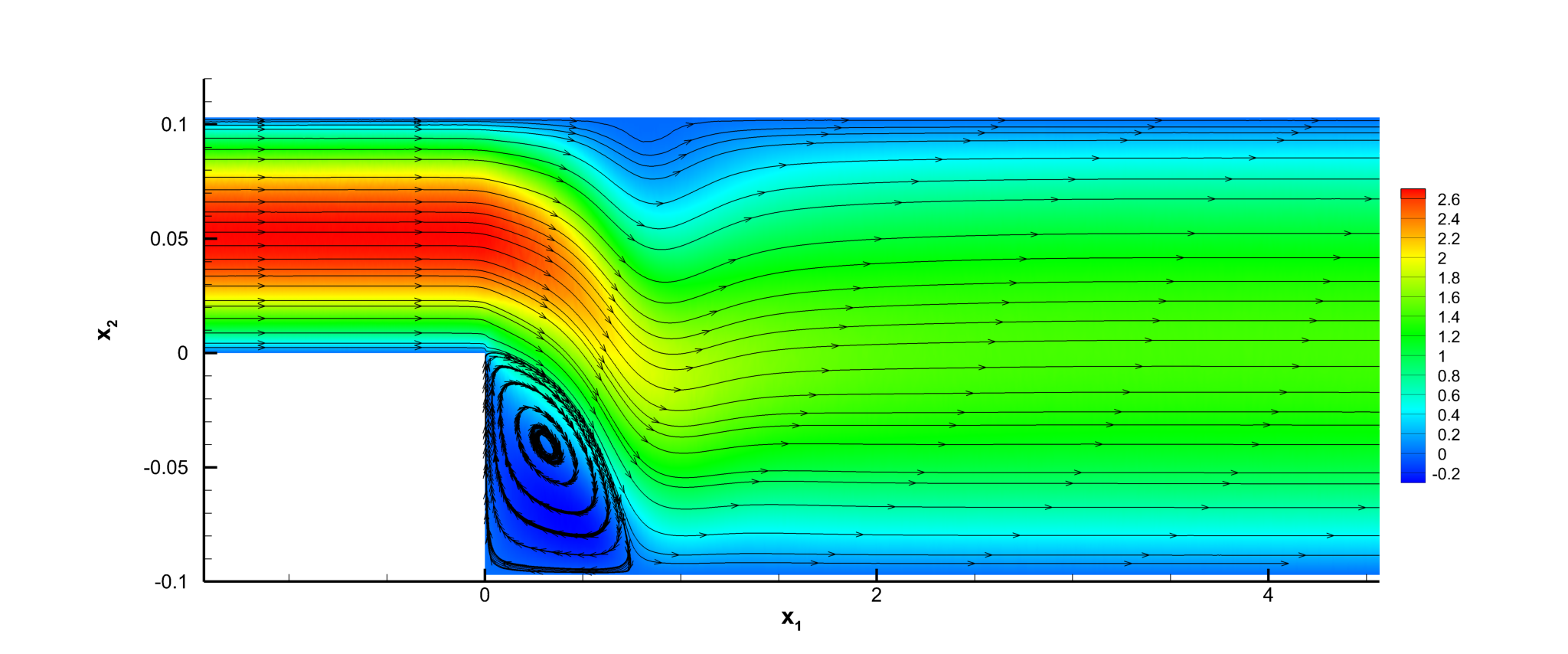}
	\caption{Backward-facing step flow. Contour plot of the velocity field obtained with the implicit hybrid FV/FE method at Reynolds $400$ and resulting streamlines near the recirculation zone downstream the step.}
	\label{fig:BSF2}
\end{figure}
\begin{figure}
	\centering
	\includegraphics[trim =10 10 10 10,clip,width=0.7\linewidth]{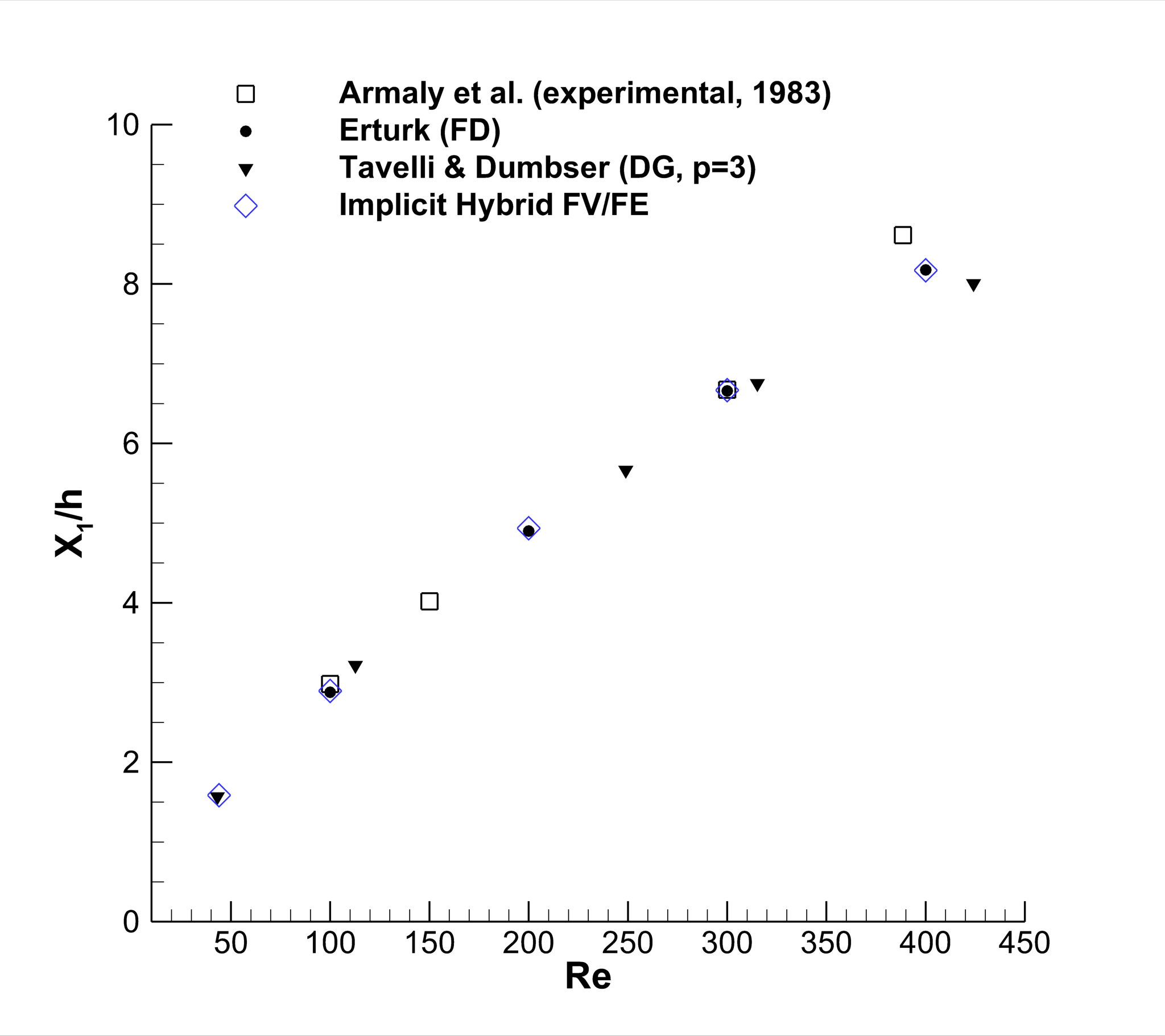}
	\caption{Backward-facing step flow. Normalized recirculation point $X_1/h$ for different Re numbers compared with the experimental data reported in \cite{armaly1983}, the numerical results obtained in \cite{erturk2008} and those presented in \cite{TD14}.}
	\label{fig:BSF2_recp}
\end{figure}

\subsection{Potential flow around a cylinder}
The sixth test case analysed here consists in the inviscid flow around a circular cylinder, \cite{TD14}. We consider the computational domain $\Omega=[-8,8]^{2}\setminus \left\lbrace \x\in\mathbb{R}^2 \, \mid \, \left\| \x \right\| < 1 \right\rbrace$ discretized with $30384$ primal triangular elements. The velocity field is strongly imposed at the left inflow boundary, making use of the known analytical potential flow solution:
\begin{equation*} 
	u_{r}     =  u_{m} \left(1-\frac{r_c^2}{r^2}\right)\cos(\phi), \qquad
	u_{\theta} = -u_{m} \left(1+\frac{r_c^2}{r^2}\right)\sin(\phi), \qquad 
	\tan(\phi)  = \dfrac{x_2}{x_1},
\end{equation*}
where $u_{r}$, $u_{\theta}$ denote the radial and angular velocities, respectively, $u_m=1$ is the mean flow velocity in the horizontal direction, $r_c=1$ is the radius of the cylinder and $r$ is the distance to the center of the cylinder of each spacial point $\x$. 
In the left boundary a pressure outlet boundary condition is defined using

\begin{equation*}
	\press(r,\phi,t) = 1 +  u_m^{2} \frac{r_c^2}{r^2} \left( \cos(2 \phi)-\halb\frac{r_c^2}{r^2}\right).
\end{equation*}
Finally at the top and bottom boundaries weak Dirichlet boundary conditions for the velocity field are employed.
As initial condition we impose the constant velocity field $\vel =(u_m,0) =(1,0)$ and $\press =1$. The results obtained with the second order implicit hybrid scheme with $c_{\alpha} = 0.5$ are reported in Figure \ref{fig:invcylpressure2d}. We can observe a that the obtained solution fits pretty well the known exact solution along the circumference of radius $r=1.01$ centred at the origin.
\begin{figure}
	\centering
	\includegraphics[width=0.48\linewidth]{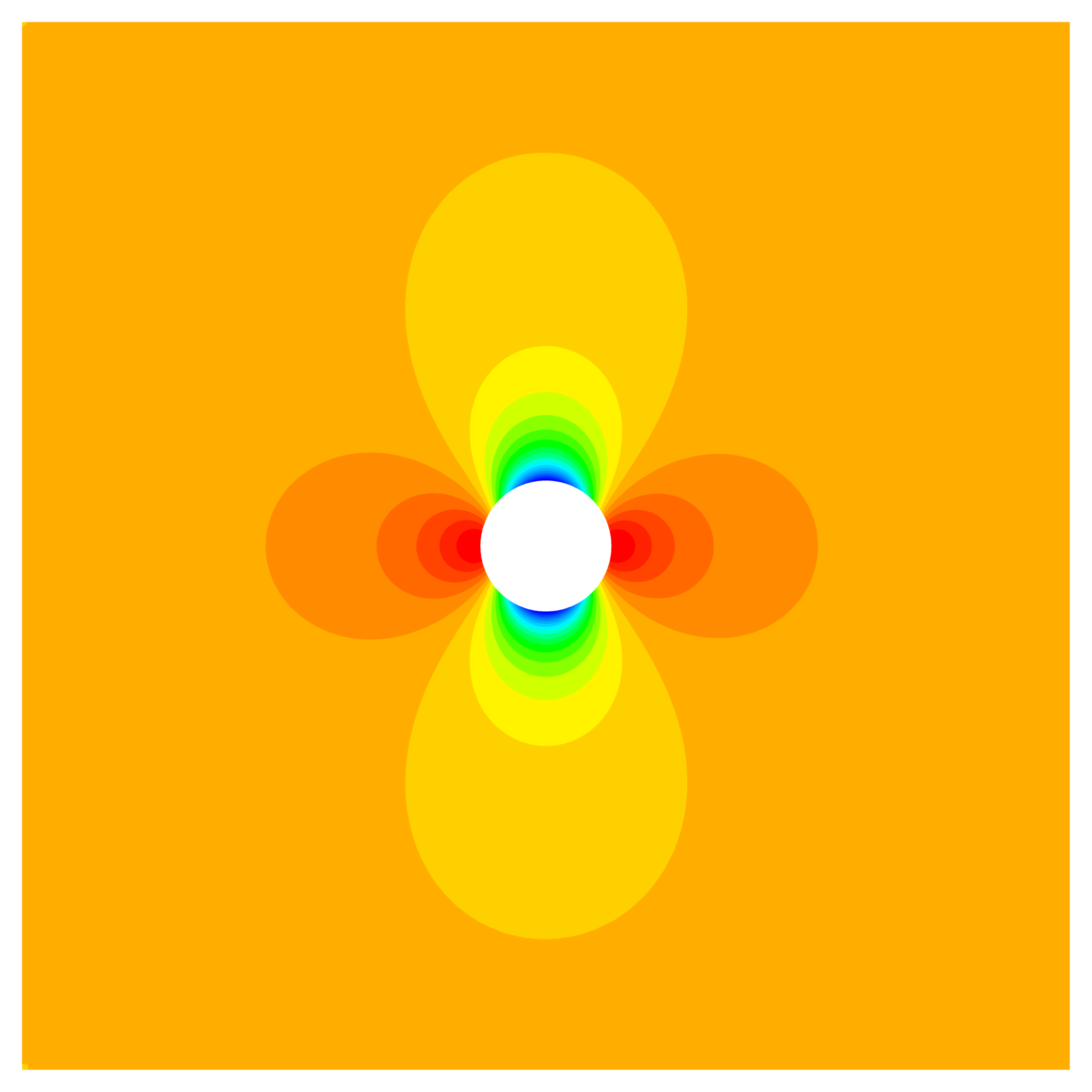}\hfill
	\includegraphics[width=0.48\linewidth]{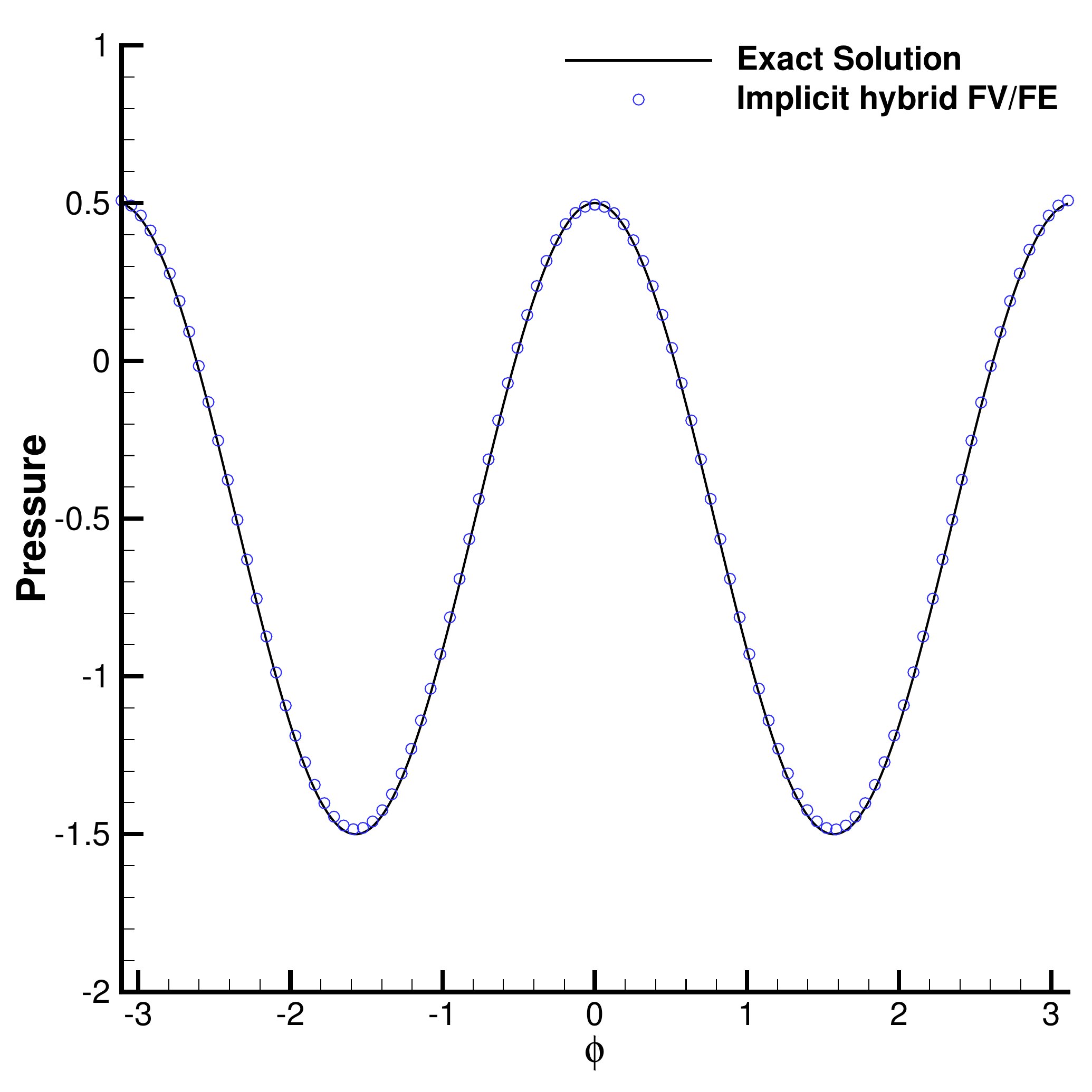}\\
	\includegraphics[width=0.48\linewidth]{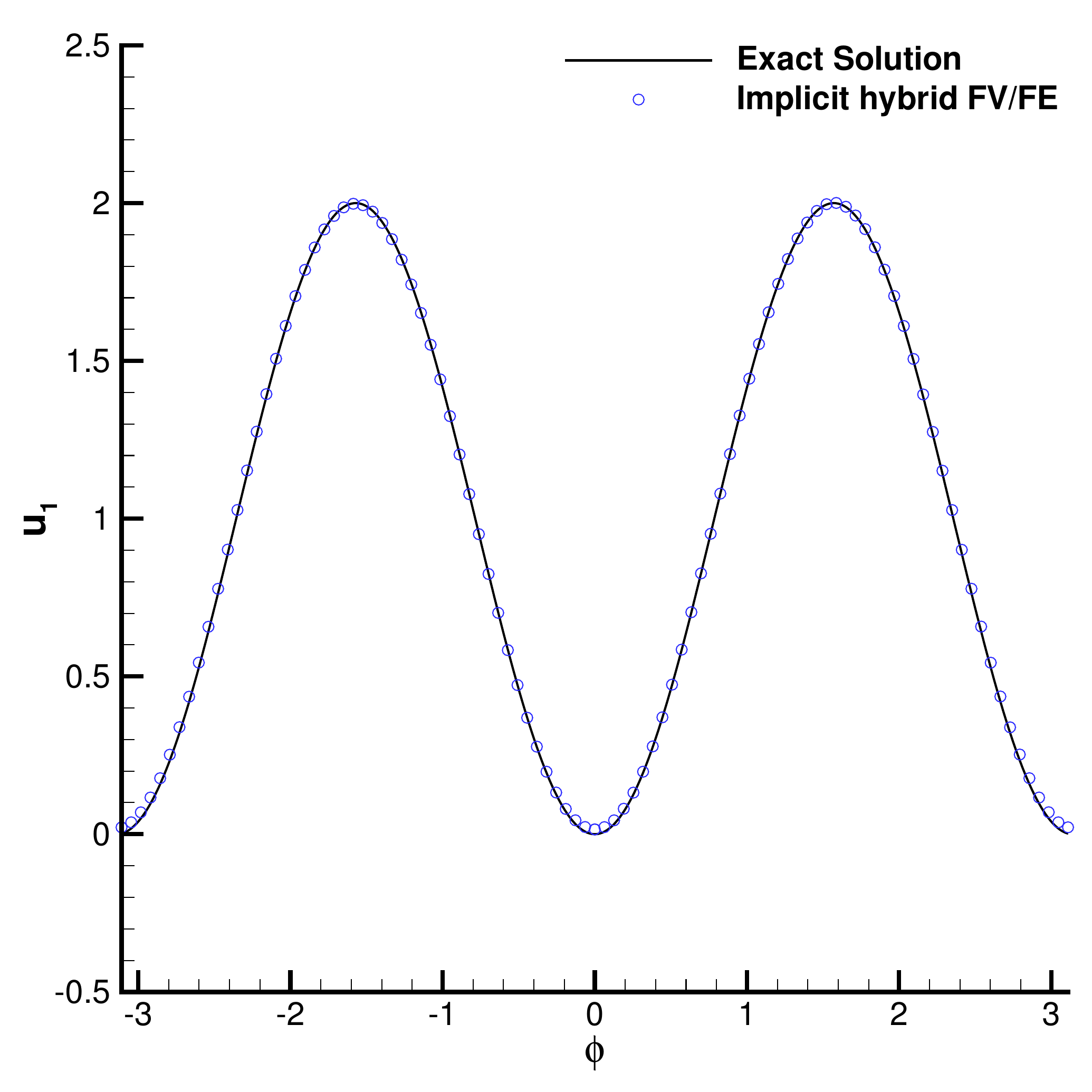}\hfill
	\includegraphics[width=0.48\linewidth]{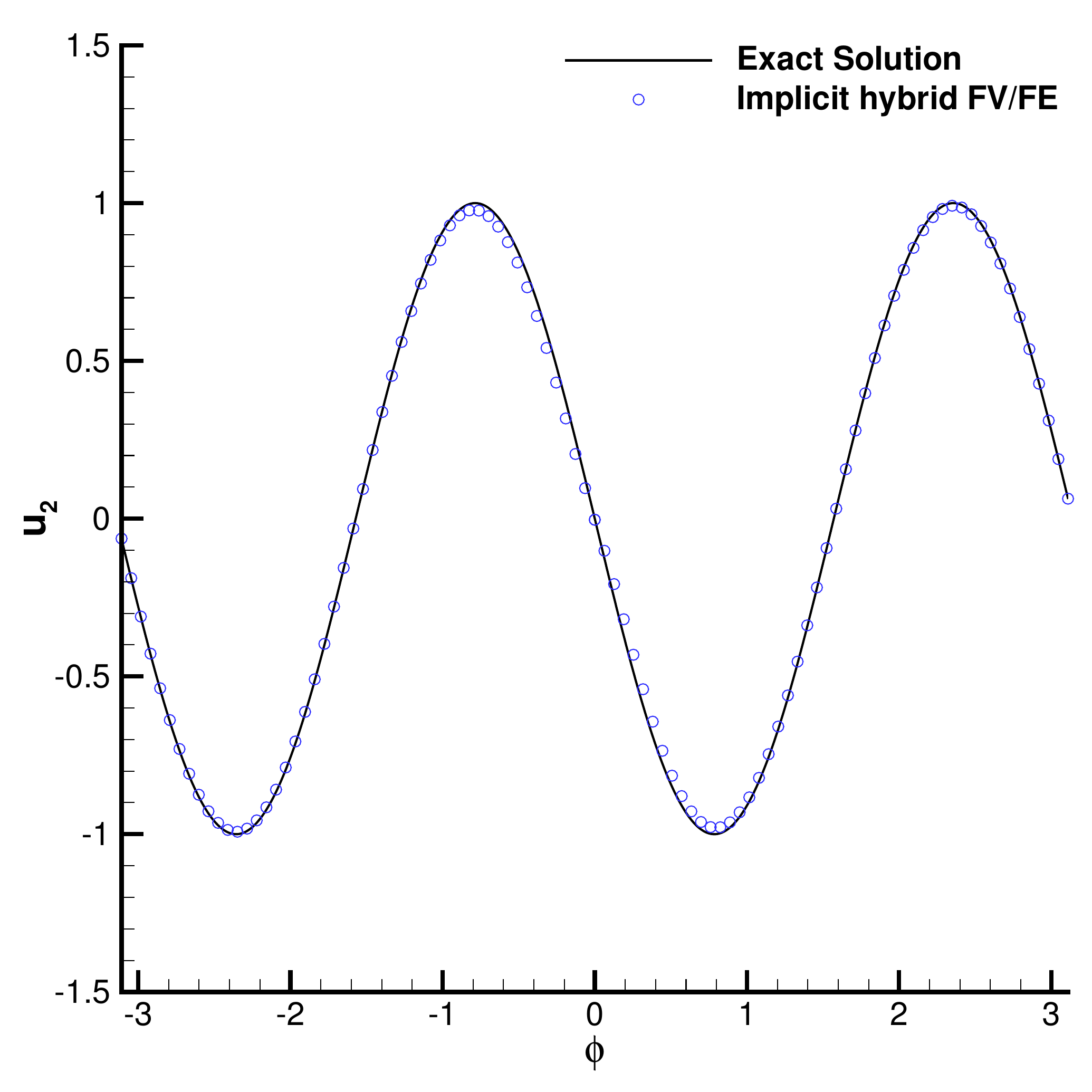}
	\caption{Inviscid flow passing around a cylinder. The top left figure depicts the contour plot of the pressure field obtained using the implicit hybrid FV/FE scheme at time $t=1$. The 1D cuts of the pressure, horizontal and vertical velocities (from  right top to left bottom, blue circles) along the cylinder of radius $r=1.01$ centred at the origin are reported together with the exact known solution (black solid line).}
	\label{fig:invcylpressure2d}
\end{figure}

\subsection{Viscous flow around a cylinder}
We now analyse the case of a viscous fluid flowing around a cylinder, \cite{WillBrown98,HybridALE}. The geometry considered consists in a 2D channel of length $50$ and height $20$ pierced by an embedded cylinder of radius $R=0.5$. At the inlet we set a constant velocity profile of value $\vel=(1,0)$, used also as initial condition The pressure $p=1$ is imposed at the outlet, while the cylinder boundary is assumed to be a viscous no-slip wall. We assume the top and bottom boundaries to be far enough away from the cylinder so that the velocity at infinity is recovered. Hence, it can be weakly imposed as a velocity boundary condition. A grid of $85050$ primal triangular elements with $64$ divisions along the cylinder is employed to discretize the computational domain. To complete the test definition, the viscosity coefficient is calculated to obtain the desired Reynolds number, e.g. to get $\textrm{Re}=185$ we have set the viscosity to $\mu = 5.4054054\cdot 10^{-3}$. The results presented have been obtained using the second order implicit hybrid FV/FE scheme with the Ducros numerical flux function and auxiliary artificial viscosity $c_{\alpha} =0.5$. All simulations are run  until the von Karman vortex street has been fully developed so that the vortex frequency can be computed. Since this benchmark is characterized by the Reynolds and Strouhal numbers, we plot in Figure \ref{fig:RevsSt} the value of $St=\frac{f_v D}{u_1}$ calculated for a set of different Reynolds numbers, $\textrm{Re} \in\left\lbrace 50, 75, 100, 125, 150, 175, 185 \right\rbrace$. The results obtained using a semi-implicit DG scheme of order $4$, \cite{TD14}, the experimental results of \cite{WillBrown98} and the so-called universal Strouhal curve are included for comparison.
Furthermore, the time series of the drag and lift coefficients of the simulation for $\textrm{Re}=185$, reported in Figure \ref{fig:draglift185}, are in agreement with the numerical results presented in \cite{Guilmineau2002,HybridALE}. Also the contour plots of the vorticity field for different times along one periodic cycle are depicted in Figure \ref{fig:contourvorticityviscouscylinder} for a qualitative comparison with further references. 
\begin{figure}
	\centering
	\includegraphics[width=0.6\linewidth]{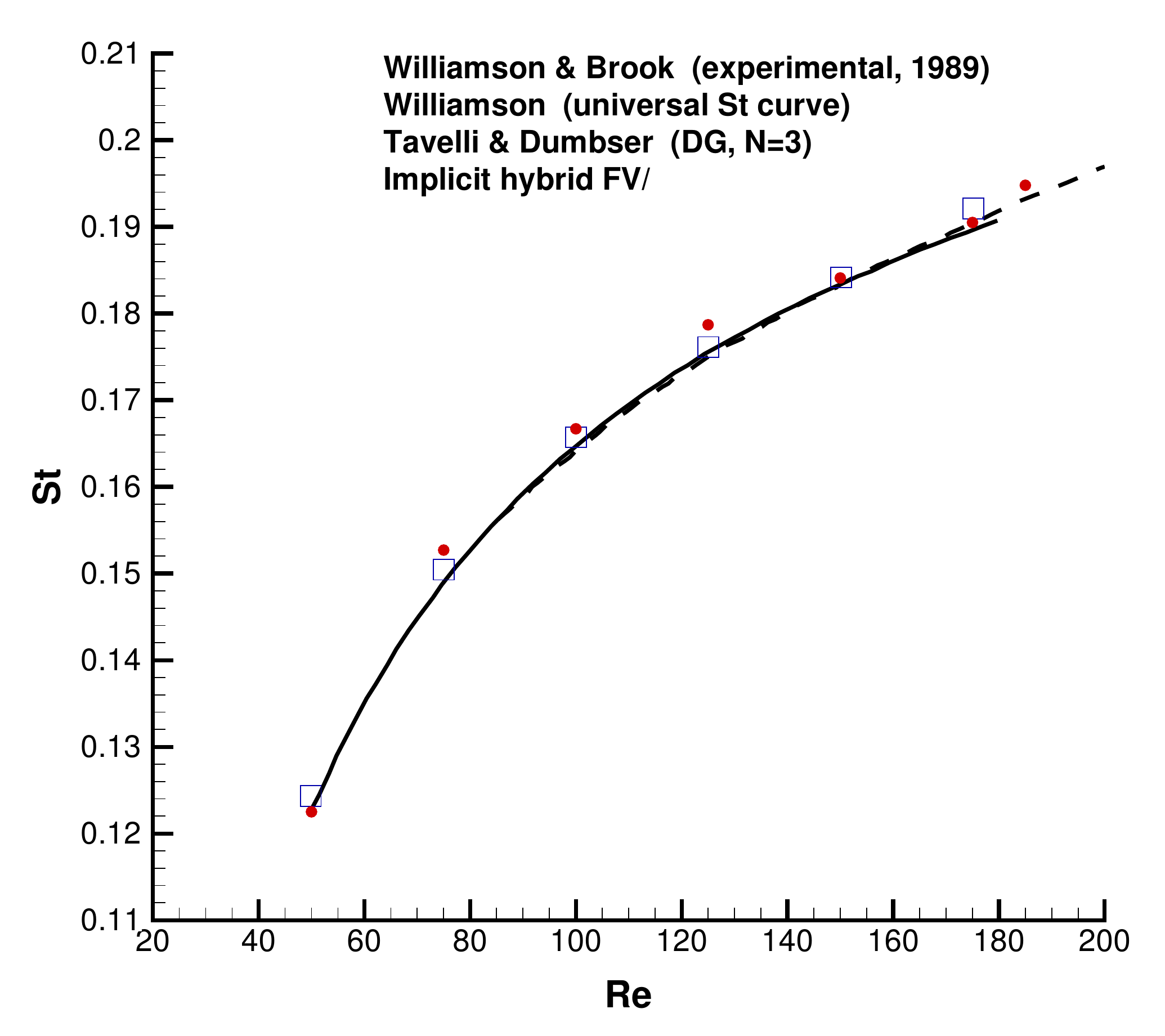}
	\caption{Viscous flow around a cylinder. Relationship of Reynolds and Strouhal numbers for the implicit hybrid FV/FE method (red dots), the
	semi-implicit DG proposed in \cite{TD14} (blue squares), the experimental results of \cite{WillBrown98} (continuous black line) and the so-called universal Strouhal relation (dashed black line).}
	\label{fig:RevsSt}
\end{figure}
\begin{figure}
	\centering
	\includegraphics[width=0.8\linewidth]{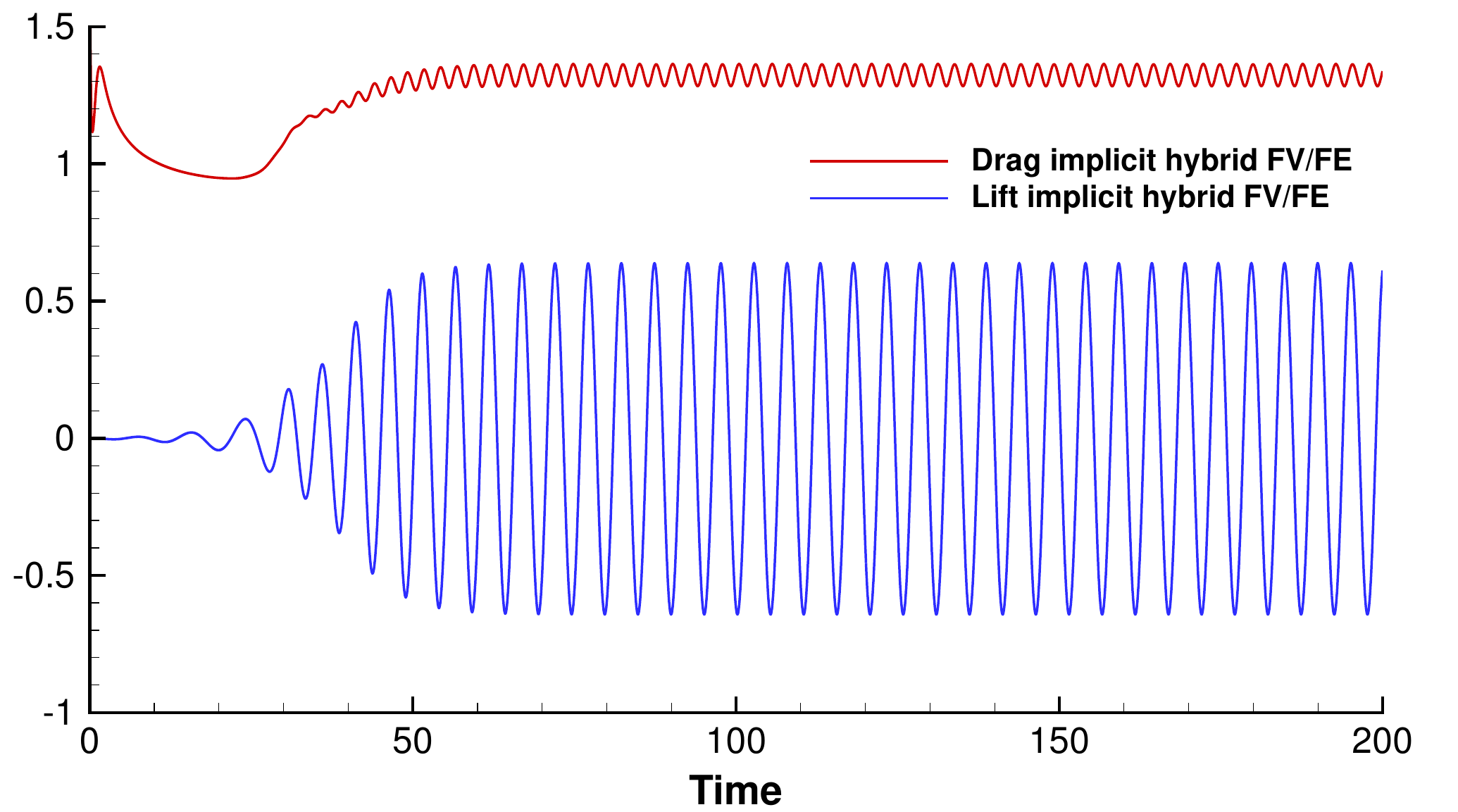}
	\caption{Viscous flow around a cylinder. Drag and lift coefficients for a flow of Re$=185$  obtained using the implicit hybrid FV/FE scheme with the Ducros flux function, $c_{\alpha}=0.5$ and $\Delta t = 10^{-2}$.}
	\label{fig:draglift185}
\end{figure}
\begin{figure}
	\centering
	\includegraphics[width=0.49\linewidth]{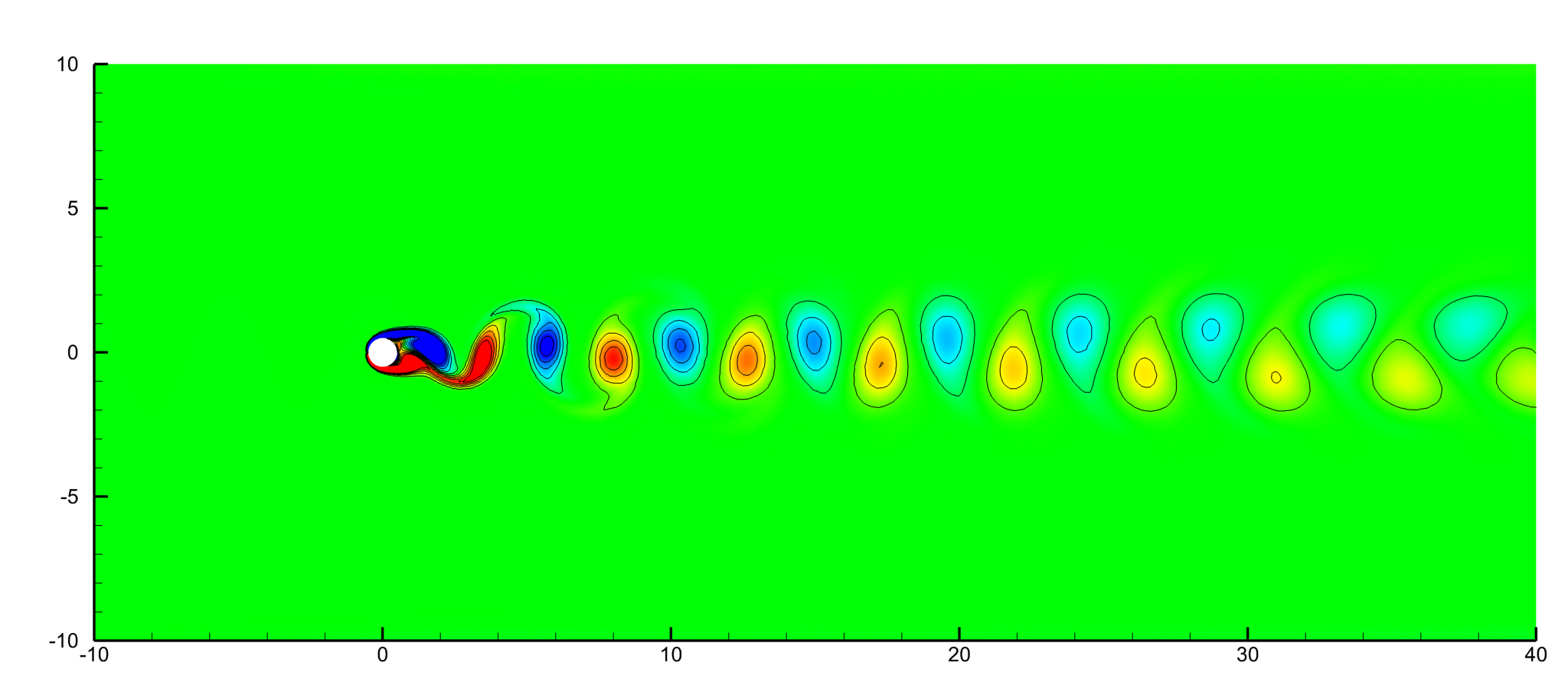}
	\includegraphics[width=0.49\linewidth]{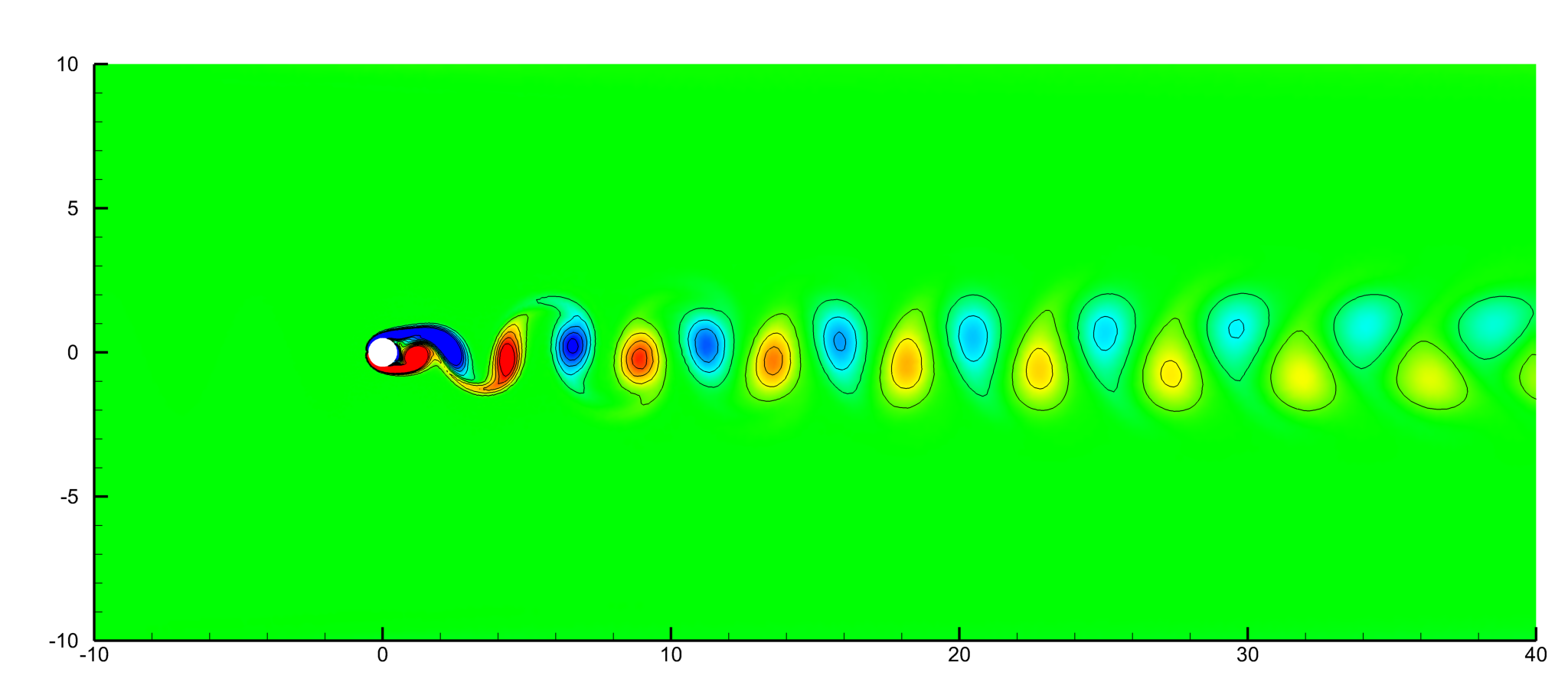}
	\includegraphics[width=0.49\linewidth]{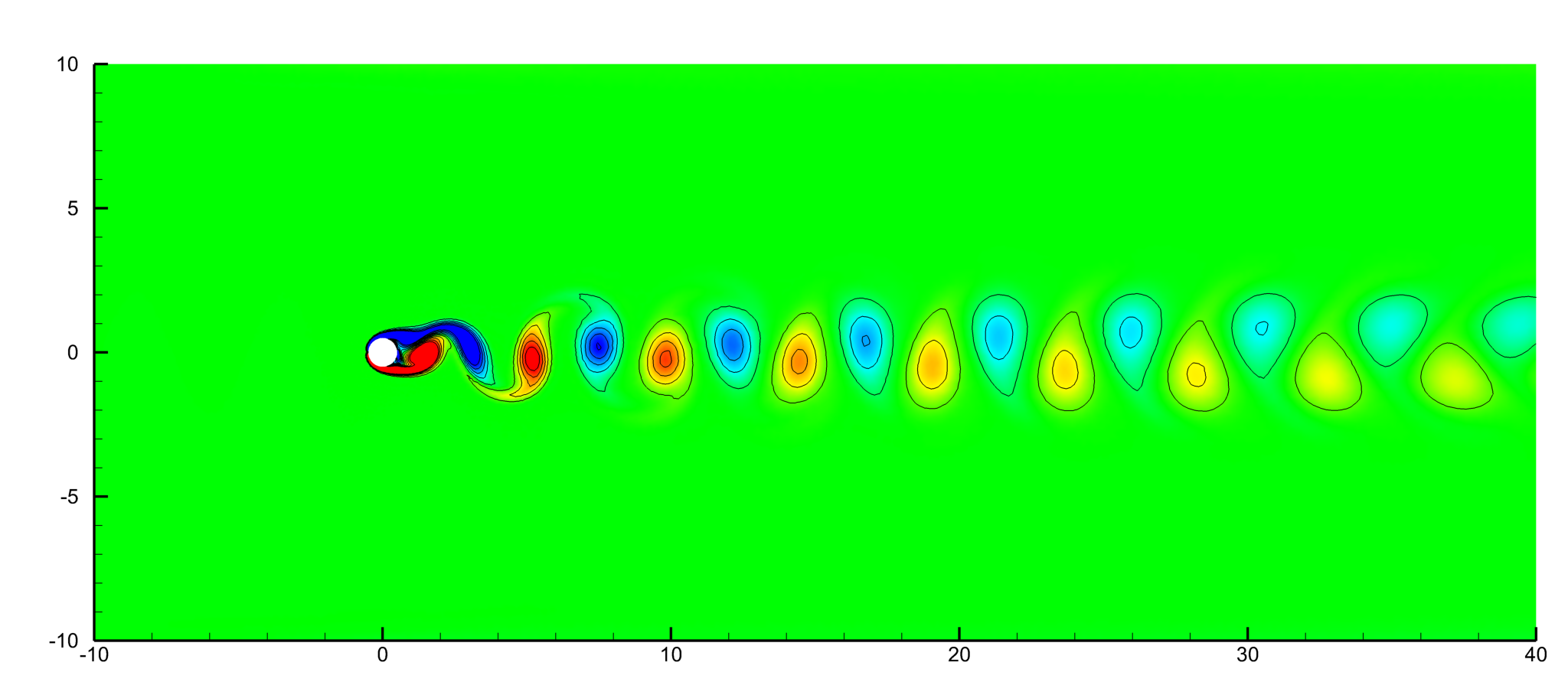}
	\includegraphics[width=0.49\linewidth]{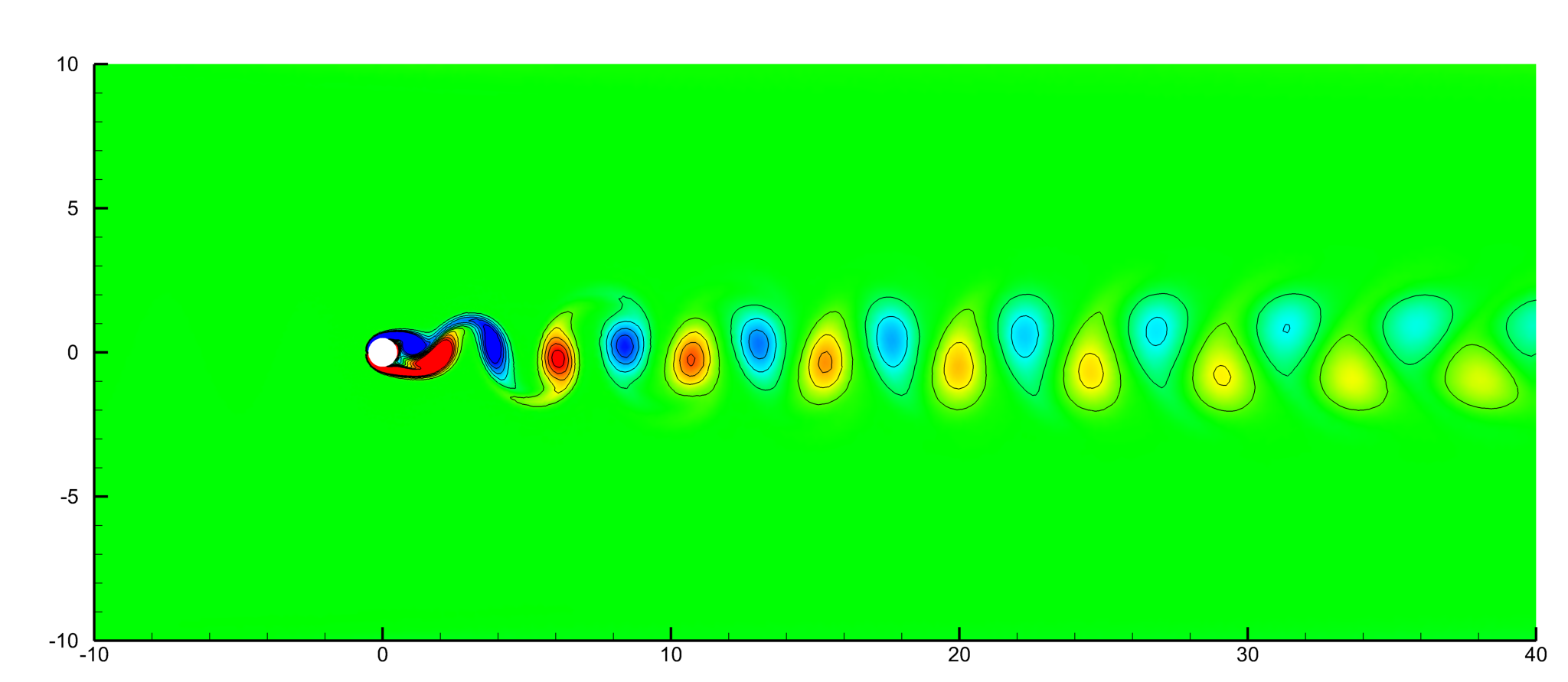}
	\includegraphics[width=0.49\linewidth]{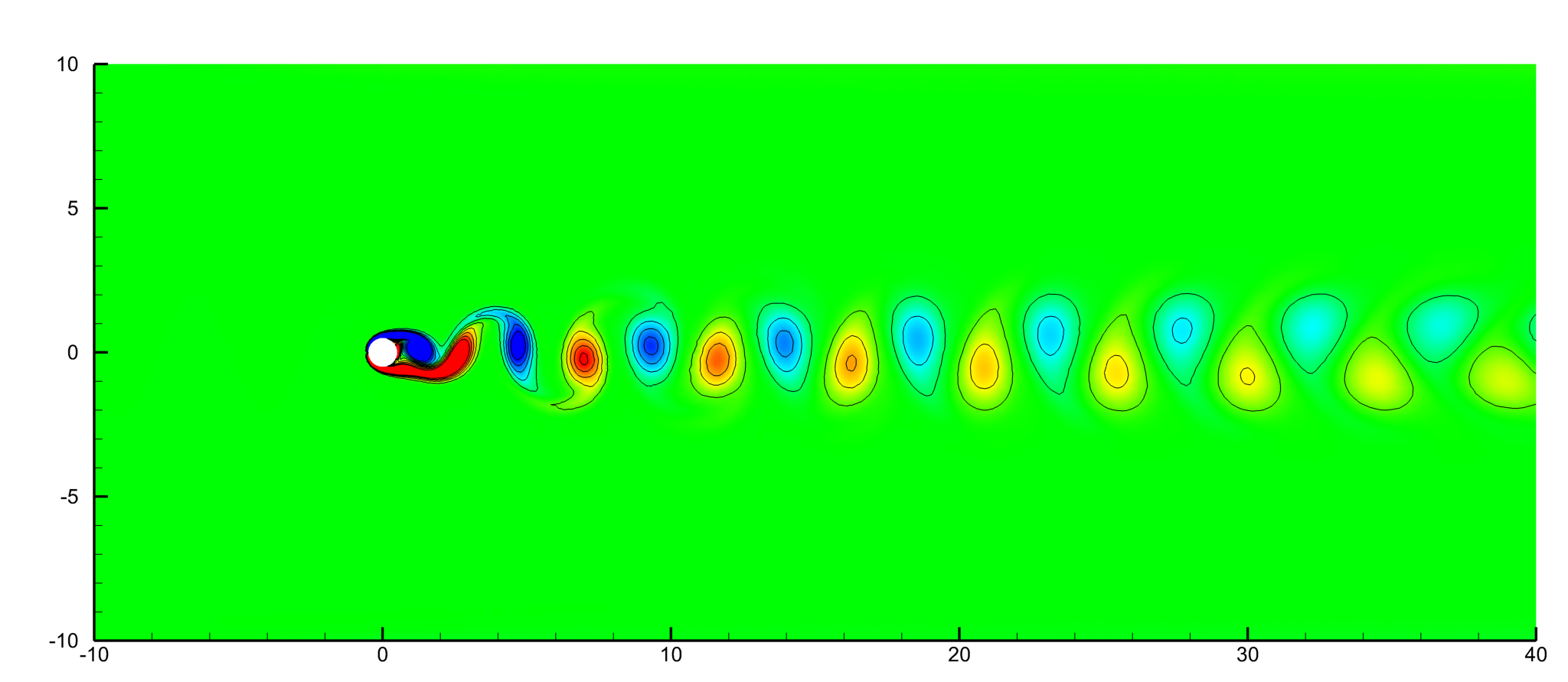}
	\includegraphics[width=0.49\linewidth]{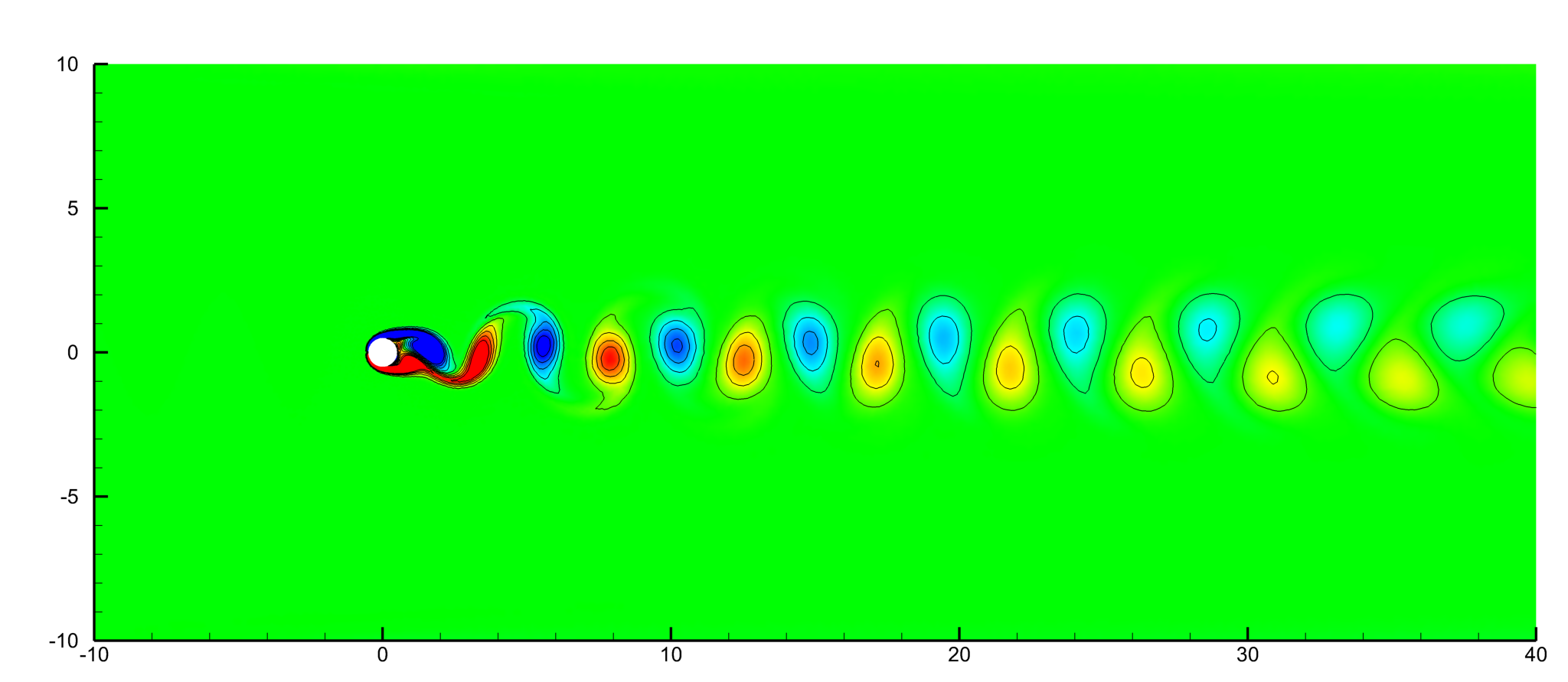}
	\caption{Viscous flow around a cylinder. Contour plots of the vorticity field for a flow of Re$=185$ around a cylinder at times $100,101,102,103,104,105$ (from top left to bottom right).}
	\label{fig:contourvorticityviscouscylinder}
\end{figure}

In this test case the reordering of elements for the preconditioner is crucial, otherwise the convergence of the method is much slower. To illustrate how the mesh reordering is working in a parallel simulation we portray in Figure \ref{fig:ViscCyl_meshreordering} the original dual cell index and the reordered one.
\begin{figure}
	\centering
	\includegraphics[trim =0 70 0 0,clip,width=0.9\linewidth]{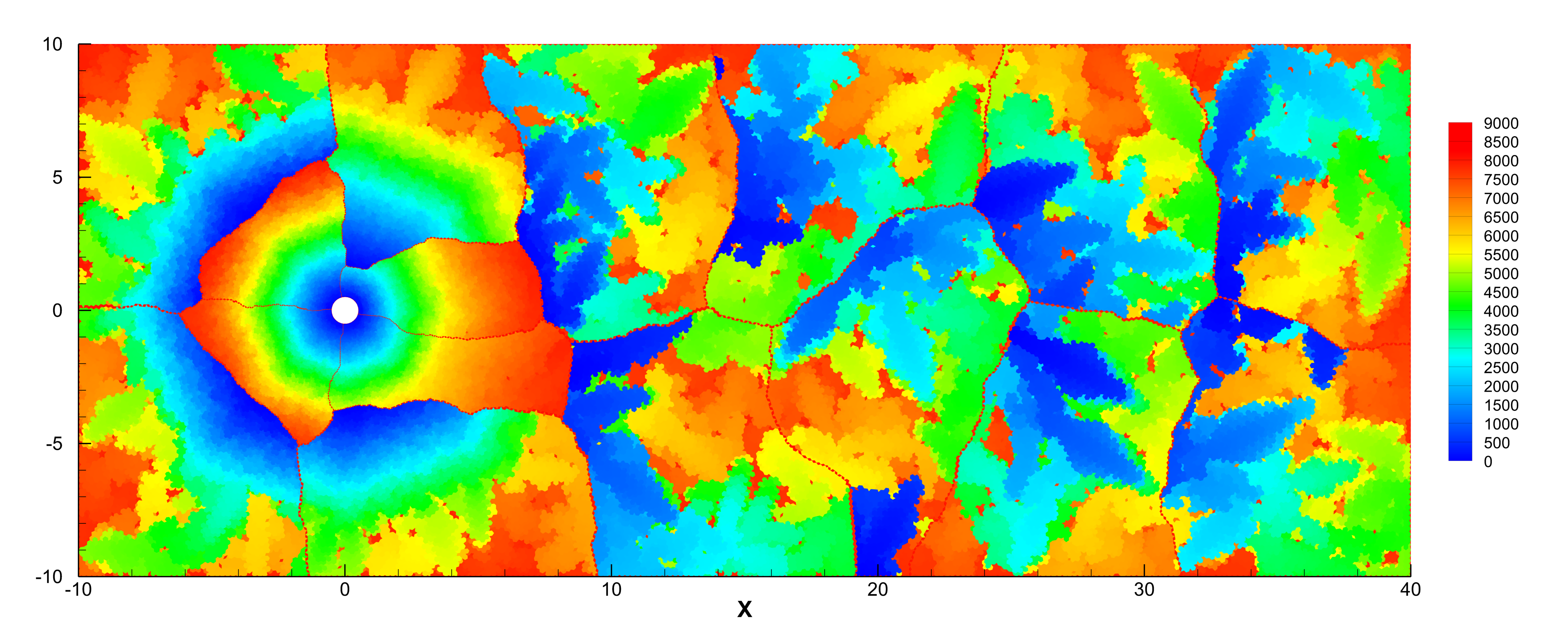}
	\includegraphics[trim =0 70 0 0,clip,width=0.9\linewidth]{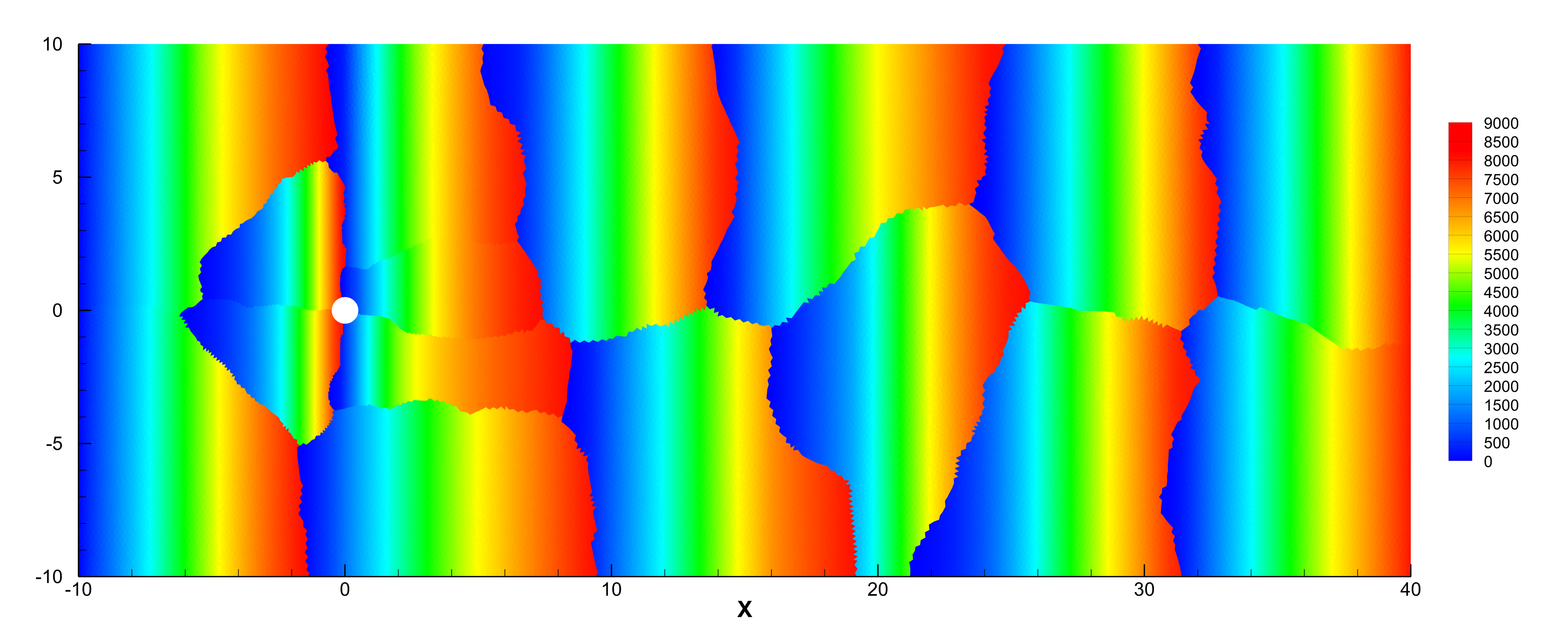}
	\caption{Viscous flow around a cylinder. Dual cell index for a partition in 16 CPUs. Top: original index number inside each partition. Bottom: index after reordering for the preconditioner.}
	\label{fig:ViscCyl_meshreordering}
\end{figure}

Besides, we also analyze the speed-up of the fully-implicit hybrid scheme with respect to the semi-implicit hybrid method presented in \cite{HybridNNT}. The semi-implicit hybrid FV/FE scheme employed discretizes the nonlinear convective terms explicitly, while the viscous terms and the pressure system are treated implicitly. Accordingly, the CFL time step restriction, taken as CFL$=0.5$, only depends on the eigenvalues of the transport subsystem. The Ducros and Rusanov numerical flux functions have been considered for both schemes. The computational cost of the simulations for $\textrm{Re}=185$ up to time $200$ is reported in Table \ref{tab:viccyl_times}. All simulations have been carried out on 128 CPU cores of the Finisterrae III supercomputer at CESGA, which accounts for Intel$^{\textrm{\textregistered}}$ Xeon Ice Lake 8352Y processors with 32 cores at 2.2GHz per node.

\begin{table}
	\centering
	\begin{tabular}{lcc}
		\hline
		& Semi-implicit scheme &  Implicit scheme \\
		\hline
		Rusanov & $16120.12$ & $1844.52$ \\
		Ducros  & $16442.73$ & $1412.57$ \\
		\hline
	\end{tabular}
	\caption{Viscous flow around a cylinder. CPU time (s) employed to run the simulation for $\textrm{Re}=185$ up to time $200$ using the novel fully implicit hybrid FV/FE scheme and the semi-implicit hybrid FV/FE method in \cite{HybridNNT}.}
	\label{tab:viccyl_times}
\end{table}

\subsection{Hagen-Poiseuille benchmark}
Here, we test the proposed method in the context of a classical three-dimensional academic test for blood flow solvers. We consider a stationary fluid in a three-dimensional pipe of axis $x_{3}$, length $L=1$ and radius $R=0.1$. A constant pressure gradient $\Delta p < 0$ is imposed between its two ends, enforcing the fluid to flow in $x_3$-axis direction. This benchmark corresponds to a Hagen-Poiseuille flow whose exact solution is a parabolic velocity profile given by, see e.g. \cite{SG16}, 
\begin{equation}
	\mathbf{u}=\left(0,0,\frac{1}{4}\frac{\Delta p}{L}\frac{\rho}{\mu}(r^2-R^2)\right).
\end{equation}
As initial condition we consider a fluid at rest. Moreover, we set $\rho=1$ and $\mu=0.1$. The pressure gradient imposed between the left inflow and the right outlet is $\Delta p=-4.8$, leading to a maximum velocity in $x_3$-axis direction of $0.12$. 
The computational results at time $t=10.13$ are shown in Figure \ref{fig:HP3D}, where we compare the sectional velocity across the flow at the mid section against the exact solution. The computed numerical solution excellently matches the exact one. In Figure \ref{fig:HP3D_velcont} the velocity field in the $x_2=0$ plane is depicted. 

\begin{figure}[h]
	\centering
	\includegraphics[trim =10 10 10 10,clip,width=0.7\linewidth]{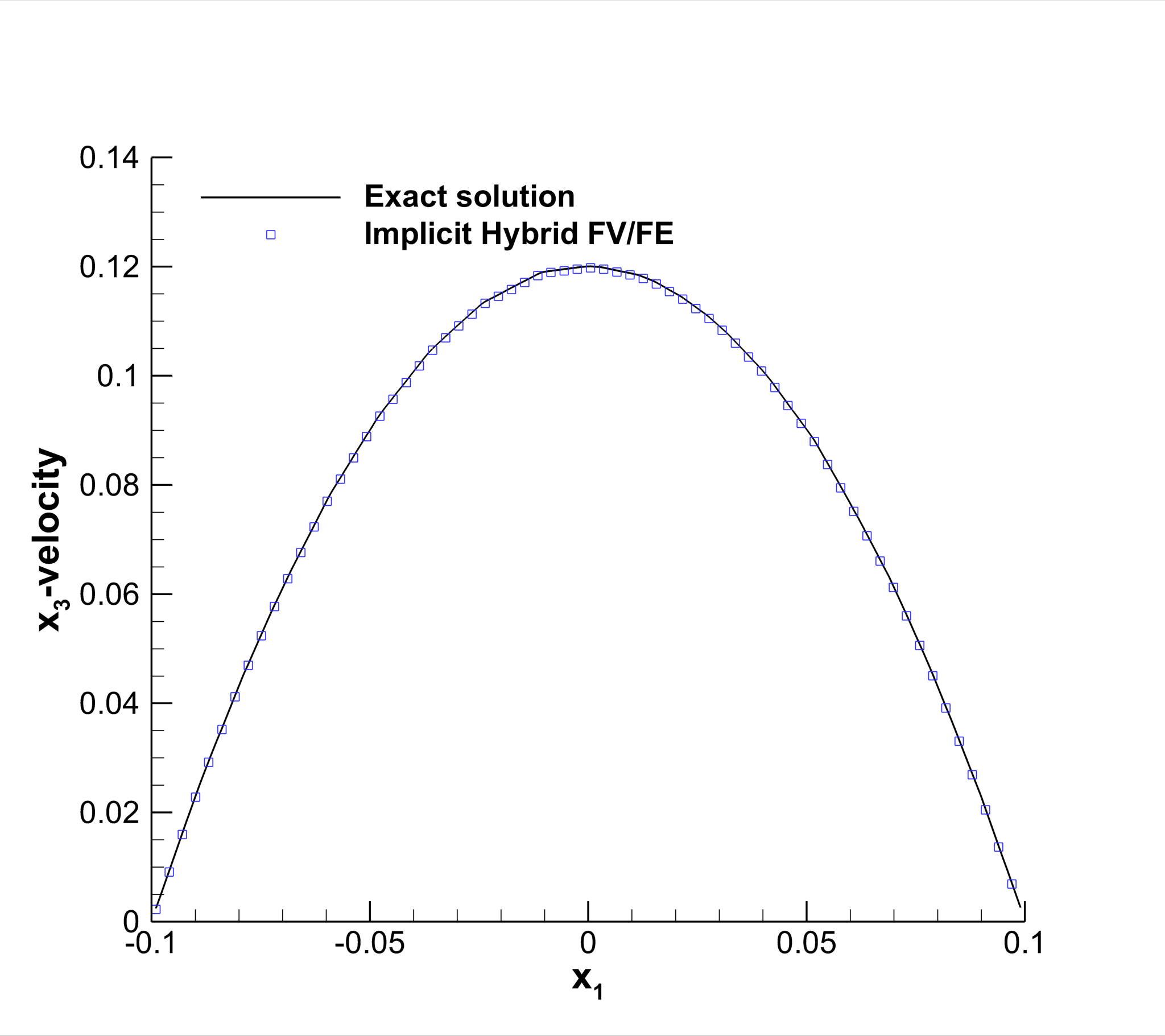}
	\caption{Steady laminar Hagen-Poiseuille flow in a cylinder. 1D cut along the $x_1$ axis at $x_2=0$ and $x_3=L/2$. }
	\label{fig:HP3D}
\end{figure}
\begin{figure}
	\centering \includegraphics[trim =10 10 10 10,clip,width=0.9\linewidth]{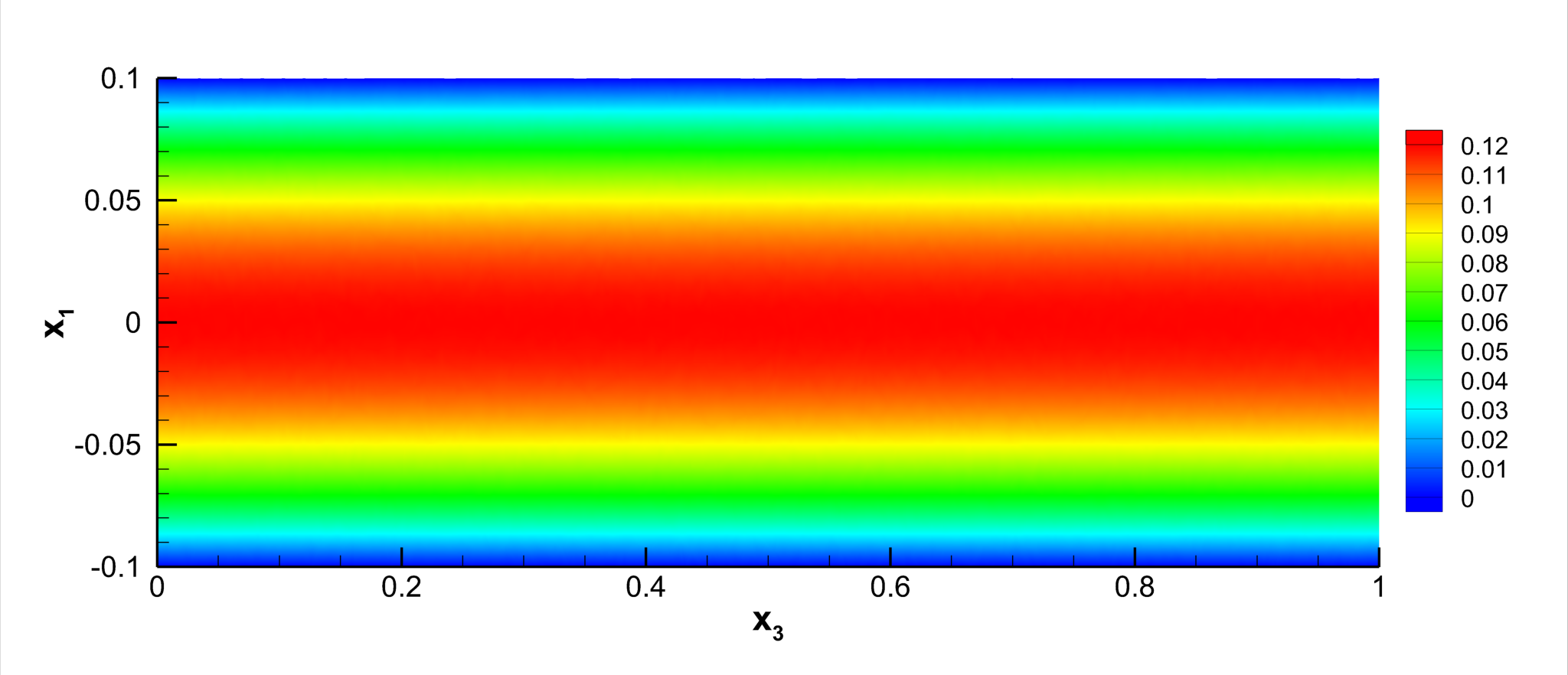}
	\caption{Steady laminar Hagen-Poiseuille flow in a cylinder. Velocity $u_{3}$ contour in the $x_2=0$ plane.}
	\label{fig:HP3D_velcont}
\end{figure}

\subsection{Womersley problem}
In this section the implicit hybrid FV/FE algorithm is validated against the exact solution for an oscillating flow of viscosity $\mu=8.94\cdot 10^{-4}$ passing between two flat parallel plates,  \cite{Womersley:1955,Loudon1998}. 
The unsteady flow is driven by a sinusoidal pressure gradient on the boundary given by 
\begin{equation}
    \frac{\partial p}{\partial x_1}=\frac{p_{out}(t)-p_{in}(t)}{l}=-Ae^{int},
    \label{eq:pgradient}
\end{equation}
where $A=1$ is the amplitude of the pressure gradient, $n$ denotes the frequency of the oscillation, $i$ represents the imaginary unit, $l$ is the length of the plates and $p_{in}$ and $p_{out}$ are the pressure at the inlet and at the outlet, respectively. By imposing \eqref{eq:pgradient} at the two ends and a no-slip boundary condition on the upper and lower flat plates, the resulting velocity field can be expressed in complex form as 

\begin{equation*}
    \mathbf{u}(x_1,x_2,t)=\left(\frac{A}{i \, \rho n}\left(1-\frac{\cosh(\textnormal{Wo} \,  i^{1/2}\frac{x_2}{a})}{\cosh(\textnormal{Wo} \, i^{1/2})}\right)e^{i \, nt},0\right),
\end{equation*}
where the Womersley number, $\textnormal{Wo}=a\sqrt{\frac{n}{\nu}}=10$, with $\nu$ the kinematic viscosity and $2a$ the distance between the plates, describes the nature of the unsteady flow. For the simulation we consider the computational domain $\Omega=[-0.5,1]\times[-0.2,0.2]$ discretized with a mesh of $228$ primal elements. The time step is fixed during the whole simulation to $\Delta t=2.5\cdot 10^{-3}$ so that we properly follow the flow oscillations.
The obtained velocity profile along the vertical cut  $x_1=0$ is reported in Figure \ref{fig:Wo} for $t\in\{ 0.35,0.7,1.4,2.1,2.45 \}$. A good agreement between exact and numerical solution can be observed.
\begin{figure}
    \centering
    \includegraphics[trim =10 10 10 10,clip,width=0.7\linewidth]{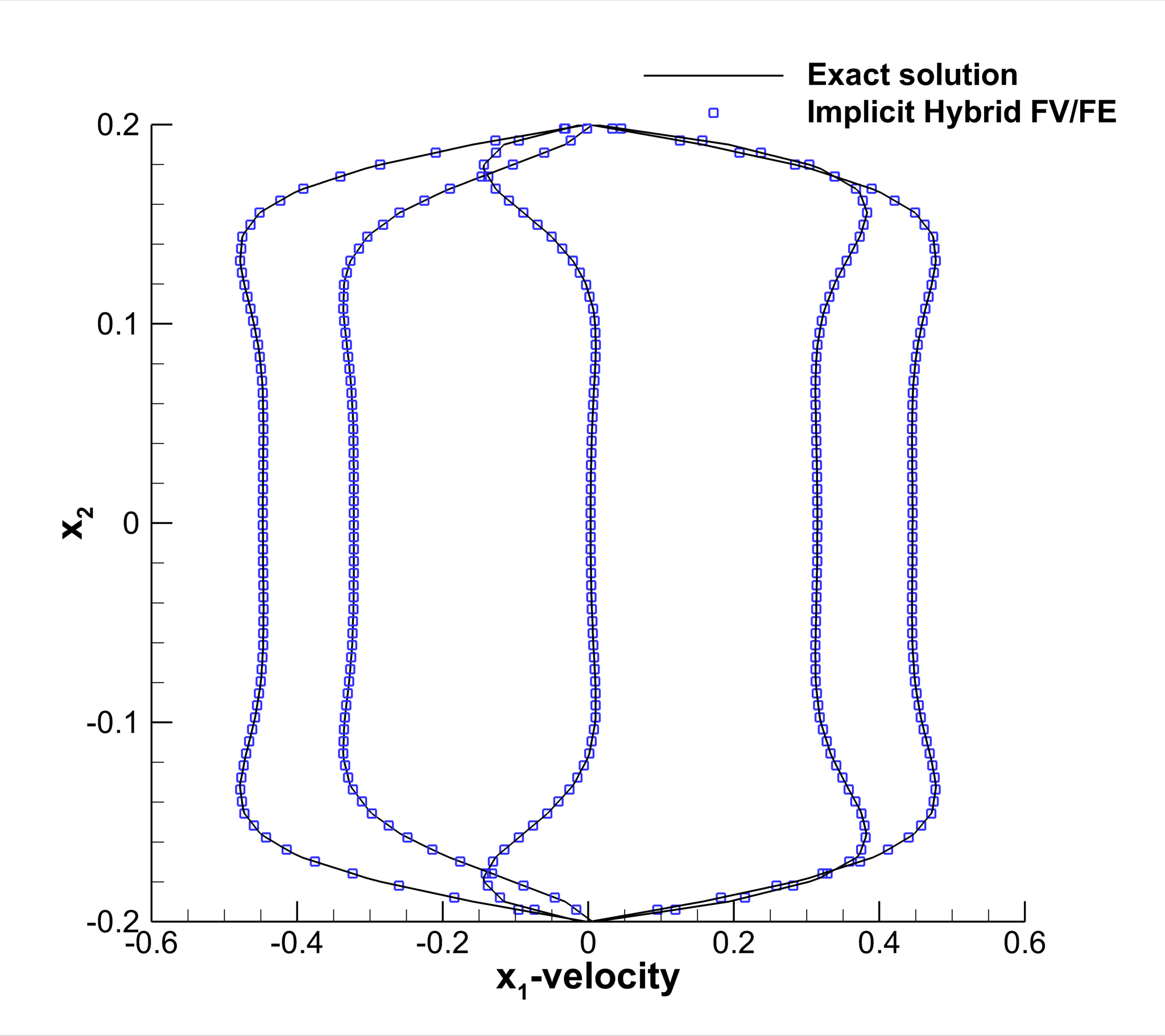}
    \caption{Womersley problem. Comparison between the exact and the numerical solution for the Womersley profiles at times $t=2.1$s, $t=2.45$s, $t=1.4$s, $t=0.35$s, $t=0.7$s, from left to right. }
    \label{fig:Wo}
\end{figure}

\subsection{Ideal artery with stenosis}
To apply the presented algorithm to a more realistic test case in the context of blood flow dynamics, we propose the study of the flow inside an ideal stenotic artery. To this end, we consider a steady flow of a viscous fluid in a duct of length $L=2$ with axis $x_1=0$ which has a shrinkage of $40\%$ in diameter at the half length.
The pressure gradient imposed between the left and right boundaries of the computational domain drives the fluid flow in the $x_1-$direction. On the lateral boundary, no-slip wall boundary conditions are imposed. 
Initially the flow is assumed to be at rest, the density and viscosity are set to  $\rho=1.0$,  $\nu=0.01$ and we define a fixed time step equal to $\Delta t=0.5$. Four different simulations regarding the method employed are run: the GMRES-Newton method and the BiCGStab-Newton algorithm with both the first and second order approach for the convective terms. 
Figure \ref{fig:Stenosis1} shows the velocity contour on a longitudinal clip of the domain together with the streamlines over its surface. Moreover, the velocity profiles of each of the four simulations on the section $x=-0.5$, upstream the stenosis, and the section located at the maximal narrowing of the duct, $x=0$, are reported in Figure \ref{fig:Stenosis2}. We can observe that both methods provide numerical solutions in strict accordance using either the first or the second order in space scheme.
\begin{figure}[h]
    \centering
    \includegraphics[trim =10 10 10 10,clip,width=0.8\linewidth]{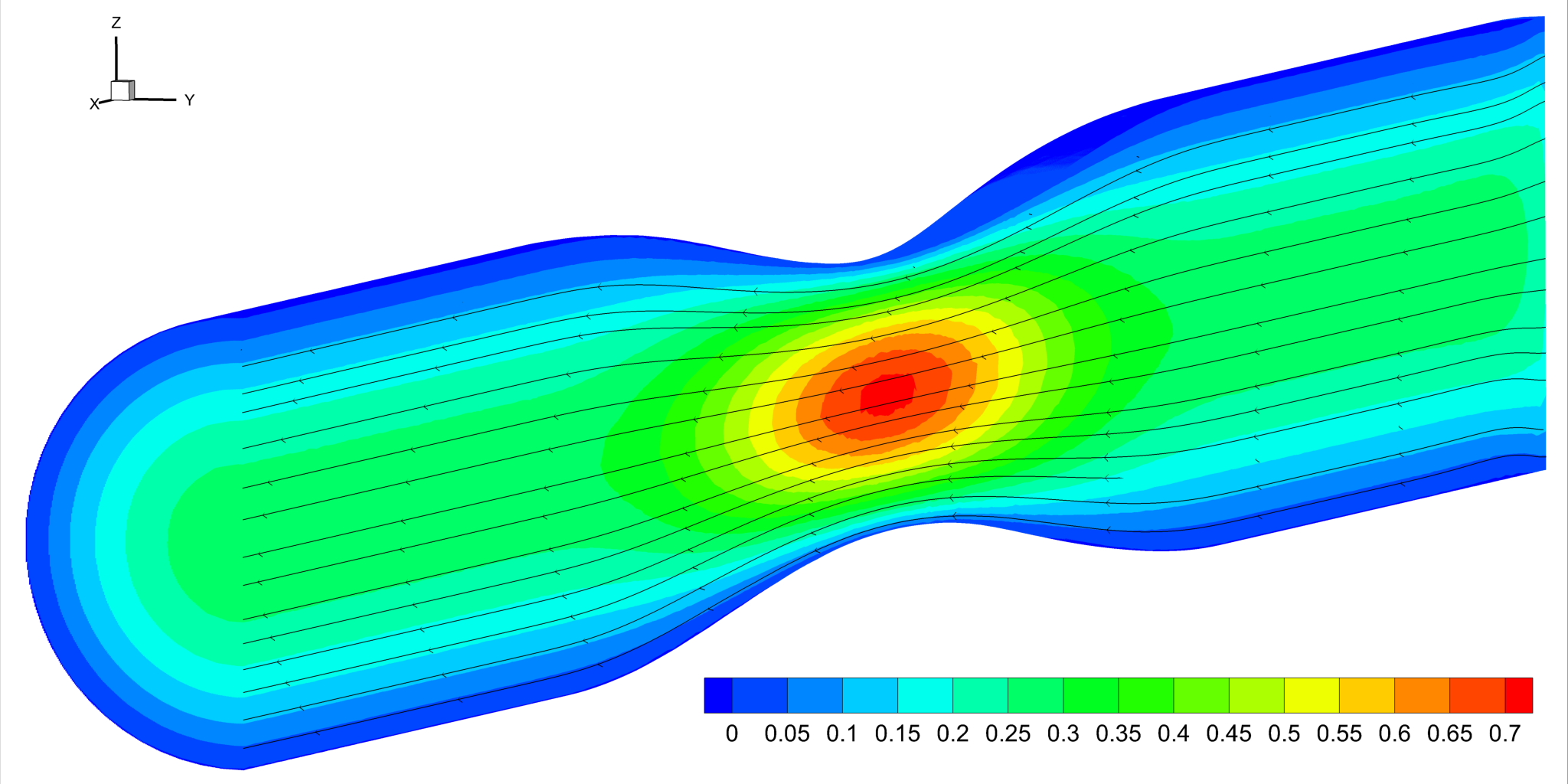}
    \caption{Ideal artery with stenosis Test. Contours of the velocity field on a longitudinal clip of the domain and streamlines over its surface.}
    \label{fig:Stenosis1}
\end{figure}    
\begin{figure}[h]
    \centering
    \includegraphics[trim =10 10 10 10,clip,width=0.45\linewidth]{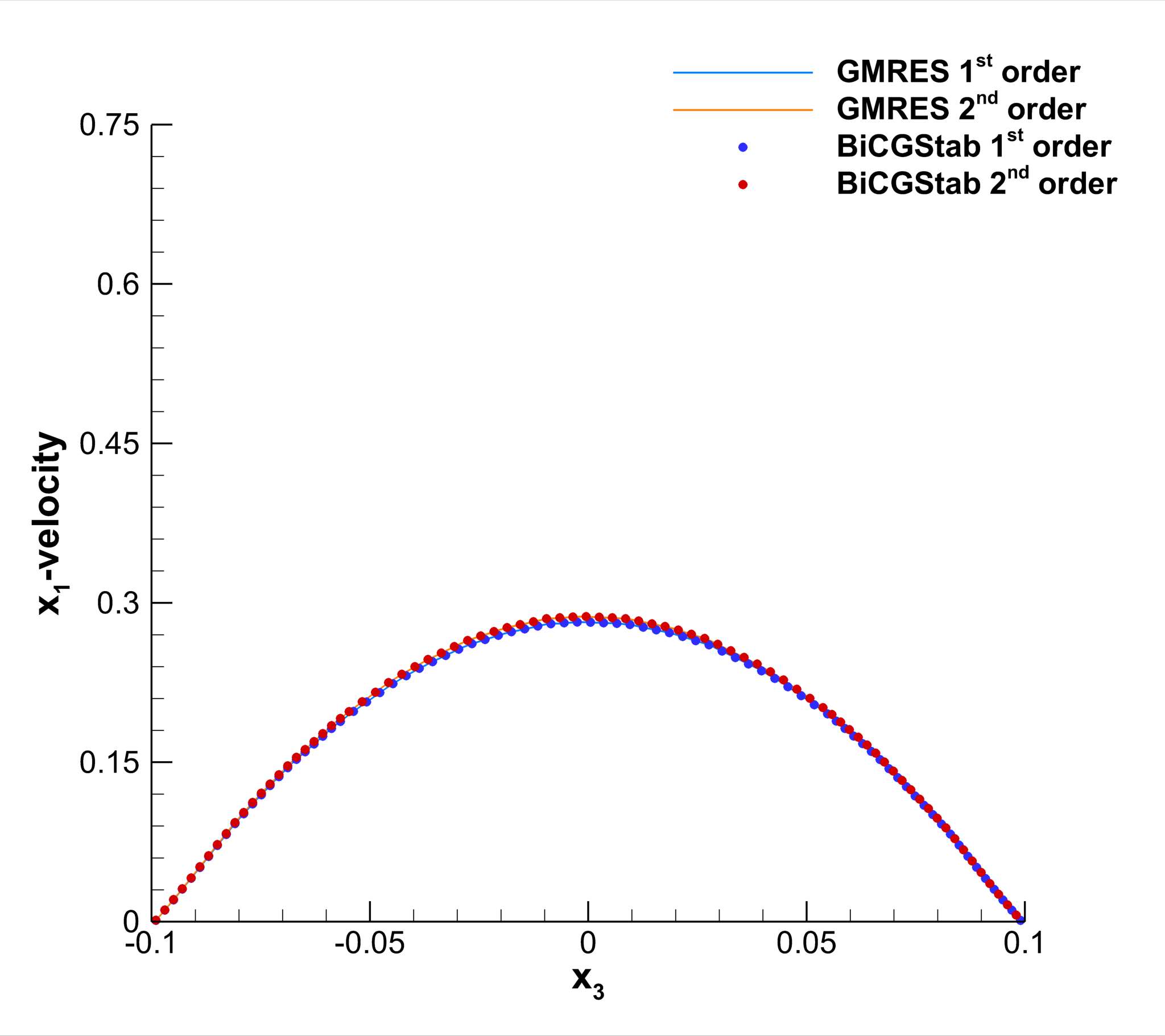}\quad
    \includegraphics[trim =10 10 10 10,clip,width=0.45\linewidth]{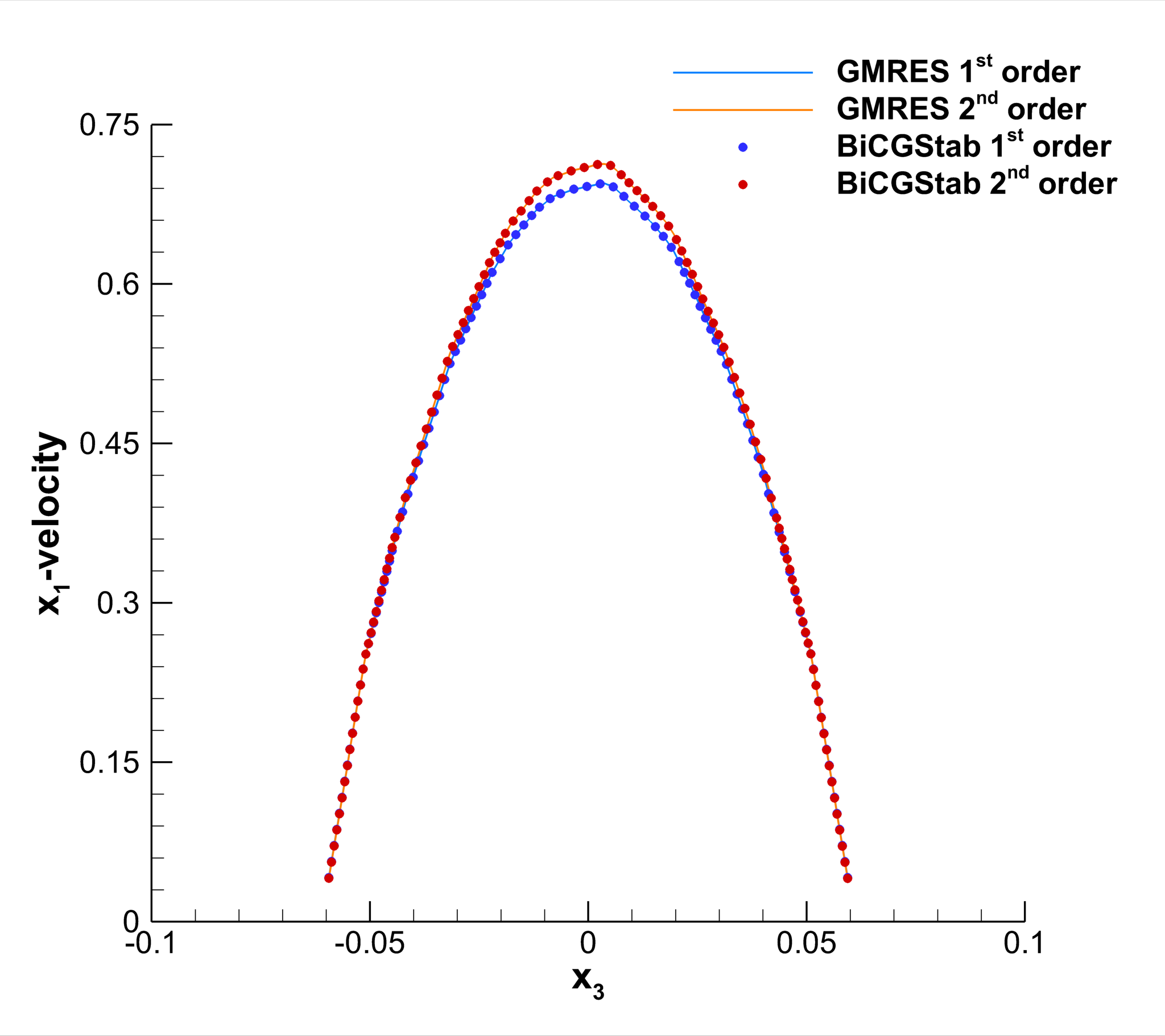}
    \caption{Ideal artery with stenosis Test. Left: velocity profiles upstream the stenosis at $x=-0.5$. Right: velocity profiles at the narrowest section $x=0$.}
    \label{fig:Stenosis2}
\end{figure}

In addition, to illustrate how the mesh reordering works on a 3D domain, we report in Figure \ref{fig:Stenosis3} the original dual cell index and the reordered one for a simulation carried out sequentially.
\begin{figure}
	\centering
	\includegraphics[trim =10 10 10 10,clip,width=0.8\linewidth]{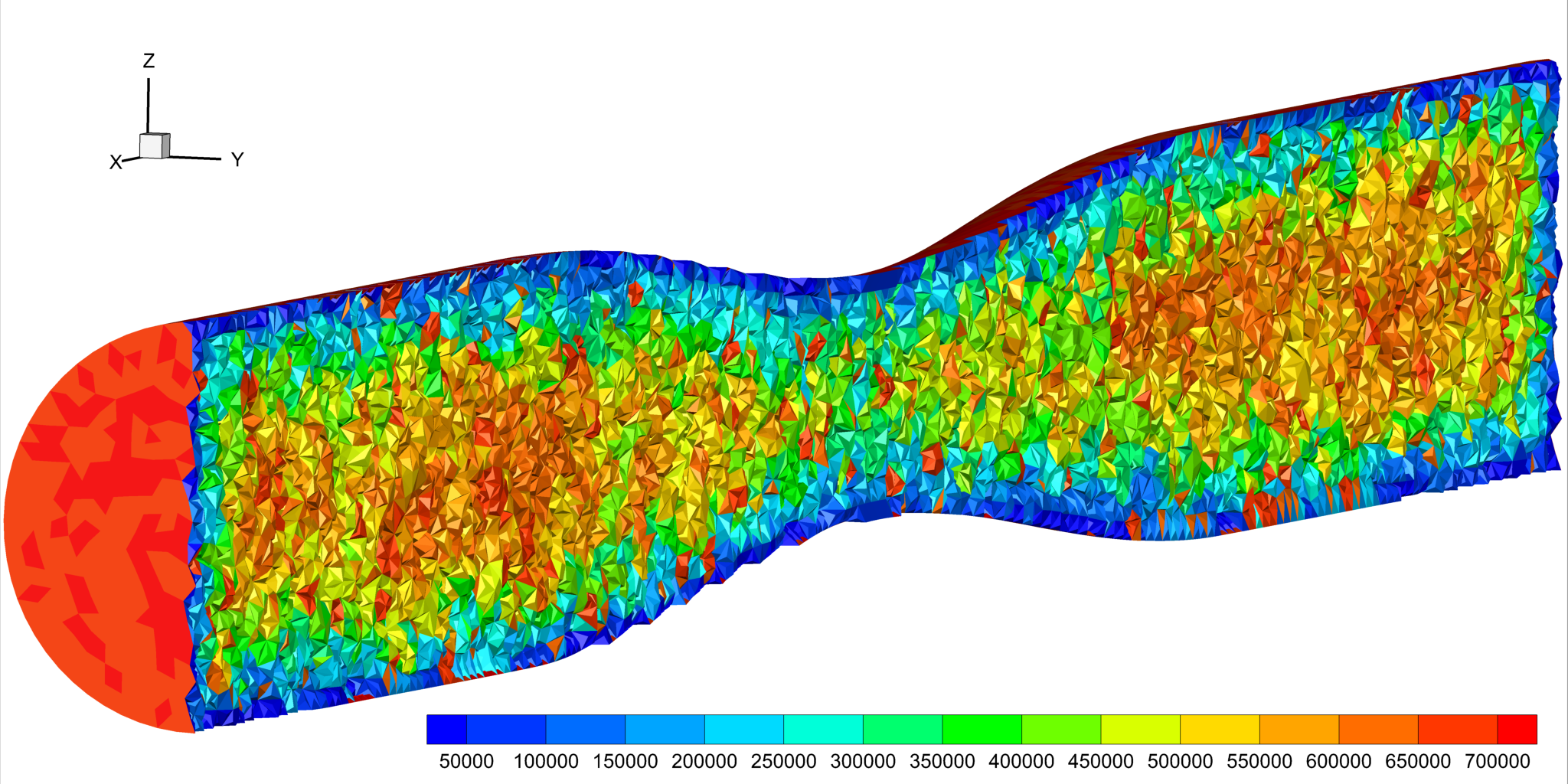}
	\includegraphics[trim =10 10 10 10,clip,width=0.8\linewidth]{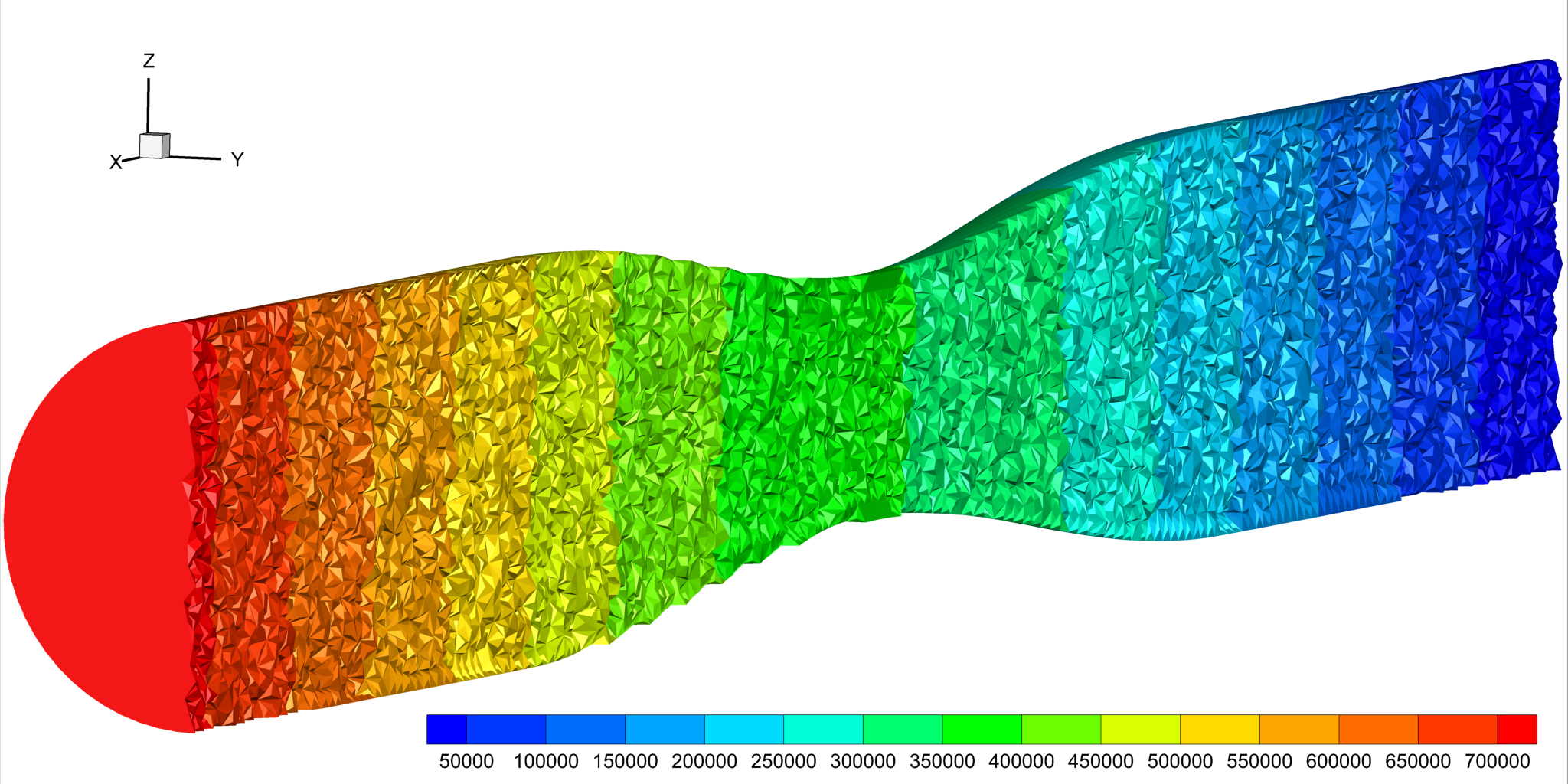}
    \caption{Ideal artery with stenosis. Top: original dual cell index inside the domain. Bottom: index after the reordering for the preconditioner.}
	\label{fig:Stenosis3}
\end{figure}

\subsection{Realistic test case: a coronary tree}
As last test we present a blood flow simulation for a realistic three-dimensional configuration. The geometry is segmented from the coronary computed tomography angiography image of a patient affected by a coronary artery disease. It represents a coronary tree, which includes the left main, the circumflex artery (LCX) and the left anterior descending artery (LAD) with the main diagonal branches. A severe stenosis of $85$\% occlusion is located at the mid LAD, while a minor lesion of $45\%$ is located at mid LCX.

\textcolor{black}{Following the modelling pipeline described in \cite{fossan2018,muller2021impact},} we simulate a hyperemic steady state flow through the entire coronary tree by prescribing a flow distribution among the vessels as outflow boundary condition and imposing a mean aortic pressure of $95.658$ mmHg at the inlet section. The total flow passing through the coronary tree is $4.02$ $cm^3/s$, the density is set to $1.05$ $g/cm^3$ and the viscosity is $0.035$ $cm^2/s$. 
The convective-diffusive system is solved by exploiting the preconditioned Newton-BiCGStab algorithm with the first order Ducros flux function and an auxiliary artificial viscosity of $c_{\alpha}=50$. 

The presence of a severe stenosis in a vessel causes a sharp drop in pressure, as we can observe in Figure \ref{fig:CorTree_P}. Indeed, distal to the vessel's occlusion we record a pressure reduction of $47.53$\%. Moreover, Figure \ref{fig:CorTree_V} highlights how a tortuous geometry with curvatures and sinus affects the shape of the sectional velocity profiles throughout the domain.
\begin{figure}
	\centering
	\includegraphics[trim =10 10 10 10,clip,width=0.7\linewidth]{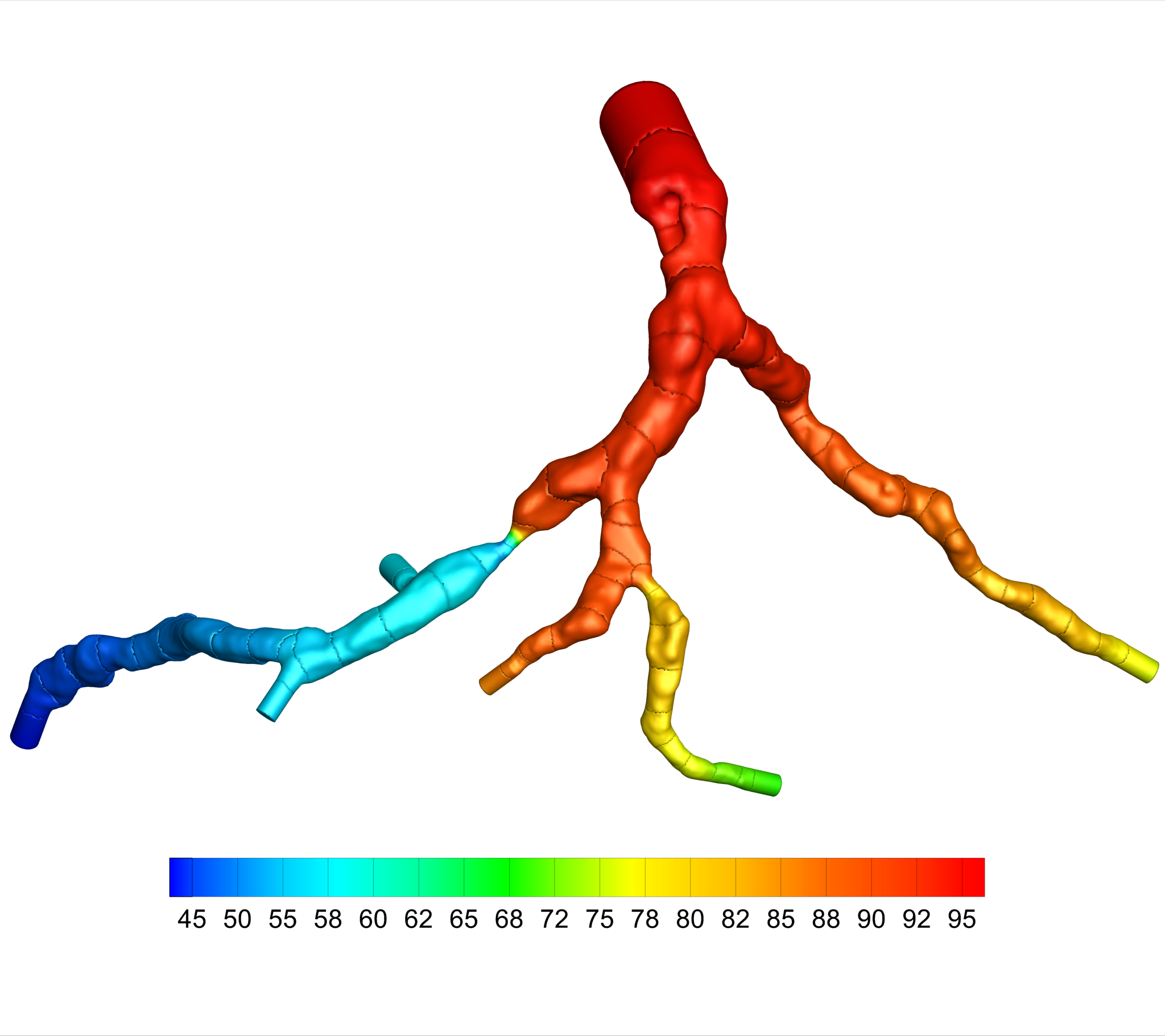}
	\caption{Coronary tree Test. Pressure distribution among the entire domain.}
	\label{fig:CorTree_P}
\end{figure}

\begin{figure}
	\centering
	\includegraphics[trim =10 10 10 10,clip,width=0.7\linewidth]{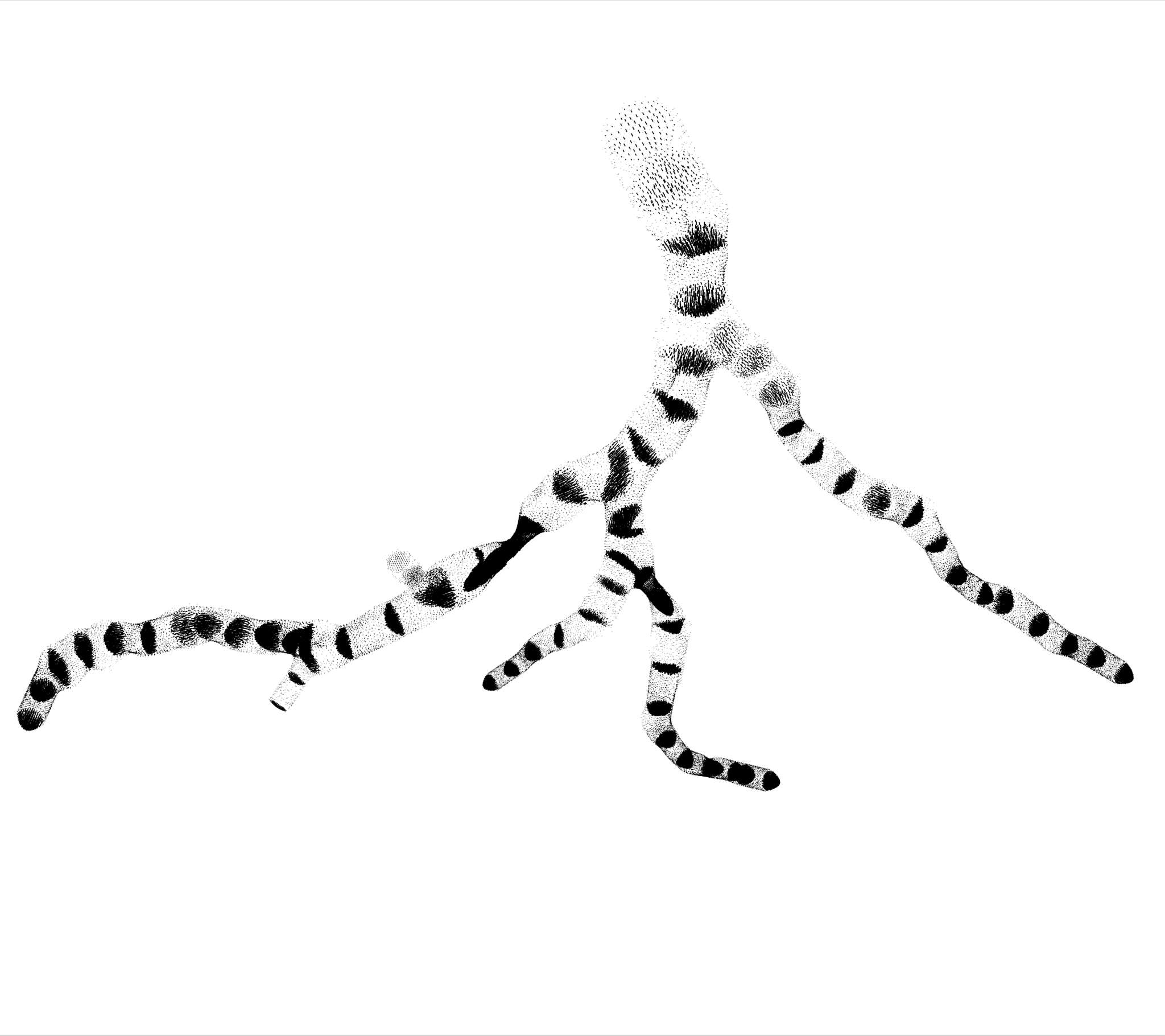}
	\caption{Coronary tree Test. Sectional velocity profiles through the domain.}
	\label{fig:CorTree_V}
\end{figure}
In case of coronary artery disease, the gold standard for diagnosis of functional severity of ischemia-inducible coronary stenosis is the Fractional Flow Reserve (FFR) \cite{pijls1995}. During catheterization, after having pharmacologically induced hyperemia, a guiding wire equipped with a pressure sensor is inserted into the coronary artery to record the pressure in the aorta ($p_{a}$) and the pressure approximately 2-3 cm distal to the lesion ($p_{d}$) to be investigated. FFR is then determined as a ratio between the mean values of $p_{d}$ and $p_{a}$ tracings, namely
\begin{equation*}
    FFR=\frac{p_{d}}{p_{a}},
\end{equation*}
and takes values between 0, standing for a complete occlusion of the vessel, and 1, corresponding to a healthy patient.
According to the FFR value, the decision on whether to proceed with surgical intervention or if it is sufficient to treat the patient with Optimal Medical Therapy is taken. \textcolor{black}{More precisely, trials evaluating the prognostic impact of the FFR have shown that a lesion is haemodynamically relevant if FFR $< 0.75$ and then revascularisation is recommended \cite{PIJLS2012}. } For the patient here considered, three FFR measurements at different locations are computed. The first one is at the end of LCX, while the others are along the LAD after the stenosis. 
\textcolor{black}{Table \ref{tab:FFR} reports the FFR values obtained using the proposed hybrid FV/FE methodology as well as the results obtained employing the open-source library CBCFLOW \cite{cbcflow} based on FEniCS \cite{fenics} which makes use of P1 and P2 FE for the computation of the velocity and pressure fields, respectively. 
\begin{table}[h]
	\centering
	\begin{tabular}{lccc}
		\hline
		& FFR$_1$ & FFR$_2$ & FFR$_3$ \\
		\hline
		CBCFLOW & 0.94 & 0.68 & 0.69 \\
		Implicit hybrid FV/FE & 0.85 & 0.52 & 0.57 \\
		\hline
	\end{tabular}
	\caption{Coronary tree Test. FFR computationally predicted by the implicit hybrid FV/FE simulation and by the FEniCS simulation at three different locations of the coronary tree.  }
	\label{tab:FFR}
\end{table}
The FFR values obtained along the LAD, namely FFR$_2$ and FFR$_3$,  indicate, for both methodologies,  the presence of a stenosis able to induce myocardial ischemia, while the stenosis located at mid LCX associated to FFR$_1$ resulted to be not functionally significant. The findings are in agreement with conclusions based on the anatomical measurements, \cite{fossan2018,muller2021impact}. 
}

In this test we can clearly appreciate the advantage in terms of computational cost when using the new fully implicit hybrid scheme proposed in this paper rather than the semi-implicit hybrid method described in \cite{HybridMPI}.
Indeed, the new fully implicit hybrid FV/FE method employs \textcolor{black}{$22262.63$s} 
to reach a final time of $t= 3.0$s. 
Meanwhile, solving the convective part of the semi-implicit scheme with the first order semi-implicit hybrid FV/FE method, in which the time step is computed according to the convective CFL condition with CFL$=0.9$, takes $260640.0$s to reach only $t=7.4472\cdot 10^{-3}$s so more than one year would be needed to arrive at a final time of $3$s. Both simulations have been performed on 60 CPU cores of an Intel$^{\textrm{\textregistered}}$ Xeon-Gold 6140M cluster with 768~GB of RAM. 
The measured speed-up factor of the new fully implicit hybrid FV/FE scheme compared to the previous semi-implicit FV/FE method is therefore approximately 4716, i.e. more than three orders of magnitude. 
\textcolor{black}{In the complex 3D domain under study, performing a mesh reordering is crucial to get reasonable computational times. Figure \ref{fig:cortree_reorder} shows how the proposed reordering technique reorganizes the indexes of the control volumes in each CPU following the centerlines of the vessels, which correspond with the main flow direction, as shown in Figure~\ref{fig:CorTree_V}.}
\begin{figure}
	\centering
	\includegraphics[trim =10 10 10 10,clip,width=0.8\linewidth]{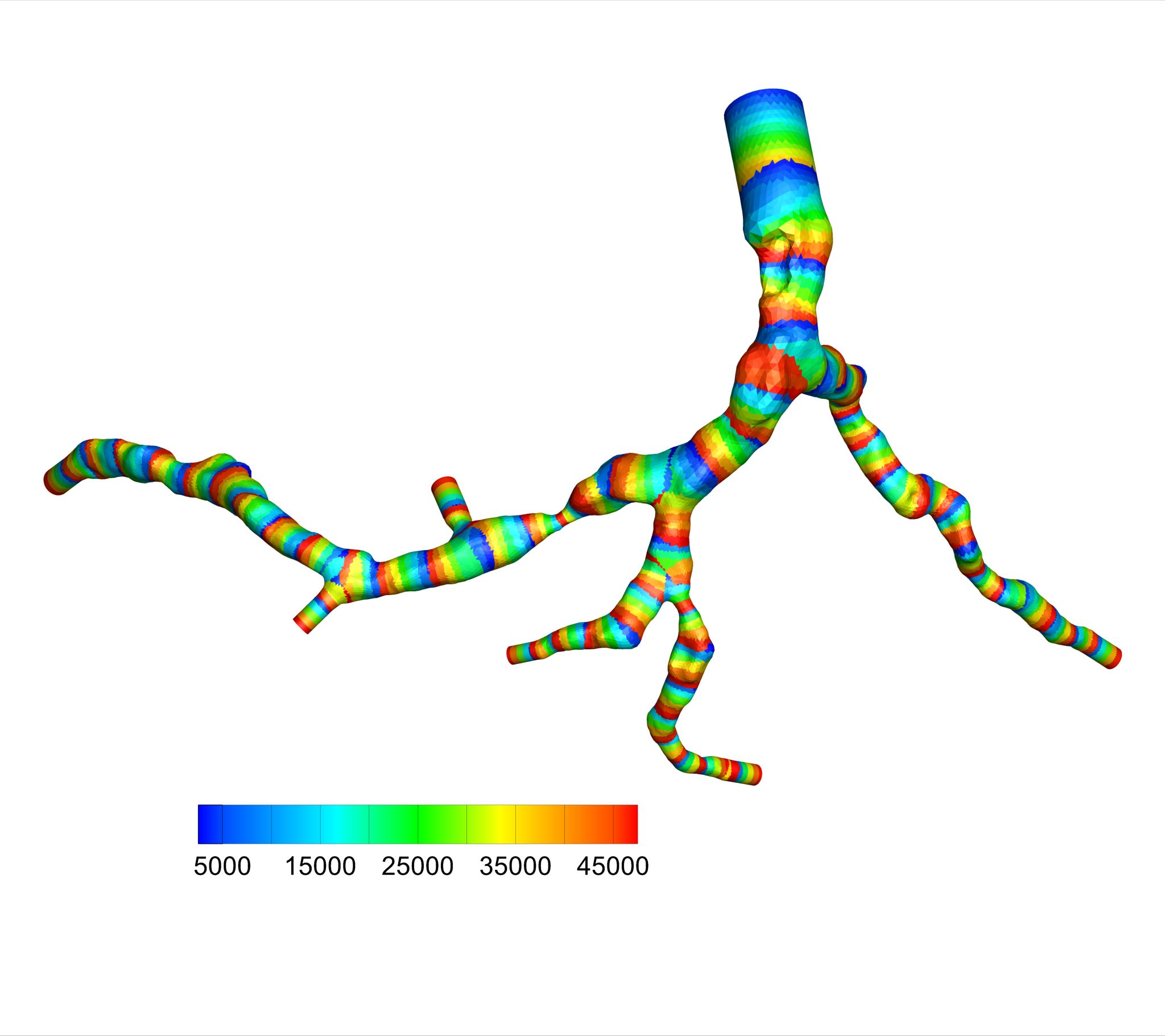}
	\caption{\textcolor{black}{Coronary tree test. Dual cell index after the reordering for a simulation carried out on $60$ CPUs.}}
	\label{fig:cortree_reorder}
\end{figure}

\section{Conclusions}\label{sec:conclusions}
A novel fully implicit hybrid FV/FE algorithm for the solution of the incompressible Navier-Stokes equations in two and three dimensions has been presented. To deal with complex geometries unstructured grids are considered and a staggered approach based on face-type dual meshes is proposed, circumventing typical issues arising for collocated grids, such as checkerboard phenomena. The splitting of the equations allows the decoupling of the pressure and momentum subsystems. Consequently, the convective-diffusion subsystem is solved at the aid of Crouzeix-Raviart elements for the viscous terms combined with an implicit finite volume discretization for the nonlinear convective terms, which, being unconditionally stable, allow larger time steps that the  explicit approach employed in the former algorithms proposed within the family of hybrid FV/FE methods. This feature is carefully assessed in the numerical results section, where, for instance, for the classical lid driven cavity test a CFL of 100 has been employed. To solve the corresponding nonlinear system an inexact Newton method has been used combined with an SGS-preconditioned BiCGStab or GMRES algorithm. As a consequence, linearised numerical flux functions have been introduced. From the simulations run in this paper it is clear that the use of a preconditioner is a key point to improve the convergence of the overall algorithm so the  symmetric-Gauss-Seidel (SGS) method is employed. To gain in computational efficiency also a simple but effective strategy for the reordering of the dual elements has been proposed. It is worthwhile to highlight that two different numerical flux functions have been considered: the implicit Rusanov flux, as well as the semi-implicit Ducros flux, which has been proven to be kinetic energy stable. 
Regarding the projection stage classical $\mathbb{P}^1$ FE are employed and the result obtained is used to correct the intermediate momentum obtained as the solution of the transport-diffusion subsystem. The promising results obtained for a set of classical fluid dynamic benchmarks assess the behaviour of the proposed methodology. Furthermore a real test case for the simulation of the flow in a coronary tree has been studied.

In future, we plan to apply this methodology to more realistic test cases in the field of blood flow studies as well as its coupling with a one-dimensional method to simulate blood flow in arteries and veins \cite{boileau2015,MuellerWB,CDT12,xiao2014,mynard2015,ioriatti2017,Mueller2014}. This may allow the simulation of complex networks of vessels with a special focus on analysing in detail the behaviour of the flow on particularly problematic areas. Further developments may include the implementation of more advanced reordering techniques and the extension of the implicit algorithm to deal with compressible or non-Newtonian flows.

\section*{Acknowledgements}
The Authors acknowledge the financial support of the Italian Ministry of Education, University and Research (MIUR) via the Departments of Excellence Initiative 2023--2027 attributed to DICAM of the University of Trento (grant L. 232/2016) and in the framework of the PRIN 2017 project \textit{Innovative numerical methods for evolutionary partial differential equations and applications} and of the Spanish Ministry of Science and Innovation, grant number PID2021-122625OB-I00.  AL acknowledges funding from the University of Trento (UNITN) for the PhD grant. SB and MD are members of the GNCS group of INdAM.

The authors would like to acknowledge the Department of Structural Engineering of the Norwegian University of Science and Technology (NTNU)  for providing the coronary tree geometry and data.
The authors thankfully acknowledges the computer resources at Finisterrae III and the technical support provided by CESGA, Spain, (RES-IM-2022-3-0017). 

\bibliographystyle{plain}
\bibliography{./mibiblio}

\end{document}